\documentclass[1 [leqno,11pt]{amsart}
\usepackage{amssymb, amsmath,latexsym,amsfonts,amsbsy, amsthm,mathtools,graphicx,CJKutf8,CJKnumb,CJKulem,color}
\usepackage{float}
\usepackage{hyperref}

\makeatletter
\newif\ifmsbmloaded@
\def\loadmsbm{\msbmloaded@true
  \font\tenmsb=msbm10 scaled 1\@ptsize00
  \font\sevenmsb=msbm7 scaled 1\@ptsize00
  \font\fivemsb=msbm5 scaled 1\@ptsize00
  \alloc@8\fam\chardef\sixt@@n\msbfam
  \textfont\msbfam=\tenmsb
  \scriptfont\msbfam=\sevenmsb
  \scriptscriptfont\msbfam=\fivemsb
  }
\@addtoreset{equation}{section}

\makeatother \loadmsbm
\def\R{\mathbb R}

\def\no{\noindent}

\def\f#1#2{\frac{#1}{#2}}
\def\f{\frac}

\def\ov{\overline}

\def\pa{\partial}

\def\be{\beta}
\def\ve{\varepsilon}
\def\na{\nabla}

\def\al{\alpha}

\def\T{{\mathbb T}}

\def\cL{{\mathcal L}}
\def\bu{\overline u}
\def\tu{\widetilde u}
\def\td{\widetilde}

\newcommand{\beq}{\begin{equation}}
\newcommand{\eeq}{\end{equation}}
\newcommand{\ben}{\begin{eqnarray}}
\newcommand{\een}{\end{eqnarray}}
\newcommand{\beno}{\begin{eqnarray*}}
\newcommand{\eeno}{\end{eqnarray*}}
\allowdisplaybreaks

\newtheorem{Theorem}{Theorem}[section]
\newtheorem{Lemma}[Theorem]{Lemma}

\newtheorem{Proposition}{Proposition}[section]

\begin{document}

\title[Transition threshold for the 3D Couette flow]
{Transition threshold for the  3D Couette flow in Sobolev space}

\author{Dongyi Wei}
\address{School of Mathematical Science, Peking University, 100871, Beijing, P. R. China}
\email{jnwdyi@163.com}

\author{Zhifei Zhang}
\address{School of Mathematical Science, Peking University, 100871, Beijing, P. R. China}
\email{zfzhang@math.pku.edu.cn}

\date{\today}

\maketitle

\begin{abstract}
In this paper, we study the transition threshold of the 3D Couette flow
in Sobolev space at high Reynolds number $\text{Re}$.  It was proved that if the initial velocity $v_0$ satisfies $\|v_0-(y,0,0)\|_{H^2}\le c_0\text{Re}^{-1}$, then the solution of the 3D Navier-Stokes equations is global in time and does not transition away from the Couette flow. This result confirms the transition threshold conjecture in physical literatures.
\end{abstract}

\section{Introduction}
In this paper, we consider the 3D incompressible Navier-Stokes
equations at high Reynolds number $\text{Re}$ regime:
\begin{align}\label{eq:NS}
\left\{
\begin{aligned}
&\partial_t v-\nu\Delta v+v\cdot\nabla v+\nabla p=0,\\
&\nabla\cdot v=0,\\
&v(0,x,y,z)=v_0(x,y,z),
\end{aligned}
\right.
\end{align}
where $v=\big(v^1(t,x,y,z),v^2(t,x,y,z),v^3(t,x,y,z)\big)$ is the velocity, $p(t,x,y,z)$ is the pressure, and $\nu=\text{Re}^{-1}>0$ is the viscosity coefficient. To avoid the boundary effect, we consider the problem in a simple domain $\Omega=\T\times \R\times \T\ni (x,y,z)$.

Beginning with Reynolds's famous paper \cite{Rey} in 1883, the stability and transition to turbulence of the laminar flows at high Reynolds number(i.e., $\nu\to 0$) has been an important and active field in the fluid mechanics \cite{Sch, Yag}. In this paper, we are concerned with the stability and transition of the Couette flow $U(y)=(y,0,0)$, which may be the simplest steady solution of \eqref{eq:NS}. Therefore, we introduce the perturbation $u(t,x,y,z)=v(t,x,y,z)-U(y)$, which satisfies
\begin{align}\label{eq:NS-p}
\left\{
\begin{aligned}
&\partial_t u-\nu\Delta u+y\partial_x u+\left(\begin{array}{l}u^2\\0\\0\end{array}\right)+\nabla p^{L}+u\cdot\nabla u+\nabla p^{NL}=0,\\
&\nabla\cdot u=0,\\
&u(0,x,y,z)=u_0(x,y,z),
\end{aligned}
\right.
\end{align}
where  the pressure $p^{L}$ and $p^{NL}$ are determined by
\ben
&&\Delta p^{L}=-2\partial_xu^2,\label{eq:p-L}\\
&&\Delta p^{NL}=-\text{div}(u\cdot\nabla u)=-\partial_iu^j\partial_ju^i.\label{eq:p-NL}
\een

It is well-known that the Couette flow $U$ is linearly stable for any $\nu\ge 0$ \cite{Rom}. However, it could be unstable and transition to turbulence for small perturbations at high Reynolds number \cite{Cha, GG, OK, Sch, Yag}, which is referred to as subcritical transition. Up to now, we are still lacking a good understanding of this transition even for some simple flows
such as Couette flow and Poiseuille flow.

To understand the transition, the traditional method is the linear stability analysis. The non-normality of the linearized operator may give rise to the transient growth of the solution even for the stable flow \cite{TTR, Tre}.  Indeed, the linear stability analysis of Couette flow predicts a linear growth of the solution for $t\lesssim \f 1 \nu$ due to the 3-D lift-up effect.  This turns out to be a primary source leading to the transition to turbulence for small perturbations.  Our goal of this paper is to find the largest perturbations (threshold) below which the solution does not transition away from the Couette flow. More precisely, as suggested by Kelvin \cite{Kel}, we study the following classical question: \smallskip

{\it
Given a norm $\|\cdot\|_X$, find a $\beta=\beta(X)$ so that
\beno
&&\|V_0\|_X\le \nu^\beta\Longrightarrow  {stability},\\
&&\|V_0\|_X\gg \nu^\beta\Longrightarrow  {instability}.
\eeno
}
The exponent $\beta$ is referred to as the transition threshold in the applied literatures. \smallskip

There are a lot of works \cite{DBL, LK, LHR, RSB, Yag}  in applied mathematics and physics devoted to estimating $\beta$. Recently, Bedrossian, Germain, Masmoudi  et  al.  made an important progress on the stability threshold problem for the 3-D Couette flow in $\mathbb{T}\times \R\times \T$ in a series of works \cite{BGM1, BGM2, BGM3, BMV, BWV}. Roughly speaking, their results could be summarized as follows:

\begin{itemize}

\item if the perturbation is in Gevrey class, then $\beta\le 1$ \cite{BGM1};

\item if the perturbation is in Sobolev space, then $\beta\le \f32$ \cite{BGM3}.

\end{itemize}
While in $\mathbb{T}\times \R$, the transition threshold is smaller:
\begin{itemize}

\item if the perturbation is in Gevrey class, then $\beta=0$ \cite{BMV};

\item if the perturbation is in Sobolev space, then $\beta\le \f12$\cite{BWV}.

\end{itemize}

More precisely, the authors in \cite{BGM3} showed that if the initial perturbation $u_0$ satisfies $\|u_0\|_{H^\sigma}\le \delta \nu^{\f32}$ for $\sigma>\f92$, then the solution is global in time, remains within $O(\nu^\f12)$ of the Couette flow in $L^2$ for any time, and converges to the streak solution for $t\gg \nu^{-\f13}$.  Compared with the result in \cite{BGM1}, it seems to mean that  the regularity of the initial data has an important effect on the transition threshold. Moreover, the result in \cite{BGM1} is consistent with the threshold conjecture in some physical literatures \cite{BT, TTR, LHR, Cha, Wal, DBL}. However, the physical literatures do not carefully consider the possible effect of the regularity. Thus, it remains open whether the transition threshold $\beta\le 1$ holds in Sobolev regularity. See the review paper \cite{BGM} for more introductions and open questions.

\smallskip

In this paper, we confirm the transition threshold conjecture in physical
literatures in Sobolev regularity. To state our result, we define
\beno
P_0f=\overline{f}=\int_{\T}f(x,y,z)dx,\quad P_{\neq}f=f_{\neq}=f-P_0f.
\eeno

Our stability result is stated as follows.

\begin{Theorem}\label{thm:main}

There exists a constant $c_0>0$ such that if $\nu\in(0,1)$ and $u_{0}$ is divergence free with $\|u_{0}\|_{H^2}\leq c_0\nu$, then the solution $u$ to the system \eqref{eq:NS-p} is global in time and satisfies the following uniform estimates: for any $t\ge 1$
\begin{align}
&\nu\|\bu(t)\|_{H^4}+\|\partial_t\bu(t)\|_{H^2}+\nu e^{2\nu^{1/3}t}\|\partial_xu(t)\|_{H^3}\leq C\|u_{0}\|_{H^2},\label{eq:u-uniform1}\\
&\|\bu^2(t)\|_{H^2}+\|\bu^3(t)\|_{H^1}+e^{2\nu^{1/3}t}\big(\|u_{\neq}^2(t)\|_{H^2}+\|(\partial_x^2+\partial_z^2)u_{\neq}^3(t)\|_{L^2}\big)\leq C\|u_{0}\|_{H^2}.\label{eq:u-uniform2}
\end{align}
Here $C$ is a constant independent of $\nu, c_0$ and $t$.
\end{Theorem}

Let us give some remarks on our result.

\begin{itemize}

\item[1.] Thanks to \eqref{eq:u-uniform1} and $\|u_{0}\|_{H^2}\leq c_0\nu$, it holds that for any $t\ge 1$
\beno
\|u(t)\|_{H^3}+e^{2\nu^{1/3}t}\|u_{\neq}(t)\|_{H^3}\le C\nu^{-1}\|u_0\|_{H^2}\le Cc_0.
\eeno
This means that the solution will remain within $Cc_0$ of the Couette flow in $H^3$ for any $t\ge 1$,
and the dynamics of the solution could be described by the streak solution for $t\gg \nu^{-\f13}$(see section 2.2).
Furthermore, the estimate $\|\pa_t\bu\|_{H^2}\le Cc_0\nu$ is very crucial to handle the linearized equation
with variable coefficients.

\item[2.] In this paper, we consider the special domain $\T\times \R\times \T$ in order to avoid the boundary effect.
Thus, it remains open whether the transition threshold conjecture holds for the system \eqref{eq:NS} in more physical domain $\T\times (a,b)\times \T$
with non-slip boundary condition on $y=a,b$.

\item[3.] In a joint work \cite{LWZ} with Li, we proved that the transition threshold $\beta\le \f 74$ for the 3D Kolmogorov flow.
This result should be not optimal. We conjecture that the transition threshold should be $\beta\le 1$ for monotone flows, and
$\beta\le \f32$ for non-monotone flows such as Kolmogorov flow and Poiseuille flow.

\end{itemize}

To prove Theorem \ref{thm:main}, we will use two important stabilizing effects: the enhanced dissipation and inviscid damping due to the mixing induced by the Couette flow. We will unify two kinds of effects into various space-time estimates of the solution for the linearized equations.
We refer to \cite{BW, CKR, Ga, IMM, WZZ3} and \cite{BM, WZZ1, WZZ2, Z1, Z2, BCV} for related works. The main instability mechanism is the 3D lift-up effect, which leads to a linear growth of the solution for $t\lesssim \f 1 \nu$. For this, we need to
study carefully nonlinear interactions between different modes of the solution, especially zero mode(streak solution) and non-zero modes.
Similar to null forms for quasilinear wave equations introduced in \cite{Kla}, some good(null) structures may avoid bad nonlinear interactions
such as the interaction between $\bu^1$ and itself.

\medskip

\no{\bf Notations.}   We use $ \partial_1, \pa_2, \pa_3$ to denote the derivative with respect to $x,y,z$ respectively. Summation notation is assumed: the repeated upper and lower indices are summed over $i,j\in\{1,2,3\}$ and $\al,\be\in\{2,3\}.$ When using other indices such as $k$, summation is not assumed. The Fourier transform $\widehat{f}(k,\eta,l)$ or $\mathcal{F}f$ of a function $f(x,y,z)$ denoted  is defined by
\begin{align*}
\widehat{f}(k,\eta,l)=\int_{x\in\T}\int_{y\in\R}\int_{z\in\T}{f}(x,y,z)e^{-2\pi i(kx+\eta y+lz)}dxdydz.
\end{align*}
Then $f(x,y,z)=\sum_{k,l\in\mathbb{Z}}\int_{\eta\in\R}\widehat{f}(k,\eta,l)e^{2\pi i\eta y}d\eta e^{2\pi i(kx+lz)}.$
The Fourier multiplier $ \langle D\rangle^N$ is defined by
\begin{align*}
&\langle D\rangle^Nf(x,y,z)=\mathcal{F}^{-1}\big(1+(2\pi)^2(k^2+\eta^2+l^2)\big)^{\frac{N}{2}}\mathcal{F}f.
\end{align*}
Then $\langle D\rangle^2=I-\Delta$ and the norm of Sobolev space $H^N(\Omega)$ is given by
\begin{align*}
\|f\|_{H^N}=\|\langle D\rangle^Nf\|_{L^2}.
\end{align*}
The inverse operator $\Delta^{-1}$ is defined by $\widehat{\Delta^{-1}f_{\neq}}=\dfrac{-\widehat{f}}{(2\pi)^2(k^2+\eta^2+l^2)}$ for $k\neq 0$ and $\widehat{\Delta^{-1}f_{\neq}}=0$ for $k= 0.$ We use $ \langle,\rangle$ to denote the $L^2$ inner product.\medskip

Throughout this paper, we denote by $C$ a constant independent of $\nu,T$.

\section{Key ingredients of the proof}

\subsection{Linear effects}

There are four kinds of linear effects: lift-up, inviscid damping, enhanced
dissipation and vortex stretching, which play a crucial role in the stability analysis.

The linearized system of \eqref{eq:NS-p}  reads
\begin{align*}
\partial_t u-\nu\Delta u+y\partial_x u+\left(\begin{array}{l}u^2\\0\\0\end{array}\right)-\nabla\Delta^{-1} 2\partial_xu^2=0.
\end{align*}
Introduce new variables $(\overline{x},y,z)=(x-ty,y,z)$ and set $\tu(t,\overline{x},y,z)={u}(t,{x},y,z)$, which solves
\begin{align*}
\partial_t \tu-\nu\Delta_{L} \tu+\left(\begin{array}{l}\tu^2\\0\\0
\end{array}\right)-\nabla_{L}\Delta_{L}^{-1} 2\partial_{\overline{x}}\tu^2=0,
\end{align*}
where $\nabla_{L}=(\partial_{\overline{x}},\partial_y-t\partial_{\overline{x}},\partial_z)$ and $\Delta_{L}=\nabla_{L}\cdot\nabla_{L}. $

Notice that $P_0\tu=\bu$, and hence it reads
\begin{align*}
\partial_t\bu-\nu\Delta {\bu}+\left(\begin{array}{l}{\bu}^2\\0\\0\end{array}\right)=0.
\end{align*}
The solution of this linear problem is given by
\begin{align*}
\bu(t)=\left(\begin{array}{c}e^{\nu t\Delta}({\bu}^1-t{\bu}^2)(0)\\e^{\nu t\Delta}{\bu}^2(0)\\e^{\nu t\Delta}{\bu}^3(0)\end{array}\right).
\end{align*}
The linear growth predicted by this solution for times $t\lesssim1/\nu
$ is known as the lift-up effect first observed in \cite{EP}. This is a crucial  mechanism leading to the instability in 3D case.
\smallskip

Turning now to nonzero frequencies in $\overline{x} $, $\widetilde{q}_{\neq}^2=\Delta_{L}{\tu}_{\neq}^2$ reads
\begin{align}\label{eq:q2}
\partial_t{\widetilde q}_{\neq}^2-\nu\Delta_{L}{\widetilde q}_{\neq}^2=0.
\end{align}
Taking the Fourier transform (denoting by $k, \eta, l$ the dual variables of $\overline{x}, y, z$ respectively), the equation \eqref{eq:q2} can be recast as
\begin{align*}
\partial_t \widehat{\widetilde{q}_{\neq}^2}-\nu(2\pi)^2(k^2+(\eta-kt)^2+l^2) \widehat{\widetilde{q}_{\neq}^2}=0.
\end{align*}
Thus, we have
\beno
\widehat{\widetilde{q}_{\neq}^2}(t,k,\eta,l)=e^{-\nu(2\pi)^2\int_0^t(k^2+(\eta-k\tau)^2+l^2)d\tau}
\widehat{\widetilde{q}_{\neq}^2}(0,k,\eta,l).
\eeno
 Due to $\int_0^t(k^2+(\eta-k\tau)^2+l^2)d\tau\geq k^2t^3/12 $,  we deduce that
 \begin{align*}
\|{\widetilde{q}_{\neq}^2}(t)\|_{H^{s}}\leq e^{-c\nu t^3}\|{\widetilde{q}_{\neq}^2}(0)\|_{H^s}\leq e^{Ca^{3/2}}e^{-a\nu^{1/3}t}\|{\widetilde{q}_{\neq}^2}(0)\|_{H^s}.
\end{align*}
Here we only take $a\in [0,4]$ for the sake of definiteness.
The exponent $\nu t^3$ gives a dissipation time scale $\nu^{-1/3}$, which is much shorter than the dissipation time scale $\nu^{-1}$. We refer to this phenomenon as the enhanced dissipation. In this paper, we will use various space-time estimates of the type:
\begin{align}\label{eq:q2-L2decay}
\|e^{a\nu^{1/3}t}{\widetilde{q}_{\neq}^2}\|_{L^2L^2}\leq C\nu^{-\frac{1}{6}}\|{\widetilde{q}_{\neq}^2}(0)\|_{L^2}.
\end{align}

The velocity field can be recovered by the formula $\widetilde{u}_{\neq}^2=\Delta_{L}^{-1}\widetilde{q}_{\neq}^2$:
\begin{align*}
 \widehat{{\tu}_{\neq}^2}=-\frac{1}{(2\pi)^2(k^2+(\eta-kt)^2+l^2)} \widehat{\widetilde{q}_{\neq}^2}.
\end{align*}
Due to the bound $\int_0^t\frac{1}{k^2+(\eta-k\tau)^2+l^2}d\tau\leq \frac{C}{k^2},$ it is expected from \eqref{eq:q2-L2decay}  that
\begin{align*}
 \|e^{a\nu^{1/3}t}\partial_{\overline{x}}\nabla_{L}{\td{u}_{\neq}^2}\|_{L^2L^2}\leq C\|{\td{q}_{\neq}^2}(0)\|_{L^2}.
\end{align*}
This effect was discovered by Orr in 1907 \cite{Orr}, and is now known as inviscid damping.\smallskip

In section 3, we will establish various space-time estimates for the linearized equations including the variable coefficient version,
which are based on two linear effects: enhanced dissipation and inviscid damping.

\subsection{Streak solutions}

If the initial data in \eqref{eq:NS-p} is independent of $x$, then so does the solution, i.e., $ u(t,x,y,z)=u(t,y,z)$. In such case, $(u^{2}, u^3)$ solves the 2D Navier-Stokes equations in $(y,z)\in\R\times\T$, whereas $u^1$ solves the linear advection-diffusion equation
\begin{align*}
\partial_t {u}^1-\nu\Delta {u}^1+(u^2\partial_y+u^3\partial_z){u}^1+u^2=0.
\end{align*}
These solutions are refereed to as the streak. Due to the lift-up effect,
main nonlinear effect comes from the interaction between the streak solution and nonzero modes. \smallskip

Our result shows that for $t\gg \nu^{-\f13}$, the streak solutions describe the dynamics of the system if the perturbation is below the threshold.

\subsection{Nonlinear interaction}

There are several nonlinear mechanisms which may lead to the instability.
So, we have to study nonlinear interactions very carefully and use null structures hidden in the system to avoid bad interactions.
To this end, we decompose the solution $u$ into $\bu+u_{\neq}$, where $\bu$ is zero mode and $u_{\neq}$ is non-zero mode. Then nonlinear interactions can be classified as follows:

\begin{itemize}

\item zero mode and zero mode interaction: $0\cdot0 \to0$;

\item zero mode and nonzero mode interaction: $0\cdot\neq \to\neq$;

\item nonzero mode and nonzero mode interaction: $\neq\cdot\neq \to\neq$ or $ \neq\cdot\neq \to0.$
\end{itemize}

As the nonlinear term takes the form $u^i\pa_iu^j$, there is no interaction
between $\bu^1$ and itself. The worst interaction between zero mode and nonzero mode is $\bu^1\pa_xu_{\neq}$ due to the lift-up effect. This seems a primary source so that the solution could transition to turbulence if  the perturbation exceeds some threshold.\smallskip

In section 5, we will study nonlinear interactions between different modes based on the bilinear anisotropic Sobolev estimates
established in section 4.

\subsection{New formulation}

According to the analysis above, we decompose the nonlinear pressure $p^{NL}$ as
\ben\label{eq:p-decom}
p^{NL}=p^{(1)}+p^{(2)}+p^{(3)}+p^{(4)},
\een
where
\begin{align*}
&\Delta p^{(1)}=-2\big(\partial_y\bu^1\partial_xu_{\neq}^2+\partial_z\bu^1\partial_xu_{\neq}^3\big),\quad
\Delta p^{(2)}=-\partial_i\bu^j\partial_j\bu^i,\\
&\Delta p^{(3)}=-2\partial_{\al}\bu^{\beta}\partial_{\beta}u_{\neq}^{\al},\quad\Delta p^{(4)}=-\partial_iu_{\neq}^j\partial_ju_{\neq}^i.
\end{align*}
The main challenge is the $0\cdot \neq\to \neq$ interaction, especially
the first part $p^{(1)}$.  So, we introduce
\beno
p^{L1}=p^L+p^{(1)},\quad V(t,y,z)=y+\bu^1(t,y,z).
\eeno
Then we have
\ben
\Delta p^{L1}=-2\big(\partial_yV\partial_xu_{\neq}^2+\partial_zV\partial_xu_{\neq}^3\big).\label{def:PL1}
\een
It holds that for $j=2,3$,
\begin{align*}
\partial_t {u}^j-\nu\Delta {u}^j+V\partial_x u^j+\partial_j p^{L1}+g_j=0,
\end{align*}
where
\beno
g_j=(\bu^2\partial_y+\bu^3\partial_z) u^j+u_{\neq}\cdot\nabla u^j+\partial_j\big(p^{(2)}+p^{(3)}+p^{(4)}\big).
\eeno
We introduce the linearized operators
\beno
&&\cL_0=\partial_t -\nu\Delta +y\partial_x,\quad \cL=\partial_t -\nu\Delta +V\partial_x.
\eeno
Then for $j=2,3$, we can write
\ben
\cL u_{\neq}^j+\partial_j p^{L1}+(g_j)_{\neq}=0.\label{eq:u23-non}
\een
Under the assumption that
\beno
\|\bu^1\|_{H^4}+\|\partial_t\bu^1\|_{H^2}/\nu<c
\eeno
for some small constant $c$, the operator $\cL$ could be viewed as a perturbation of $\cL_0$.\smallskip

Let $\kappa(t,y,z)=\f {\partial_zV} {\partial_yV}=\f {\partial_z\bu^1} {(1+\partial_y\bu^1)}$ and $W^2=u_{\neq}^2+\kappa u_{\neq}^3$.
Then we find from \eqref{def:PL1} that
\beno
\Delta p^{L1}=-2\partial_yV\partial_xW^2.
\eeno
Thus, it is natural to derive the equation of $W^2$, which satisfies
\ben
\cL_1W^2+G_2=(\partial_t \kappa-\nu\Delta \kappa)u_{\neq}^3-2\nu\nabla\kappa\cdot\nabla u_{\neq}^3,\label{eq:W2}
\een
where $G_2=(g_{2}+\kappa g_3)_{\neq}$ and
\beno
\cL_1 f=\cL f-2(\partial_y+\kappa\partial_z)\Delta^{-1}(\partial_yV\partial_xf).\label{def:L1}
\eeno
For the linearized operator $\cL_1$, an important property is that
\beno
\Delta\cL_1f=\cL\Delta f+\text{good terms}.
\eeno
In the equation \eqref{eq:W2}, the main trouble term is $-2\nu\nabla\kappa\cdot\nabla u_{\neq}^3. $ To handle it, we introduce a good derivative $(\partial_z-\kappa\partial_y),$ which has a good communication relation with $\cL$.  Let
\ben\label{def:rho12}
 \rho_1=\dfrac{\partial_y\kappa+\kappa\partial_z\kappa}{\partial_y V(1+\kappa^2)},\quad \rho_2=\dfrac{\partial_z\kappa-\kappa\partial_y\kappa}{(1+\kappa^2)}.
 \een
Then it holds that
\ben
 \nabla\kappa\cdot\nabla u_{\neq}^3=\rho_1\nabla V\cdot\nabla u_{\neq}^3+\rho_2(\partial_z-\kappa\partial_y)u_{\neq}^3.
\een
Now the second term is good. For the first term, we need to use more subtle structure, which will be uncovered by introducing a good decomposition. See section 7.1 for more details.

 \subsection{Energy functional}

First of all, the following local well-posedness result is standard.

\begin{Proposition}\label{prop:lwp}
There exist two positive constants ${c_{00}}, C_0$ independent of $\nu$ so that if  $\|u_{0}\|_{H^2}\le {c_{00}}\nu$, then the system \eqref{eq:NS-p} has a unique solution $u\in C\big([0,2];H^2\big)$ which satisfies
\beno
&&\|u(t)\|_{H^2}\leq C_0\|u_{0}\|_{H^2}\quad \text{for}\quad t\in[0,2],\\
&&\nu\|u(t)\|_{H^4}+\|\partial_tu(t)\|_{H^2}\leq C_0\|u_{0}\|_{H^2}\quad \text{for}\quad t\in[1,2].
\eeno
\end{Proposition}

\no Therefore, we only need to establish the uniform estimates of the solution
in the interval $[1,T]$ for any $T>1$. From now on, we always assume that  $\nu\in(0,1)$ and $T>1,$ and all norms are taken over the interval [1, T] unless stated otherwise,  such as
\begin{align*}
&\|f\|_{L^pH^s}=\big\|\|f(t)\|_{H^s(\Omega)}\big\|_{L^p(1,T)},\quad \|f\|_{L^pL^q}=\big\|\|f(t)\|_{L^q(\Omega)}\big\|_{L^p(1,T)}.
\end{align*}

For $a\ge 0$, we introduce two norms
\begin{align*}
\|f\|_{Y_0}^2=&\|f\|_{L^{\infty}L^2}^2+\nu\|\nabla f\|_{L^{2}L^2}^2,\\ \|f\|_{X_a}^2=&\|e^{a\nu^{1/3}t}f\|_{L^{\infty}L^2}^2+\|e^{a\nu^{1/3}t}\nabla\Delta^{-1}\partial_x f\|_{L^{2}L^2}^2\\&+\nu^{1/3}\|e^{a\nu^{1/3}t} f\|_{L^{2}L^2}^2+\nu\|e^{a\nu^{1/3}t}\nabla f\|_{L^{2}L^2}^2.
\end{align*}
Now we introduce the energy functional
\begin{align*}
&E_1=\|\bu\|_{L^{\infty}H^4}+
\nu^{\frac{1}{2}}\|\nabla \bu\|_{L^{2}H^4}+\big(\|\partial_t\bu\|_{L^{\infty}H^2}+\|\bu(1)\|_{H^2}\big)/\nu,\\
&E_2=\|\Delta \bu^2\|_{Y_0}+\|\bu^3\|_{Y_0}+\|\nabla \bu^3\|_{Y_0}+\|\min(\nu^{\frac{2}{3}}+\nu t,1)^{\frac{1}{2}}\Delta \bu^3\|_{Y_0},\\
&E_3=\|\Delta u_{\neq}^2\|_{X_2}+\| (\partial_x^2+\partial_z^2)u_{\neq}^3\|_{X_2}+\nu^{\frac{2}{3}}\|\Delta u_{\neq}^3\|_{X_3},\\
&E_4= \|e^{2\nu^{1/3}t}(\partial_x,\partial_z)u_{\neq}\|_{L^{\infty}H^3}+
\nu^{\frac{1}{2}}\|e^{2\nu^{1/3}t}\nabla(\partial_x,\partial_z)u_{\neq}\|_{L^{2}H^3},\\
&E_5=\|\partial_x^2 u^2\|_{X_3}+\| \partial_x^2u^3\|_{X_3}.
\end{align*}
In fact, the norms $X_2$ and $X_3$ can be replaced by any $X_{c}$ and $X_{c'}$ with $c<c'<2c.$
Each norm defined in $X_a$ and $Y_0$ has the same scaling in some sense. \smallskip

The energy $E_1, E_2$ correspond to zero mode of the solution, while the energy $E_3, E_4, E_5$ correspond to non-zero mode of
the solution. Due to the lift-up effect,  $E_1, E_4$ are expected to be small at best, and $E_2, E_3, E_5$ are expected to be bounded by $o(\nu)$.
Based on the the evolutional equation of $\bu$ or its vorticity formulation, we will use the energy method to estimate $E_1$ and $E_2$.
The estimate of $E_3$ is based on the space-time estimate for the linearized operator $\cL_0$(Proposition \ref{prop:decay-L0-2}) and the formulation:
\begin{align*}
\cL_0 u_{\neq}^j+\partial_j p^{L}+\partial_j p^{(1)}+\bu^1\partial_xu_{\neq}^j+(g_j)_{\neq}=0,\quad \Delta p^{L}=-2\partial_x{u}^2,\quad j=2,3.
\end{align*}
The estimate of $E_4$ is based on the space-time estimate for $\cL_0$(Proposition \ref{prop:decay-L0}) and the formulation:
\begin{align*}
&\cL_0 u+\mathbb{P}\left(\begin{array}{l}u^2\\0\\0\end{array}\right)+\mathbb{P}(u\cdot\nabla u)=0.
\end{align*}
Here $ \mathbb{P}$ is the Helmholtz-Leray projection.
The estimate of $E_5$ is the most difficult. For this, we need to use the space-time estimates for the linearized operators $\cL, \cL_1$ and the subtle structure
hidden in the system.\smallskip

The estimates of $E_1, E_2$ and $E_3, E_4, E_5$ will be conducted in section 6 and section 7 respectively.

 \subsection{Proof of Theorem \ref{thm:main}}

The proof of Theorem \ref{thm:main} is based on a bootstrap argument.
Firstly, we assume that
\ben\label{ass:boot}
E_1\le \ve_0, \quad E_2\le \ve_0\nu,\quad  E_3\le \ve_0\nu,
\een
where $\ve_0$ is determined later. These are true for some $T>1$ by Proposition \ref{prop:lwp}.
In section 6 and section 7, we will show that  there exists a constant $C$ independent of $\nu, \ve_0, T$ so that
\begin{align*}
&E_1\le C\big(\|\bu(1)\|_{H^4}+\nu^{-1}\|\bu(1)\|_{H^2}+\nu^{-1}E_2+E_3+E_4\big),\\
&E_2\le C\big(\|u(1)\|_{H^2}+\nu^{-1}E_3^2\big),\\
&E_3\le C\|u(1)\|_{H^2},\\
&E_4\le C\big(\|u(1)\|_{H^4}+\nu^{-1}E_3+\nu^{-1}E_5\big),\\
&E_5\le C\big(\|u(1)\|_{H^2}+\nu^{-1}E_3^2\big).
\end{align*}
Then it follows from Proposition \ref{prop:lwp}  that
\beno
E_3\le Cc_0\nu, \quad E_2\le Cc_0\nu+C{c_0^2}\nu\le Cc_0\nu,\quad
E_5\le Cc_0\nu.
\eeno
and hence,
\beno
E_4\le Cc_0,\quad E_1\le Cc_0.
\eeno
Taking $\ve_0=2Cc_0$, we can conclude that $T=+\infty$, and the uniform estimates \eqref{eq:u-uniform1} and \eqref{eq:u-uniform2}  follow easily from the definitions of $E_1$-$E_5$.

\section{Enhanced decay estimates of the linearized equations}

In this section, we derive the enhanced decay estimates of the linearized equations, which will make use of two linear effects of the Couette flow: inviscid damping and enhanced dissipation.

\subsection{Decay estimates of the linearized equation $\cL_0f=g$}

\begin{Proposition}\label{prop:decay-L0}
Let $f$ solve the linear equation
\beno
\cL_0 f=\partial_xf_1+f_2+\text{div}f_3
\eeno
 for $t\in [1,T]$. If $P_0 f=P_0f_1=P_0f_2=P_0f_3=0$, then for $a\in[0,4]$, we have
 \begin{align*}
&\|f\|_{X_a}^2\leq C\big(\|f(1)\|_{L^2}^2+\|e^{a\nu^{1/3}t}\nabla f_1\|_{L^2L^2}^2+\nu^{-\frac{1}{3}}\|e^{a\nu^{1/3}t} f_2\|_{L^2L^2}^2+\nu^{-1}\|e^{a\nu^{1/3}t} f_3\|_{L^2L^2}^2\big).
\end{align*}
\end{Proposition}

Switch to new variables $(\overline{x},y,z)=(x-ty,y,z)$ by setting $\td{f}(t,\overline{x},y,z)={f}(t,{x},y,z),\ \td{f}_j(t,\overline{x},y,z)={f}_j(t,{x},y,z),$ then it holds that
\begin{align}\label{L0f}
&\partial_t\td{f}-\nu\Delta_L\td{f}=\partial_{\overline{x}}\td{f}_1+\td{f}_2+\nabla_L\cdot \td{f}_3.
\end{align}
Taking the Fourier transform, we get
\begin{align}\label{L0f1}
\partial_t \widehat{\td{f}}+\nu(2\pi)^2\big(k^2+(\eta-kt)^2+l^2\big) \widehat{\td{f}}=2\pi i k \widehat{\td{f}_{1}}+\widehat{\td{f}_{2}}+2\pi i(k,\eta-kt,l)\cdot\widehat{\td{f}_{3}}.
\end{align}
Then it is obvious to study  \eqref{L0f1} first.

\begin{Lemma}\label{lem3b}
Let $f$ solve the equation
\beno
\partial_t {f}+\nu(2\pi)^2\big(k^2+(\eta-kt)^2+l^2\big){f}= 2\pi ik {f}_{1}+{f}_{2}+{f}_{3}
\eeno
for $t\in [1,T]$, $k,l\in\mathbb{Z},\ \eta\in\R$. Then for $a\in[0,4]$ and $k\neq 0$, it holds that
 \begin{align*}
&\|e^{a\nu^{1/3}t}f\|_{L^{\infty}(1,T)}^2+\|e^{a\nu^{1/3}t} 2\pi ik\gamma(t)^{-\frac{1}{2}}f\|_{L^{2}(1,T)}^2
+\nu\|e^{a\nu^{1/3}t}\gamma(t)^{\frac{1}{2}}f\|_{L^{2}(1,T)}^2\\&\quad+\nu^{\frac{1}{3}}\|e^{a\nu^{1/3}t}f\|_{L^{2}(1,T)}^2\leq C\Big(|f(1)|^2+\|e^{a\nu^{1/3}t}\gamma(t)^{\frac{1}{2}} f_1\|_{L^2(1,T)}^2|k|(k^2+l^2)^{-\frac{1}{2}}\\&\qquad+\nu^{-\frac{1}{3}}\|e^{a\nu^{1/3}t} f_2\|_{L^2(1,T)}^2+\nu^{-1}\|e^{a\nu^{1/3}t}\gamma(t)^{-\frac{1}{2}} f_3\|_{L^2(1,T)}^2\Big),\end{align*}
where $\gamma(t)= (2\pi)^2\big(k^2+(\eta-kt)^2+l^2\big).$
\end{Lemma}

\begin{proof}By the definition of $\gamma(t)$, we have $\partial_t {f}+\nu\gamma(t) {f}= 2\pi ik {f}_{1}+{f}_{2}+{f}_{3}$. Let $\gamma_1(t)=\int_1^t\gamma(s)ds,$ then the solution is given by
\begin{align*}
f(t)=e^{-\nu\gamma_1(t)}f(1)+\int_1^te^{-\nu(\gamma_1(t)-\gamma_1(s))}(2\pi ik {f}_{1}(s)+{f}_{2}(s)+{f}_{3}(s))ds=\sum_{j=0}^3F_j(t),
\end{align*}
where $F_0(t)=e^{-\nu\gamma_1(t)}f(1)$ and
\beno
&&F_1(t)=\int_1^te^{-\nu(\gamma_1(t)-\gamma_1(s))}2\pi ik {f}_{1}(s)ds,\\
&&F_j(t)=\int_1^te^{-\nu(\gamma_1(t)-\gamma_1(s))} {f}_{j}(s)ds\quad j=2,3.
\eeno

Using the fact that for $t_1>t_2$
\begin{align*}
\nu(\gamma_1(t_1)-\gamma_1(t_2))&=\nu\int_{t_2}^{t_1}\gamma(s)ds\geq \nu(2\pi)^2 \int_{t_2}^{t_1}(\eta-ks)^2ds\\
&\geq \nu(2\pi)^2k^2(t_1-t_2)^3/12\geq (a+1)\nu^{1/3}(t_1-t_2)-C,
\end{align*}
we deduce that for $t\in [1,T]$
 \begin{align*}
 |F_0(t)|&\leq e^{-\nu\gamma_1(t)}|f(1)|\leq Ce^{-(a+1)\nu^{1/3}(t-1)}|f(1)|\\
 &\leq Ce^{-(a+1)\nu^{1/3}t}|f(1)|,
\end{align*}
which shows that
\beno
&&\|e^{a\nu^{1/3}t}F_0\|_{L^{\infty}(1,T)}\leq C|f(1)|,\\
&&\|e^{a\nu^{1/3}t}F_0\|_{L^{2}(1,T)}^2
\leq C\|e^{-\nu^{1/3}t}\|_{L^{2}(1,T)}^2|f(1)|^2\leq C\nu^{-\frac{1}{3}}|f(1)|^2.
\eeno

Using the fact that
 \begin{align}\label{kL2}
&\left\|\frac{2\pi k}{\gamma(t)^{\frac{1}{2}}}\right\|_{L^2(\R)}^2=\int_{\R}\frac{ k^2dt}{k^2+(\eta-kt)^2+l^2}=\int_{\R}\frac{ |k|dt}{k^2+t^2+l^2}=\frac{|k|\pi}{(k^2+l^2)^{\frac{1}{2}}}\leq \pi,
\end{align}
we deduce that for $t\in[1,T]$
\begin{align*}
|F_1(t)|\leq& \int_1^tCe^{-(a+1)\nu^{1/3}(t-s)}|2\pi k {f}_{1}(s)|ds\\
\leq& C\|e^{(a+1)\nu^{1/3}(s-t)}\gamma(s)^{\frac{1}{2}} f_1(s)\|_{L^2(1,t)}\|2\pi k\gamma(s)^{-\frac{1}{2}} \|_{L^2(\R)}\\ \leq& C\|e^{(a+1)\nu^{1/3}(s-t)}\gamma(s)^{\frac{1}{2}} f_1(s)\|_{L^2(1,t)}|k|^{\frac{1}{2}}(k^2+l^2)^{-\frac{1}{4}}\\ \leq& C\|e^{a\nu^{1/3}(s-t)}\gamma(s)^{\frac{1}{2}} f_1(s)\|_{L^2(1,t)}|k|^{\frac{1}{2}}(k^2+l^2)^{-\frac{1}{4}},
\end{align*}
which implies
\begin{align*}
\|e^{a\nu^{1/3}t}F_1\|_{L^{\infty}(1,T)}&\leq C\sup_{t\in[1,T]}\|e^{a\nu^{1/3}s}\gamma(s)^{\frac{1}{2}} f_1(s)\|_{L^2(1,t)}|k|^{\frac{1}{2}}(k^2+l^2)^{-\frac{1}{4}}\\ &\leq C\|e^{a\nu^{1/3}t}\gamma(t)^{\frac{1}{2}}f_1\|_{L^2(1,T)}|k|^{\frac{1}{2}}(k^2+l^2)^{-\frac{1}{4}},
\end{align*}
and
\begin{align*}
\|e^{a\nu^{1/3}t}F_1\|_{L^{2}(1,T)}^2(k^2+l^2)^{\frac{1}{2}}/|k|\leq& C\int_1^T\int_1^te^{2a\nu^{1/3}s+2\nu^{1/3}(s-t)}\gamma(s)|f_1(s)|^2dsdt\\
=&C\int_1^Te^{2a\nu^{1/3}s}\gamma(s)|f_1(s)|^2\int_s^Te^{2\nu^{1/3}(s-t)}dtds\\
=&C\nu^{-\frac{1}{3}}\int_1^Te^{2a\nu^{1/3}s}\gamma(s)|f_1(s)|^2ds=C\nu^{-\frac{1}{3}}\|e^{a\nu^{1/3}t}\gamma(t)^{\frac{1}{2}} f_1\|_{L^2(1,T)}^2.
\end{align*}

For $t\in[1,T],$ we have
\begin{align*}
|F_2(t)|\leq& \int_1^tCe^{-(a+1)\nu^{1/3}(t-s)}| {f}_{2}(s)|ds\\ \leq& C\|e^{(a+1/2)\nu^{1/3}(s-t)} f_2(s)\|_{L^2(1,t)}\|e^{-\nu^{1/3}(t-s)/2} \|_{L^2(1,t)}\\ \leq& C\|e^{(a+1/2)\nu^{1/3}(s-t)} f_2(s)\|_{L^2(1,t)}\nu^{-\frac{1}{6}}\leq C\|e^{a\nu^{1/3}(s-t)} f_2(s)\|_{L^2(1,t)}\nu^{-\frac{1}{6}},
\end{align*}
which implies
\begin{align*}
&\|e^{a\nu^{1/3}t}F_2\|_{L^{\infty}(1,T)}\leq C\sup_{t\in[1,T]}\|e^{a\nu^{1/3}s} f_2(s)\|_{L^2(1,t)}\nu^{-\frac{1}{6}}\leq C\|e^{a\nu^{1/3}t}f_2\|_{L^2(1,T)}\nu^{-\frac{1}{6}},\\
&\|e^{a\nu^{1/3}t}F_2\|_{L^{2}(1,T)}^2\leq C\nu^{-\frac{2}{3}}\|e^{a\nu^{1/3}t} f_2\|_{L^2(1,T)}^2.
\end{align*}

For $t\in[1,T],$ we have
\begin{align*}
|F_3(t)|\leq& \int_1^te^{-\nu(\gamma_1(t)-\gamma_1(s))} | {f}_{3}(s)|ds\\ \leq& C\|e^{(a+1/2)\nu^{1/3}(s-t)}\gamma(s)^{-\frac{1}{2}} f_3(s)\|_{L^2(1,t)}\|e^{-\nu(\gamma_1(t)-\gamma_1(s))+(a+1/2)\nu^{1/3}(t-s)}\gamma(s)^{\frac{1}{2}} \|_{L^2(1,t)}.
\end{align*}
Thanks to $ \gamma=\gamma_1'$ and $a\in[0,4],$ we have
\begin{align*}
&\|e^{-\nu(\gamma_1(t)-\gamma_1(s))+(a+1/2)\nu^{1/3}(t-s)}\gamma(s)^{\frac{1}{2}} \|_{L^2(1,t)}^2\\=&\int_1^te^{-2\nu(\gamma_1(t)-\gamma_1(s))}e^{(2a+1)\nu^{1/3}(t-s)}\gamma(s)ds\\
&=(2\nu)^{-1}e^{-2\nu(\gamma_1(t)-\gamma_1(s))}e^{(2a+1)\nu^{1/3}(t-s)}\big|_{s=1}^t\\&\quad+(2\nu)^{-1}(2a+1)\nu^{1/3}
\int_1^te^{-2\nu(\gamma_1(t)-\gamma_1(s))}e^{(2a+1)\nu^{1/3}(t-s)}ds\\
&\leq(2\nu)^{-1}+C(2\nu)^{-1}(2a+1)\nu^{1/3}
\int_1^te^{-2(a+1)\nu^{1/3}(t-s)}e^{(2a+1)\nu^{1/3}(t-s)}ds\leq C\nu^{-1}.
\end{align*}
Then we infer that
\begin{align*}
&|F_3(t)|\leq  C\|e^{(a+1/2)\nu^{1/3}(s-t)}\gamma(s)^{-\frac{1}{2}} f_3(s)\|_{L^2(1,t)}\nu^{-\frac{1}{2}},
\end{align*}
which implies
\begin{align*}
&\|e^{a\nu^{1/3}t}F_3\|_{L^{\infty}(1,T)}\leq C\|e^{a\nu^{1/3}t}\gamma(t)^{-\frac{1}{2}} f_3\|_{L^2(1,T)}\nu^{-\frac{1}{2}},\\
&\|e^{a\nu^{1/3}t}F_3\|_{L^{2}(1,T)}^2\leq C\nu^{-\frac{4}{3}}\|e^{a\nu^{1/3}t} \gamma(t)^{-\frac{1}{2}} f_3\|_{L^2(1,T)}^2.
\end{align*}

Summing up, we conclude that
\begin{align}\label{f1e1}
&\|e^{a\nu^{1/3}t}f\|_{L^{\infty}(1,T)}^2+\nu^{\frac{1}{3}}\|e^{a\nu^{1/3}t}f\|_{L^{2}(1,T)}^2\leq C\Big(|f(1)|^2+\nu^{-\frac{1}{3}}\|e^{a\nu^{1/3}t} f_2\|_{L^2(1,T)}^2 \\ \nonumber
&\quad+\|e^{a\nu^{1/3}t}\gamma(t)^{\frac{1}{2}}f_1\|_{L^2(1,T)}^2|k|(k^2+l^2)^{-\frac{1}{2}} +\nu^{-1}\|e^{a\nu^{1/3}t}\gamma(t)^{-\frac{1}{2}} f_3\|_{L^2(1,T)}^2\Big).
\end{align}

By \eqref{kL2}, we have
\begin{align}\label{f1e2}
\|e^{a\nu^{1/3}t} 2\pi ik\gamma(t)^{-\frac{1}{2}}f\|_{L^{2}(1,T)}^2&\leq \| 2\pi k\gamma(t)^{-\frac{1}{2}}\|_{L^{2}(\R)}^2\|e^{a\nu^{1/3}t}f\|_{L^{\infty}(1,T)}^2\\
\nonumber&\leq C|k|(k^2+l^2)^{-\f12}\|e^{a\nu^{1/3}t}f\|_{L^{\infty}(1,T)}^2.
\end{align}

Using the fact that
\begin{align*}
\partial_t |{f}|^2+2\nu\gamma(t) |{f}|^2= 2\text{Re}((2\pi ik {f}_{1}+{f}_{2}+{f}_{3})\overline{f})\leq \gamma(t) |f_1|^2+\gamma(t)^{-1}|2\pi ikf|^2\\+\nu^{\frac{1}{3}}|f|^2+\nu^{-\frac{1}{3}}|f_2|^2+\nu\gamma(t) |{f}|^2+\nu^{-1}\gamma(t)^{-1} |{f}_3|^2,
\end{align*}
 where $\overline{f} $ is the complex conjugate of $f,$ we infer that\begin{align*}
&\partial_t |e^{a\nu^{1/3}t}{f}|^2+\nu\gamma(t) |e^{a\nu^{1/3}t}{f}|^2\leq e^{2a\nu^{1/3}t}\big(\gamma(t)|k|(k^2+l^2)^{-\f12} |f_1|^2+\gamma(t)^{-1}|k|^{-1}(k^2+l^2)^{\f12}|2\pi ikf|^2\\&\quad+\nu^{\frac{1}{3}}|f|^2+\nu^{-\frac{1}{3}}|f_2|^2\big)+\nu^{-1}\gamma(t)^{-1} |e^{a\nu^{1/3}t}{f}_3|^2+2a\nu^{1/3}|e^{a\nu^{1/3}t}{f}|^2,
\end{align*}
which implies
\begin{align*}
&\nu\|e^{a\nu^{1/3}t}\gamma(t)^{\frac{1}{2}}f\|_{L^{2}(1,T)}^2\leq -|e^{a\nu^{1/3}t}{f}|^2\big|_1^T+\|e^{a\nu^{1/3}t}\gamma(t)^{\frac{1}{2}} f_1\|_{L^2(1,T)}^2|k|(k^2+l^2)^{-\frac{1}{2}}\\
&\quad+|k|^{-1}(k^2+l^2)^{\f12}\|e^{a\nu^{1/3}t} 2\pi ik\gamma(t)^{-\frac{1}{2}}f\|_{L^{2}(1,T)}^2+(1+2a)\nu^{\frac{1}{3}}\|e^{a\nu^{1/3}t}f\|_{L^{2}(1,T)}^2
\\&\qquad+\nu^{-\frac{1}{3}}\|e^{a\nu^{1/3}t} f_2\|_{L^2(1,T)}^2+\nu^{-1}\|e^{a\nu^{1/3}t}\gamma(t)^{-\frac{1}{2}} f_3\|_{L^2(1,T)}^2.
\end{align*}
Thanks to $-|e^{a\nu^{1/3}t}{f}|^2\big|_1^T\leq |e^{a\nu^{1/3}}{f}(1)|^2\leq C|{f}(1)|^2$, we infer from \eqref{f1e1} and \eqref{f1e2}  that
\begin{align}\label{f1e3}
&\nu\|e^{a\nu^{1/3}t}\gamma(t)^{\frac{1}{2}}f\|_{L^{2}(1,T)}^2\leq C\Big(|f(1)|^2+\|e^{a\nu^{1/3}t}\gamma(t)^{\frac{1}{2}} f_1\|_{L^2(1,T)}^2|k|(k^2+l^2)^{-\frac{1}{2}}\\ \nonumber&+\nu^{-\frac{1}{3}}\|e^{a\nu^{1/3}t} f_2\|_{L^2(1,T)}^2+\nu^{-1}\|e^{a\nu^{1/3}t}\gamma(t)^{-\frac{1}{2}} f_3\|_{L^2(1,T)}^2\Big).
\end{align}
Now the lemma follows from \eqref{f1e1},\eqref{f1e2} and \eqref{f1e3}.
\end{proof}
\smallskip

Now we are in a position to prove Proposition \ref{prop:decay-L0}.

\begin{proof}
As $P_0 f=P_0 f_j=0(j=1,2,3)$, we have $ \widehat{\td{f}}=\widehat{\td{f}_j}=0$ for $k=0$. Since $\widehat{\td{f}}$ satisfies \eqref{L0f1}, we deduce from Lemma \ref{lem3b} that for every $k,l\in\mathbb{Z},\ \eta\in\R,\ k\neq 0,$
\begin{align*}
&\big\|e^{a\nu^{1/3}t}\widehat{\td{f}}\big\|_{L^{\infty}(1,T)}^2+\big\|e^{a\nu^{1/3}t} k(k^2+(\eta-kt)^2+l^2)^{-\frac{1}{2}}\widehat{\td{f}}\big\|_{L^{2}(1,T)}^2
+\nu^{\frac{1}{3}}\big\|e^{a\nu^{1/3}t}\widehat{\td{f}}\big\|_{L^{2}(1,T)}^2
\\&\quad+\nu\big\|e^{a\nu^{1/3}t}(2\pi)(k^2+(\eta-kt)^2+l^2)^{\frac{1}{2}}\widehat{\td{f}}\big\|_{L^{2}(1,T)}^2\\
&\leq C\Big(\big|\widehat{\td{f}}(1,k,\eta,l)\big|^2+\big\|e^{a\nu^{1/3}t}(2\pi)(k^2+(\eta-kt)^2+l^2)^{\frac{1}{2}} \widehat{\td{f}_1}\big\|_{L^2(1,T)}^2\\&\qquad+\nu^{-\frac{1}{3}}\big\|e^{a\nu^{1/3}t} \widehat{\td{f}_2}\big\|_{L^2(1,T)}^2+\nu^{-1}\big\|e^{a\nu^{1/3}t}\widehat{\td{f}_3}\big\|_{L^2(1,T)}^2\Big).
\end{align*}
Using Plancherel's formula, we find that
\begin{align*}
&\|e^{a\nu^{1/3}t}\td{f}\|_{L^{\infty}L^{2}}^2\leq\sum_{k\in\mathbb{Z}}\int_{\eta\in\R}\sum_{l\in\mathbb{Z}}
\big\|e^{a\nu^{1/3}t}\widehat{\td{f}}(\cdot,k,\eta,l)\big\|_{L^{\infty}(1,T)}^2d\eta,\\
&\|e^{a\nu^{1/3}t}\td{f}\|_{L^{2}L^{2}}^2=\sum_{k\in\mathbb{Z}}\int_{\eta\in\R}\sum_{l\in\mathbb{Z}}
\big\|e^{a\nu^{1/3}t}\widehat{\td{f}}(\cdot,k,\eta,l)\big\|_{L^{2}(1,T)}^2d\eta,
\end{align*}
and
\begin{align*}
&\|e^{a\nu^{1/3}t}\nabla_L\Delta_L^{-1}\partial_{\overline{x}}\td{f}\|_{L^{2}L^{2}}^2
\le\sum_{k\in\mathbb{Z}}\int_{\eta\in\R}\sum_{l\in\mathbb{Z}}
\big\|\frac{e^{a\nu^{1/3}t}k}{(k^2+(\eta-kt)^2+l^2)^{\frac{1}{2}}}
\widehat{\td{f}}(\cdot,k,\eta,l)\big\|_{L^{2}(1,T)}^2d\eta,\\
&{\|e^{a\nu^{1/3}t}\nabla_L\td{f}\|_{L^{2}L^{2}}^2}
\le{(2\pi)^2}\sum_{k\in\mathbb{Z}}\int_{\eta\in\R}\sum_{l\in\mathbb{Z}}
\big\|e^{a\nu^{1/3}t}(k^2+(\eta-kt)^2+l^2)^{\frac{1}{2}}\widehat{\td{f}}(\cdot,k,\eta,l)\big\|_{L^{2}(1,T)}^2d\eta.
\end{align*}
Then we can conclude that
\begin{align*}
&\|e^{a\nu^{1/3}t}\td{f}\|_{L^{\infty}L^{2}}^2+
\|e^{a\nu^{1/3}t}\nabla_L\Delta_L^{-1}\partial_{\overline{x}}\td{f}\|_{L^{2}L^{2}}^2
+\nu^{\frac{1}{3}}\|e^{a\nu^{1/3}t}\td{f}\|_{L^{2}L^{2}}^2+\nu\|e^{a\nu^{1/3}t}\nabla_L\td{f}\|_{L^{2}L^{2}}^2
\\ &\leq\sum_{k\in\mathbb{Z}}\int_{\eta\in\R}\sum_{l\in\mathbb{Z}}\bigg(
\big\|e^{a\nu^{1/3}t}\widehat{\td{f}}\big\|_{L^{\infty}(1,T)}^2+\big\|e^{a\nu^{1/3}t} k(k^2+(\eta-kt)^2+l^2)^{-\frac{1}{2}}\widehat{\td{f}}\big\|_{L^{2}(1,T)}^2
\\&\qquad+\nu^{\frac{1}{3}}\big\|e^{a\nu^{1/3}t}\widehat{\td{f}}\big\|_{L^{2}(1,T)}^2
+\nu\big\|e^{a\nu^{1/3}t}(2\pi)(k^2+(\eta-kt)^2+l^2)^{\frac{1}{2}}\widehat{\td{f}}\big\|_{L^{2}(1,T)}^2\bigg)d\eta\\
&\leq C\sum_{k\in\mathbb{Z}}\int_{\eta\in\R}\sum_{l\in\mathbb{Z}}\bigg(\big|\widehat{\td{f}}(1,k,\eta,l)\big|^2
+\big\|e^{a\nu^{1/3}t}(2\pi)(k^2+(\eta-kt)^2+l^2)^{\frac{1}{2}} \widehat{\td{f}_1}\big\|_{L^2(1,T)}^2\\&\qquad+\nu^{-\frac{1}{3}}\big\|e^{a\nu^{1/3}t} \widehat{\td{f}_2}\big\|_{L^2(1,T)}^2+\nu^{-1}\big\|e^{a\nu^{1/3}t}\widehat{\td{f}_3}\big\|_{L^2(1,T)}^2\bigg)d\eta\\
&\leq C\Big(\|\td{f}(1)\|_{L^{2}}^2
+\|e^{a\nu^{1/3}t}\nabla_L\td{f}_1\|_{L^{2}L^{2}}^2+\nu^{-\frac{1}{3}}\|e^{a\nu^{1/3}t}\td{f}_2\|_{L^{2}L^{2}}^2
+\nu^{-1}\|e^{a\nu^{1/3}t}\td{f}_3\|_{L^{2}L^{2}}^2\Big).
\end{align*}
This gives our result by the fact that $\|{f}\|_{L^{2}}=\|\td{f}\|_{L^{2}},\ \|\nabla\Delta^{-1}\partial_{{x}}{f}\|_{L^{2}}=\|\nabla_L\Delta_L^{-1}\partial_{\overline{x}}\td{f}\|_{L^{2}},$ and $\|\nabla{f}\|_{L^{2}}=\|\nabla_L\td{f}\|_{L^{2}}$.
\end{proof}
\smallskip

The following proposition will be used to estimate $\Delta u_{\neq}^2$ and $u_{\neq}^3$.

\begin{Proposition}\label{prop:decay-L0-2}
Let $(f,h)$ solve
\beno
\cL_0 f=\Delta f_1,\quad \cL_0 h-2\partial_x\partial_z\Delta^{-2}f=h_1,
\eeno
for $t\in [1,T]$. If $P_0 f=P_0 f_1=P_0 h=P_0 h_1=0,$ then for
$a\in[0,4]$, it holds that
\begin{align*}
\|f\|_{X_a}^2+\|(\partial_x^2+\partial_z^2)h\|_{X_a}^2\leq& C\Big(\|f(1)\|_{L^2}^2+\|h(1)\|_{H^2}^2\\&+\nu^{-1}\|e^{a\nu^{1/3}t} \nabla f_1\|_{L^2L^2}^2+\nu^{-1}\|e^{a\nu^{1/3}t} (\partial_x,\partial_z) h_1\|_{L^2L^2}^2\Big).
\end{align*}
\end{Proposition}

\begin{proof}
Thanks to $\cL_0 f=\Delta f_1=\text{div}\nabla f_1, $  we infer from Proposition \ref{prop:decay-L0}  that
\begin{align*}
\|f\|_{X_a}^2\leq C\big(\|f(1)\|_{L^2}^2+\nu^{-1}\|e^{a\nu^{1/3}t} \nabla f_1\|_{L^2L^2}^2\big).
\end{align*}
So, it suffices to estimate $\|(\partial_x^2+\partial_z^2)h\|_{X_a}^2. $ Let $(\td{f},\td{f}_1,\td{h},\td{h}_1)(t,\overline{x},y,z)=({f},f_1,h,h_1)(t,{x},y,z)$ with $(\overline{x},y,z)=(x-ty,y,z)$. Then we have
\begin{align*}
&\partial_t\td{f}-\nu\Delta_L\td{f}=\Delta_L\td{f}_1,\quad\partial_t\td{h}-\nu\Delta_L\td{h}-2\partial_{\overline{x}}\partial_z\Delta_L^{-2}\td{f}=\td{h}_1.
\end{align*}
Thanks to $P_0 f=P_0 f_1=P_0 h=P_0 h_1=0,$ we have $ \widehat{\td{f}}=\widehat{\td{f}_1}=\widehat{\td{h}}=\widehat{\td{h}_1}=0$ for $k=0.$ Taking the Fourier transform,  we obtain
\begin{align*}
&\partial_t \widehat{\td{f}}+\nu(2\pi)^2\big(k^2+(\eta-kt)^2+l^2\big) \widehat{\td{f}}=(2\pi)^2\big(k^2+(\eta-kt)^2+l^2\big)\widehat{\td{f}_{1}},\\ &\partial_t \widehat{\td{h}}+\nu(2\pi)^2\big(k^2+(\eta-kt)^2+l^2\big) \widehat{\td{h}}+2kl(2\pi)^{-2}\big(k^2+(\eta-kt)^2+l^2\big)^{-2}\widehat{\td{f}}= \widehat{\td{h}_{1}}.
\end{align*}

Now we deduce from Lemma \ref{lem3b} that for every $k,l\in\mathbb{Z},\ \eta\in\R,\ k\neq 0,$
\begin{align*}
&\big\|e^{a\nu^{1/3}t}\widehat{\td{f}}\big\|_{L^{\infty}(1,T)}^2\leq C\big|\widehat{\td{f}}(1,k,\eta,l)\big|^2
+C\nu^{-1}\big\|e^{a\nu^{1/3}t}(k^2+(\eta-kt)^2+l^2)^{\frac{1}{2}}\widehat{\td{f}_1}\big\|_{L^2(1,T)}^2,
\end{align*}
and then
\begin{align*}
&\big\|e^{a\nu^{1/3}t}\widehat{\td{h}}\big\|_{L^{\infty}(1,T)}^2+\big\|e^{a\nu^{1/3}t} k(k^2+(\eta-kt)^2+l^2)^{-\frac{1}{2}}\widehat{\td{h}}\big\|_{L^{2}(1,T)}^2
+\nu^{\frac{1}{3}}\big\|e^{a\nu^{1/3}t}\widehat{\td{h}}\big\|_{L^{2}(1,T)}^2
\\&\quad+\nu\big\|e^{a\nu^{1/3}t}(2\pi)(k^2+(\eta-kt)^2+l^2)^{\frac{1}{2}}\widehat{\td{h}}\big\|_{L^{2}(1,T)}^2\\
&\leq C\Big(\big|\widehat{\td{h}}(1,k,\eta,l)\big|^2+\big\|e^{a\nu^{1/3}t}l(k^2+(\eta-kt)^2+l^2)^{-\frac{3}{2}} \widehat{\td{f}}\big\|_{L^2(1,T)}^2|k|(k^2+l^2)^{-\frac{1}{2}}\\&
\quad+\nu^{-1}\big\|e^{a\nu^{1/3}t}(k^2+(\eta-kt)^2+l^2)^{-\frac{1}{2}}\widehat{\td{h}_1}\big\|_{L^2(1,T)}^2\Big)\\
&\leq C\Big(\big|\widehat{\td{h}}(1,k,\eta,l)\big|^2+\big\|e^{a\nu^{1/3}t}\widehat{\td{f}}\big\|_{L^{\infty}(1,T)}^2
\big\|l(k^2+(\eta-kt)^2+l^2)^{-\frac{3}{2}}\big\|_{L^2(\R)}^2|k|(k^2+l^2)^{-\frac{1}{2}}\\
&\quad+\nu^{-1}\big\|e^{a\nu^{1/3}t}(k^2+l^2)^{-\frac{1}{2}}\widehat{\td{h}_1}\big\|_{L^2(1,T)}^2\Big)\\
&\leq C\Big(\big|\widehat{\td{h}}(1,k,\eta,l)\big|^2
+\nu^{-1}\big\|e^{a\nu^{1/3}t}(k^2+(\eta-kt)^2+l^2)^{\frac{1}{2}}\widehat{\td{f}_1}\big\|_{L^2(1,T)}^2
(k^2+l^2)^{-2}\\&\quad
+\big|\widehat{\td{f}}(1,k,\eta,l)\big|^2(k^2+l^2)^{-2}
+\nu^{-1}(k^2+l^2)^{-1}\big\|e^{a\nu^{1/3}t}\widehat{\td{h}_1}\big\|_{L^2(1,T)}^2\Big),
\end{align*}
here we used the fact that
\begin{align*}
\left\|l(k^2+(\eta-kt)^2+l^2)^{-\frac{3}{2}} \right\|_{L^2(\R)}^2|k|&=\int_{\R}\frac{l^2|k|dt}{(k^2+(\eta-kt)^2+l^2)^3}=\int_{\R}\frac{l^2dt}{(k^2+t^2+l^2)^3}\\&
=Cl^2(k^2+l^2)^{-\frac{5}{2}}\leq C(k^2+l^2)^{-\frac{3}{2}}.
\end{align*}

Let $h_2=(\partial_x^2+\partial_z^2)h,\ \td{h}_2(t,\overline{x},y,z)=h_2(t,{x},y,z)$. Then $\td{h}_2=(\partial_{\overline{x}}^2+\partial_z^2)\td{h},\ \widehat{\td{h}_2}=-(2\pi)^2(k^2+l^2)\widehat{\td{h}}.$
Then we conclude that
 \begin{align*}
&\|(\partial_x^2+\partial_z^2)h\|_{X_a}^2=\|h_2\|_{X_a}^2=\|e^{a\nu^{1/3}t}\td{h}_2\|_{L^{\infty}L^{2}}^2+
\|e^{a\nu^{1/3}t}\nabla_L\Delta_L^{-1}\partial_{\overline{x}}\td{h}_2\|_{L^{2}L^{2}}^2
\\&\quad+\nu^{\frac{1}{3}}\|e^{a\nu^{1/3}t}\td{h}_2\|_{L^{2}L^{2}}^2+\nu\|e^{a\nu^{1/3}t}\nabla_L\td{h}_2\|_{L^{2}L^{2}}^2\\
&\leq\sum_{k\in\mathbb{Z}}\int_{\eta\in\R}\sum_{l\in\mathbb{Z}}(2\pi)^4\bigg(
\big\|e^{a\nu^{1/3}t}\widehat{\td{h}}\big\|_{L^{\infty}(1,T)}^2+\big\|e^{a\nu^{1/3}t} k(k^2+(\eta-kt)^2+l^2)^{-\frac{1}{2}}\widehat{\td{h}}\big\|_{L^{2}(1,T)}^2
\\&\quad+\nu^{\frac{1}{3}}\big\|e^{a\nu^{1/3}t}\widehat{\td{h}}\big\|_{L^{2}(1,T)}^2
+\nu\big\|e^{a\nu^{1/3}t}(2\pi)(k^2+(\eta-kt)^2+l^2)^{\frac{1}{2}}\widehat{\td{h}}\big\|_{L^{2}(1,T)}^2
\bigg)(k^2+l^2)^2d\eta\\
&\leq C\sum_{k\in\mathbb{Z}}\int_{\eta\in\R}\sum_{l\in\mathbb{Z}}\bigg(\big|\widehat{\td{f}}(1,k,\eta,l)\big|^2
+\nu^{-1}\big\|e^{a\nu^{1/3}t}(2\pi)(k^2+(\eta-kt)^2+l^2)^{\frac{1}{2}} \widehat{\td{f}_1}\big\|_{L^2(1,T)}^2\\&
\quad+\big|\widehat{\td{h}}(1,k,\eta,l)\big|^2(2\pi)^4(k^2+l^2)^2
+\nu^{-1}\big\|e^{a\nu^{1/3}t}\widehat{\td{h}_1}\big\|_{L^2(1,T)}^2(2\pi)^2(k^2+l^2)\bigg)d\eta\\
&=C\big(\|\td{f}(1)\|_{L^{2}}^2
+\nu^{-1}\|e^{a\nu^{1/3}t}\nabla_L\td{f}_1\|_{L^{2}L^{2}}^2+
\|(\partial_{\overline{x}}^2+\partial_z^2)\td{h}(1)\|_{L^{2}}^2\\&
\quad+\nu^{-1}\|e^{a\nu^{1/3}t}(\partial_{\overline{x}},\partial_z)\td{h}_1\|_{L^{2}L^{2}}^2\big)\\
&=C\big(\|{f}(1)\|_{L^{2}}^2
+\nu^{-1}\|e^{a\nu^{1/3}t}\nabla{f}_1\|_{L^{2}L^{2}}^2+
\|(\partial_{{x}}^2+\partial_z^2){h}(1)\|_{L^{2}}^2\\&
\quad+\nu^{-1}\|e^{a\nu^{1/3}t}(\partial_{{x}},\partial_z){h}_1\|_{L^{2}L^{2}}^2\big).
\end{align*}
This gives our result.
\end{proof}

\subsection{Decay estimates of the linearized equation $\cL f=g$}

\begin{Proposition}\label{prop:decay-L}
Let $f$ solve the linear equation
\beno
\cL f=\partial_xf_1+f_2+\text{div} f_3
\eeno
for $t\in [1,T]$. Assume that
\ben\label{ass:u1-smallc1}
\|\bu^1\|_{H^4}+\|\partial_t\bu^1\|_{H^2}/\nu<c_{1}
\een
for some small constant $c_1$ independent of $\nu$ and $T$. If $P_0 f=P_0 f_1=P_0 f_2=P_0 f_3=0,$  then for $a\in[0,4]$, it holds that \begin{align*}
&\|f\|_{X_a}^2\leq C\Big(\|f(1)\|_{L^2}^2+\|e^{a\nu^{1/3}t}\nabla f_1\|_{L^2L^2}^2+\nu^{-\frac{1}{3}}\|e^{a\nu^{1/3}t} f_2\|_{L^2L^2}^2+\nu^{-1}\|e^{a\nu^{1/3}t} f_3\|_{L^2L^2}^2\Big).
\end{align*}
\end{Proposition}

We define the map $(Id +g): \Omega\to  \R^3$ by
\beno
(x,y,z)\mapsto(X,Y,Z)=(x,V(t,y,z),z)\quad V=y+\bu^1.
\eeno
The following lemma {is} classical.

\begin{Lemma}\label{lem:com-inv}
Let $g:\T\times\R\times\T\to \R^3$ be a $C^1$ map such that $\|\nabla g\|_{L^{\infty}}<1/2$. Then  it holds that
\begin{itemize}

\item[1.] $\|f\circ(Id +g)\|_{L^p}\sim\|f\|_{L^p},\ \|\nabla(f\circ(Id +g))\|_{L^p}\sim\|\nabla f\|_{L^p}$ for every $1\leq p\leq+\infty$ and $f\in W^{1,p}.$ Here $A\sim B$ means $C^{-1}A\leq B\leq CA$ for some absolute constant $C.$

\item[2.] there exists a unique $C^1$ solution $h$ to $ h(x,y,z)=g\big((x,y,z)+h(x,y,z)\big)$ satisfying $\|\nabla h\|_{L^{\infty}}\leq C\|\nabla g\|_{L^{\infty}}. $
\end{itemize}
\end{Lemma}

If $\|\nabla \bu^1\|_{L^{\infty}}<1/2,$ let $F(t,X,Y,Z)$ be such that $F(t,x,V(t,y,z),z)=f(t,x,y,z)$. By Lemma \ref{lem:com-inv}, $F(t,X,Y,Z)$ is well defined, and as $V$ is $C^1$ in $t,$ if $f$ is $C^1$ in $t,$ then so does $F$. Let $\psi_t(t,Y,Z),\ \psi_y(t,Y,Z),\ \psi_z(t,Y,Z)$ be such that
\begin{align*}
&\psi_t\big(t,V(t,y,z),z\big)=\partial_t\bu^1(t,y,z)=\partial_tV(t,y,z),\\
&\psi_y\big(t,V(t,y,z),z\big)=\partial_y\bu^1(t,y,z)=\partial_yV(t,y,z)-1,\\
&\psi_z\big(t,V(t,y,z),z\big)=\partial_z\bu^1(t,y,z)=\partial_zV(t,y,z).\end{align*}
It is easy to check that
\begin{align*}&\partial_t f=(\partial_t+\psi_t\partial_Y)F, \quad  \partial_x f=\partial_XF,\\ &\partial_y f=(1+\psi_y)\partial_YF,\quad \partial_z f=(\partial_Z+\psi_z\partial_Y)F.
\end{align*}
We introduce the notations
\beno
&&\partial_Y^t=(1+\psi_y)\partial_Y,\quad \partial_Z^t=\partial_Z+\psi_z\partial_Y, \\
&&G=(1+\psi_y)^2+\psi_z^2-1, \quad H\big(t,V(t,y,z),z\big)=\Delta \bu^1(t,y,z).
\eeno
Then $V\partial_x f=Y\partial_X F$ and
\begin{align*}
\Delta f=\Delta_t F=\big((\partial_X)^2+(\partial_Y^t)^2+(\partial_Z^t)^2\big)F=\Delta F+G\partial_Y^2F+2\psi_z\partial_Z\partial_YF+H\partial_YF.
\end{align*}
Thus, the operator $\cL$ \underline{is} transformed into
\begin{align*}
\cL f=&\partial_t f-\nu\Delta f+V\partial_x f=(\partial_t+\psi_t\partial_Y)F-\nu\big(\Delta F+G\partial_Y^2F+2\psi_z\partial_Z\partial_YF+H\partial_YF\big)\\&+Y\partial_X F=\cL_0 F-\nu\big(G\partial_Y^2F+2\psi_z\partial_Z\partial_YF\big)+(\psi_t-\nu H)\partial_YF.
\end{align*}
Notice that
\beno
&&\partial_YF=\partial_y f/\partial_y V,\quad \partial_ZF=\partial_z f-\partial_z V\partial_y f/\partial_y V,\\
&&G(t,V(t,y,z),z)=(1+\partial_y\bu^1(t,y,z))^2+\partial_z\bu^1(t,y,z)^2-1=(\partial_y V)^2+(\partial_z V)^2-1,
\eeno
then we have
\begin{align*}
&\partial_YG+2\partial_Z\psi_z=\partial_y((\partial_y V)^2+(\partial_z V)^2-1)/\partial_y V+2(\partial_z-(\partial_z V/\partial_y V)\partial_y)\partial_z V\\
&=2\partial_y^2 V+2\partial_z V\partial_y\partial_z V/\partial_y V+2\partial_z^2 V-2(\partial_z V/\partial_y V)\partial_y\partial_z V=2\Delta V=2\Delta \bu^1=2H.
\end{align*}
Thus, we can write $\cL f $ in the divergence form
\begin{align}\label{LL0}
\cL f=\cL_0 F-\nu\partial_Y(G\partial_YF)-2\nu\partial_Z(\psi_z\partial_YF)+(\psi_t+\nu H)\partial_YF.
\end{align}
Now $ \cL$ could be viewed as a small perturbation of $ \cL_0.$

By Lemma \ref{lem:com-inv}, we have
\beno
\|F\|_{L^2}\sim \|f\|_{L^2},\quad \|\nabla F\|_{L^2}\sim \|\nabla f\|_{L^2},
\eeno
We also need the equivalence of $H^{-1}$ norm.

\begin{Lemma}\label{lemHs2}
If $\|\nabla \bu^1\|_{L^{\infty}}<1/2,\ \| \bu^1\|_{H^{4}}<1,\ P_0 f=0$, then
$\|\nabla\Delta^{-1} F\|_{L^2}\sim \|\nabla\Delta^{-1} f\|_{L^2},$ thus $ \|F\|_{X_a}\sim \|f\|_{X_a}.$
\end{Lemma}

\begin{proof}
We still have $P_0 F=0.$ As $\|\bu^1\|_{H^{4}}<1,\ V=y+\bu^1,$ we have
\beno
\|\nabla V\|_{L^{\infty}}+ \|\nabla^2 V\|_{L^{\infty}}\leq 1+\|\nabla \bu^1\|_{L^{\infty}}+ \|\nabla^2 \bu^1\|_{L^{\infty}}\leq 1+C\|\bu^1\|_{H^{4}}\leq C.
\eeno
Now we introduce
\beno
&&F_1=\Delta^{-1} F,\quad f_1(t,x,y,z)=F_1\big(t,x,V(t,y,z),z\big),\\
&&f_2=\Delta^{-1} f,\quad f_2(t,x,y,z)=F_2\big(t,x,V(t,y,z),z\big).
\eeno
By Lemma \ref{lem:com-inv},  $\|\nabla F_j\|_{L^2}\sim \|\nabla f_j\|_{L^2}.$ On the one hand, we have
\begin{align*}
\|\nabla\Delta^{-1} F\|_{L^2}^2=&\|\nabla F_1\|_{L^2}^2=-\langle F,F_1\rangle=-\langle f,\partial_y Vf_1\rangle=\langle \nabla f_2,\nabla(\partial_y Vf_1)\rangle\\
\leq&\|\nabla f_2\|_{L^2}\|\nabla (\partial_y Vf_1)\|_{L^2}\leq \|\nabla f_2\|_{L^2}(\|\nabla V\|_{L^{\infty}}\|\nabla f_1\|_{L^2}+ \|\nabla^2 V\|_{L^{\infty}}\| f_1\|_{L^2})\\
\leq&C\|\nabla f_2\|_{L^2}\big(\|\nabla f_1\|_{L^2}+ \| f_1\|_{L^2}\big)\leq C\|\nabla f_2\|_{L^2}\|\nabla f_1\|_{L^2}\leq C\|\nabla f_2\|_{L^2}\|\nabla F_1\|_{L^2},
\end{align*}
which gives $\|\nabla F_1\|_{L^2}\leq C\|\nabla f_2\|_{L^2}. $ On the other hand, we have
\begin{align*}
\|\nabla\Delta^{-1} f\|_{L^2}^2=&\|\nabla f_2\|_{L^2}^2=-\langle f,f_2\rangle=-\langle F,F_2/\psi_y\rangle=\langle \nabla F_1,\nabla(F_2/\psi_y)\rangle\\
\leq&\|\nabla F_1\|_{L^2}\|\nabla (F_2/\psi_y)\|_{L^2}.
\end{align*}
Since $\psi_y\big(t,V(t,y,z),z\big)=1-\partial_y \bu^1(t,y,z)$ and  $\|\nabla \bu^1\|_{L^{\infty}}<1/2,$ we have $1/2<\psi_y<3/2, $ and by Lemma \ref{lem:com-inv} we have $\|\nabla \psi_y\|_{L^{\infty}}\leq C\|\nabla \partial_y \bu^1\|_{L^{\infty}}\leq C, $ thus $\|\nabla (1/\psi_y)\|_{L^{\infty}}\leq C,$ and
$$\|\nabla (F_2/\psi_y)\|_{L^2}\leq C(\|\nabla F_2\|_{L^2}+ \| F_2\|_{L^2})\leq C\|\nabla F_2\|_{L^2}\leq C\|\nabla f_2\|_{L^2}.$$
Thus we obtain
\beno
\|\nabla f_2\|_{L^2}^2\leq\|\nabla F_1\|_{L^2}\|\nabla (F_2/\psi_y)\|_{L^2}\leq C\|\nabla F_1\|_{L^2}\|\nabla f_2\|_{L^2},
\eeno
which gives $\|\nabla f_2\|_{L^2}\leq C\|\nabla F_1\|_{L^2}.$

Thus, $\|\nabla\Delta^{-1} F\|_{L^2}=\|\nabla F_1\|_{L^2}\sim \|\nabla f_2\|_{L^2}=\|\nabla\Delta^{-1} f\|_{L^2}$.
\end{proof}\smallskip

Now we are in a position to prove Proposition \ref{prop:decay-L}.

\begin{proof}
First of all, we choose $c_1$ small enough so that  $\|\nabla \bu^1\|_{L^{\infty}}\leq C\| \bu^1\|_{H^{4}}<1/2$.

Let $F(t,x,V(t,y,z),z)=f(t,x,y,z),\ F_j(t,x,V(t,y,z),z)=f_j(t,x,y,z)$ for $j=1,2,$ and $F_3(t,x,V(t,y,z),z)=\text{div}f_3(t,x,y,z)$. By \eqref{LL0}, we have\begin{align*}
\cL_0 F-\nu\partial_Y(G\partial_YF)-2\nu\partial_Z(\psi_z\partial_YF)+(\psi_t+\nu H)\partial_YF=\pa_XF_1+F_2+F_3.
\end{align*}
It follows from Proposition \ref{prop:decay-L0} that
\begin{align*}
\|F\|_{X_a}^2\leq& C\Big(\|F(1)\|_{L^2}^2+\|e^{a\nu^{1/3}t}\nabla F_1\|_{L^2L^2}^2+\nu^{-\frac{1}{3}}\|e^{a\nu^{1/3}t} F_2\|_{L^2L^2}^2\\&+\nu^{-1}\|e^{a\nu^{1/3}t}\nabla\Delta^{-1} F_3\|_{L^2L^2}^2+\nu\|e^{a\nu^{1/3}t} (G,2\psi_z)\partial_YF\|_{L^2L^2}^2\\&+\nu^{-\frac{1}{3}}\|e^{a\nu^{1/3}t} (\psi_t+\nu H)\partial_YF\|_{L^2L^2}^2\Big).
\end{align*}
By Lemma \ref{lem:com-inv} and Lemma \ref{lemHs2}, we have
\beno
&&\|e^{a\nu^{1/3}t}\nabla F_1\|_{L^2L^2}^2\leq C\|e^{a\nu^{1/3}t}\nabla f_1\|_{L^2L^2}^2,\\
&&\|e^{a\nu^{1/3}t} F_2\|_{L^2L^2}^2\leq  C\|e^{a\nu^{1/3}t} f_2\|_{L^2L^2}^2,\\
&&\|e^{a\nu^{1/3}t}\nabla\Delta^{-1} F_3\|_{L^2L^2}^2\leq C\|e^{a\nu^{1/3}t}f_3\|_{L^2L^2}^2,
\eeno
and by the assumption {\eqref{ass:u1-smallc1}},
\beno
{C}(\nu\|e^{a\nu^{1/3}t} (G,2\psi_z)\partial_YF\|_{L^2L^2}^2+\nu^{-\frac{1}{3}}\|e^{a\nu^{1/3}t} (\psi_t+\nu H)\partial_YF\|_{L^2L^2}^2){\le Cc_1^2\|F\|_{X_a}^2}\le 1/2\|F\|_{X_a}^2.
\eeno
This gives our result by using Lemma \ref{lemHs2} again.
\end{proof}

The following proposition gives the decay estimates of the solution for good derivatives $\pa_x$ and $\pa_z-\kappa\pa_y$,
which have good commutation relations with $\cL$.

\begin{Proposition}\label{prop:decay-L-1}
Let $f$ solve the linear equation
\beno
\cL f=f_1+f_2+f_3
\eeno
for $t\in [1,T]$. Assume that
\ben\label{ass:u1-smallc2}
\|\bu^1\|_{H^4}+\|\partial_t\bu^1\|_{H^2}/\nu<c_{2}.
\een
for some small constant $c_2$ independent of $\nu$ and $T$.
If $P_0 f=P_0 f_1=P_0 f_2=P_0 f_3=0,$  then for $a\in[0,4]$, it holds that \begin{align*}
&\|\partial_x^2f\|_{X_a}^2+\|\partial_x(\partial_z-\kappa\partial_y)f\|_{X_a}^2\leq C\Big(\|f(1)\|_{H^2}^2+\|e^{a\nu^{1/3}t}\Delta f_1\|_{L^2L^2}^2\\
&\quad+\nu^{-\frac{1}{3}}\|e^{a\nu^{1/3}t} \partial_x^2 f_2\|_{L^2L^2}^2+\nu^{-\frac{1}{3}}\|e^{a\nu^{1/3}t} \partial_x(\partial_z-\kappa\partial_y) f_2\|_{L^2L^2}^2+\nu^{-1}\|e^{a\nu^{1/3}t} \partial_xf_3\|_{L^2L^2}^2\Big).
\end{align*}
\end{Proposition}

We need the following lemma, whose proof is a simple exercise.

\begin{Lemma}\label{lem5}
Assume that   $\|\bu^1\|_{H^4}<c_3$ for some small constant $c_3$ independent of $\nu$ and $T$. Let $\kappa(t,y,z)=\f {\partial_z\bu^1} {(1+\partial_y\bu^1)}$ and $(\rho_1,\rho_2)$ be given by \eqref{def:rho12}. Then it holds that
\beno
&&\|\partial_y\bu^1\|_{L^{\infty}}< 1/2,\quad \|\kappa\|_{H^3}\leq C\|\bu^1\|_{H^4},\quad \|\partial_t\kappa\|_{H^1}\leq C\|\partial_t\bu^1\|_{H^2},\quad \|\kappa\|_{H^1}\leq C\|\bu^1\|_{H^2},\\
&& \|\rho_1\|_{H^2}+\|\rho_2\|_{H^2}\leq C\|\bu^1\|_{H^4},\quad \|\partial_t\rho_1\|_{L^2}\leq C\|\partial_t\bu^1\|_{H^2}.
\eeno
\end{Lemma}

Now we prove Proposition \ref{prop:decay-L-1} (take $c_2=\min(c_1,c_3)$).

\begin{proof}
As $\cL\partial_x^2f=\partial_x^2\cL f=\partial_x^2f_1+\partial_x^2f_2+ \partial_x^2f_3,$ we infer from Proposition \ref{prop:decay-L} that
\begin{align}\label{f1Xa}
\|\partial_x^2f\|_{X_a}^2\leq& C\Big(\|\partial_x^2f(1)\|_{L^2}^2+\|e^{a\nu^{1/3}t}\nabla\partial_x f_1\|_{L^2L^2}^2+\nu^{-\frac{1}{3}}\|e^{a\nu^{1/3}t} \partial_x^2 f_2\|_{L^2L^2}^2\\ \nonumber&\quad+\nu^{-1}\|e^{a\nu^{1/3}t} \partial_xf_3\|_{L^2L^2}^2\Big).
\end{align}

Thanks to $(\partial_z-\kappa\partial_y)V=0$, we have
\begin{align*}
(\partial_z-\kappa\partial_y)\cL f-\cL((\partial_z-\kappa\partial_y)f)&=(\partial_t \kappa-\nu\Delta \kappa)\partial_yf-2\nu\nabla\kappa\cdot\nabla \partial_yf\\
&=(\partial_t \kappa+\nu\Delta \kappa)\partial_yf-2\nu\text{div}(\partial_yf\nabla\kappa).
\end{align*}
which gives
\beno
\partial_x((\partial_z-\kappa\partial_y)\cL f-\cL((\partial_z-\kappa\partial_y)f))=\Delta h_1-2\nu\text{div}(\partial_x\partial_yf\nabla\kappa ),
\eeno
where $\Delta h_1=(\partial_t \kappa+\nu\Delta \kappa)\partial_x\partial_yf $ and $P_0 h_1=0$. Hence,
\begin{align*}
&\cL(\partial_x(\partial_z-\kappa\partial_y)f)=\partial_x(\partial_z-\kappa\partial_y)\big(f_1+f_2+f_3\big)-\Delta h_1+2\nu\text{div}(\partial_x\partial_yf\nabla\kappa ).
\end{align*}
Then it follows from Proposition \ref{prop:decay-L} that
\begin{align*}
&\|\partial_x(\partial_z-\kappa\partial_y)f\|_{X_a}^2\leq C\Big(\|\partial_x(\partial_z-\kappa\partial_y)f(1)\|_{L^2}^2+\|e^{a\nu^{1/3}t}\nabla(\partial_z-\kappa\partial_y) f_1\|_{L^2L^2}^2\\&\quad+\nu^{-\frac{1}{3}}\|e^{a\nu^{1/3}t} \partial_x(\partial_z-\kappa\partial_y) f_2\|_{L^2L^2}^2+\nu^{-1}\|e^{a\nu^{1/3}t}(\partial_y\kappa f_3, -\kappa\partial_xf_3,\partial_xf_3)\|_{L^2L^2}^2\\&\quad+\nu^{-1}\|e^{a\nu^{1/3}t}\nabla h_1\|_{L^2L^2}^2+\nu\|e^{a\nu^{1/3}t}\partial_x\partial_yf\nabla\kappa\|_{L^2L^2}^2\Big).
\end{align*}
Here we used $\partial_x(\partial_z-\kappa\partial_y)f_3=\partial_x(\partial_y\kappa f_3) -\partial_y(\kappa\partial_xf_3)+\partial_z\partial_xf_3.$

By Lemma \ref{lem5}, we have
\beno
\|\kappa\|_{L^{\infty}}+\|\nabla\kappa\|_{L^{\infty}}+\|\kappa\|_{H^2}\leq C\|\kappa\|_{H^3}\leq C\|\bu^1\|_{H^4}\leq C.
\eeno
Then we infer that
\begin{align*}
&\|\partial_x(\partial_z-\kappa\partial_y)f(1)\|_{L^2}\leq (1+\|\kappa\|_{L^{\infty}})\|\nabla^2f(1)\|_{L^2}\leq C\|f(1)\|_{H^2},
\end{align*}
and $\|\nabla(\partial_z-\kappa\partial_y) f_1\|_{L^2}\leq C(1+\|\kappa\|_{H^{2}})\|\nabla f_1\|_{H^1}\leq C\|\Delta f_1\|_{L^2}$, hence,
\begin{align*}
&\|e^{a\nu^{1/3}t}\nabla(\partial_z-\kappa\partial_y) f_1\|_{L^2L^2}^2\leq C\|e^{a\nu^{1/3}t}\Delta f_1\|_{L^2L^2}^2.
\end{align*}
Thanks to $\|(\partial_y\kappa f_3, -\kappa\partial_xf_3,\partial_xf_3)\|_{L^2}\leq \|\nabla\kappa\|_{L^{\infty}} \| f_3\|_{L^2}+ (1+\|\kappa\|_{L^{\infty}})\| \partial_xf_3\|_{L^2}\leq C\| \partial_xf_3\|_{L^2},
$ we have
\begin{align*}
&\|e^{a\nu^{1/3}t}(\partial_y\kappa f_3, -\kappa\partial_xf_3,\partial_xf_3)\|_{L^2L^2}^2\leq C\|e^{a\nu^{1/3}t} \partial_xf_3\|_{L^2L^2}^2.
\end{align*}
Thanks to $\|\partial_x\partial_yf\nabla\kappa\|_{L^2}\leq \|\partial_x\nabla f\|_{L^2}\|\nabla\kappa\|_{L^{\infty}} \leq C\|\nabla \partial_x^2f\|_{L^2},
$ we have\begin{align*}
&\nu\|e^{a\nu^{1/3}t}\partial_x\partial_yf\nabla\kappa\|_{L^2L^2}^2\leq C\nu\|e^{a\nu^{1/3}t} \nabla \partial_x^2f\|_{L^2L^2}^2\leq C\|\partial_x^2f\|_{X_a}^2.
\end{align*}
By Lemma \ref{lem5}, we have
\beno
\|\partial_t \kappa+\nu\Delta \kappa\|_{L^{4}}\leq \|\partial_t \kappa+\nu\Delta \kappa\|_{H^{1}}\leq \|\partial_t \kappa\|_{H^{1}}+\nu\|\kappa\|_{H^3}\leq C(\|\partial_t\bu^1\|_{H^2}+\nu\|\bu^1\|_{H^4})\leq C\nu,
\eeno
which gives
\begin{align*}
&\|\nabla h_1\|_{L^2}^2=-\langle\Delta h_1,h_1\rangle=-\langle(\partial_t \kappa+\nu\Delta \kappa)\partial_x\partial_yf,h_1\rangle\\ &\leq \|\partial_t \kappa+\nu\Delta \kappa\|_{L^{4}}\|\partial_x\partial_yf\|_{L^{2}}\|h_1\|_{L^{4}}\leq C\nu\|\partial_x\nabla f\|_{L^{2}}\|h_1\|_{H^{1}}\leq C\nu\|\nabla\partial_x^2 f\|_{L^{2}}\|\nabla h_1\|_{L^2}.
\end{align*}
Thus, $\|\nabla h_1\|_{L^2}\leq C\nu\|\nabla\partial_x^2 f\|_{L^{2}}$ and\begin{align*}
&\nu^{-1}\|e^{a\nu^{1/3}t}\nabla h_1\|_{L^2L^2}^2\leq C\nu\|e^{a\nu^{1/3}t} \nabla \partial_x^2f\|_{L^2L^2}^2\leq C\|\partial_x^2f\|_{X_a}^2.
\end{align*}
Summing up, we conclude that
\begin{align}\label{f1zXa}
&\|\partial_x(\partial_z-\kappa\partial_y)f\|_{X_a}^2\leq C\Big(\|f(1)\|_{H^2}^2+\|e^{a\nu^{1/3}t}\Delta f_1\|_{L^2L^2}^2\\ \nonumber&+\nu^{-\frac{1}{3}}\|e^{a\nu^{1/3}t} \partial_x(\partial_z-\kappa\partial_y) f_2\|_{L^2L^2}^2+\nu^{-1}\|e^{a\nu^{1/3}t}\partial_xf_3\|_{L^2L^2}^2+\|\partial_x^2f\|_{X_a}^2\Big).
\end{align}
Now the proposition follows from \eqref{f1Xa} and \eqref{f1zXa} and the fact that $\|e^{a\nu^{1/3}t}\nabla\partial_x f_1\|_{L^2L^2} \leq \|e^{a\nu^{1/3}t}\Delta f_1\|_{L^2L^2}$.
\end{proof}

\subsection{Decay estimates of the linearized equation $\cL_1f=g$}

Notice that $ \Delta\cL f=\cL \Delta f+\Delta V\partial_x f+2\nabla V\cdot\nabla\partial_xf$ and
\begin{align*}
&\Delta\cL_1 f=\cL \Delta f+\Delta V\partial_x f+2\nabla V\cdot\nabla\partial_xf\\
&\quad-2(\partial_y+\kappa\partial_z) (\partial_yV\partial_xf)-2(\Delta\kappa)\partial_z \Delta^{-1}(\partial_yV\partial_xf)-4\nabla\kappa\cdot\nabla\partial_z \Delta^{-1}(\partial_yV\partial_xf).
\end{align*}
Using the facts that $\partial_xV=0$ and $\partial_zV=\kappa\partial_yV$, we have
\begin{align*}
&2\nabla V\cdot\nabla\partial_xf-2(\partial_y+\kappa\partial_z) (\partial_yV\partial_xf)\\&=2\partial_yV(\partial_y+\kappa\partial_z)\partial_x f-2(\partial_y+\kappa\partial_z)(\partial_yV\partial_xf)\\
&=-2(\partial_y^2V+\kappa\partial_z\partial_yV)\partial_xf.
\end{align*}
Then we deduce that
\begin{align*}
&\Delta\cL_1 f=\cL \Delta f+(\Delta V-2\partial_y^2V-2\kappa\partial_z\partial_yV)\partial_x f\\&\quad-2(\Delta\kappa)\partial_z \Delta^{-1}(\partial_yV\partial_xf)-4\nabla\kappa\cdot\nabla\partial_z \Delta^{-1}(\partial_yV\partial_xf).
\end{align*}
The following proposition shows that $\Delta\cL_1-\cL \Delta$ is good.

\begin{Proposition}\label{prop:decay-L1}
Let $f$ solve $\cL_1 f=f_1$ for $t\in [1,T]$. Assume that
\ben\label{ass:u1-smallc4}
\|\bu^1\|_{H^4}+\|\partial_t\bu^1\|_{H^2}/\nu<c_{4}
\een
for some small constant ${c_4\in (0,c_2)}$ independent of $\nu$ and $T$. If $P_0 f=P_0f_1=0,$  then for $a\in[0,4]$, it holds that
\begin{align*}
&\|\Delta f\|_{X_a}^2\leq C\big(\|\Delta f(1)\|_{L^2}^2+\nu^{-1}\|e^{a\nu^{1/3}t}\nabla f_1\|_{L^2L^2}^2\big).
\end{align*}
\end{Proposition}

\begin{proof}
We introduce
\beno
&&h_1(t,y,z)=\Delta V-2\partial_y^2V-2\kappa\partial_z\partial_yV=\Delta \bu^1-2\partial_y^2\bu^1-2\kappa\partial_z\partial_y\bu^1,\\
&&h_2=\Delta^{-1}(\partial_yVf),\quad h_3=2(\Delta\kappa)\partial_z h_2+4\nabla\kappa\cdot\nabla\partial_z h_2.
\eeno
Then we find that
\begin{align*}
\Delta f_1&=\Delta\cL_1 f=\cL \Delta f+h_1\partial_x f-2(\Delta\kappa)\partial_z\partial_x h_2-4\nabla\kappa\cdot\nabla\partial_z \partial_xh_2\\&=\cL \Delta f+\partial_x(h_1 f)-\partial_xh_3.
\end{align*}
It follows from Proposition \ref{prop:decay-L} that
\begin{align}\label{fXa}
\|\Delta f\|_{X_a}^2\leq C\big(\|\Delta f(1)\|_{L^2}^2+\nu^{-1}\|e^{a\nu^{1/3}t}\nabla f_1\|_{L^2L^2}^2+\|e^{a\nu^{1/3}t}\nabla (h_1 f-h_3)\|_{L^2L^2}^2\big).
\end{align}
By Lemma \ref{lem5}, we have
$\|\kappa\|_{H^3}\leq C\|\bu^1\|_{H^4}\leq Cc_4.$
Then we get
\begin{align*}
\|h_1\|_{H^2}&\leq\|\Delta \bu^1-2\partial_y^2\bu^1\|_{H^2}+2\|\kappa\partial_z\partial_y\bu^1\|_{H^2}\leq C\|\bu^1\|_{H^4}+C\|\kappa\|_{H^2}\|\partial_z\partial_y\bu^1\|_{H^2}\\&\leq C(1+\|\kappa\|_{H^2})\|\bu^1\|_{H^4}\leq C\|\bu^1\|_{H^4}\leq Cc_4,
\end{align*}
which gives
\begin{align*}
\|\nabla (h_1 f)\|_{L^2}\leq \|h_1 f\|_{H^1}\leq C\|h_1\|_{H^2}\| f\|_{H^1}\leq Cc_{4}\|\nabla  f\|_{L^2}.
\end{align*}
As $h_3=2(\Delta\kappa)\partial_z h_2+4\nabla\kappa\cdot\nabla\partial_z h_2, $ we have
\begin{align*}
\|\nabla h_3\|_{L^2}&\leq \|h_3\|_{H^1}\leq C\|\Delta\kappa\|_{H^1}\| \partial_z h_2\|_{H^2}+C\|\nabla\kappa\|_{H^2}\| \nabla\partial_z h_2\|_{H^1}\leq C\|\kappa\|_{H^3}\| h_2\|_{H^3}\\
&\leq Cc_{4}\|\Delta h_2\|_{H^1}=Cc_{4}\| \partial_yVf\|_{H^1}\leq Cc_{4}\|f\|_{H^1}\leq Cc_{4}\|\nabla  f\|_{L^2},
\end{align*}
here we used
\begin{align*}
&\| \partial_yVf\|_{H^1}=\|(1+ \partial_y\bu^1)f\|_{H^1}\leq C(1+\|\partial_y\bu^1\|_{H^2})\|f\|_{H^1}\leq C\|f\|_{H^1}.
\end{align*}
Thus, $\|\nabla (h_1 f-h_3)\|_{L^2}\leq  Cc_{4}\|\nabla  f\|_{L^2}\leq  Cc_{4}\|\nabla \partial_x f\|_{L^2},$ and  then
\begin{align*}
\|e^{a\nu^{1/3}t}\nabla (h_1 f-h_3)\|_{L^2L^2}^2\leq Cc_{4}^2\|e^{a\nu^{1/3}t}\nabla\partial_x f\|_{L^2L^2}^2\leq Cc_{4}^2\|\Delta f\|_{X_a}^2,
\end{align*}
which along with \eqref{fXa} gives
\begin{align*}
&\|\Delta f\|_{X_a}^2\leq C\big(\|\Delta f(1)\|_{L^2}^2+\nu^{-1}\|e^{a\nu^{1/3}t}\nabla f_1\|_{L^2L^2}^2+c_{4}^2\|\Delta f\|_{X_a}^2\big),
\end{align*}
which gives the desired result by taking $ c_4$ so that $Cc_{4}^2<1/2.$
\end{proof}

\section{Basic analysis tools}

In this section, we introduce some basic inequalities and  anisotropic bilinear estimates in Sobolev space, and the velocity estimates in terms of
two components $(u^2, u^3)$, which play an important role for nonlinear interaction of next section.

\subsection{Some basic inequalities}
Recall that for $s > 3/2 , H^s(\Omega)$ is an algebra. Hence,
\beno
\|f_1f_2\|_{H^s}\leq C\|f_1\|_{H^s}\|f_2\|_{H^s},\quad \|f\|_{L^{\infty}}\leq C\|f\|_{H^s}.
\eeno
The following inequalities are classical:
\begin{align*}
&\|f_1f_2\|_{L^2}\leq \|f_1\|_{L^4}\|f_2\|_{L^4}\leq C\|f_1\|_{H^1}\|f_2\|_{H^1},\quad  \|f_1f_2\|_{H^1}\leq C\|f_1\|_{H^1}\|f_2\|_{H^2},\\
&\|f_1f_2\|_{H^2}\leq C\|f_1\|_{H^2}\|f_2\|_{H^2},\quad \|f_1f_2\|_{H^3}\leq C\big(\|f_1\|_{H^3}\|f_2\|_{H^2}+\|f_1\|_{H^2}\|f_2\|_{H^3}\big).
\end{align*}
We will use interpolation inequalities such as $\|f\|_{H^s}\leq\|f\|_{H^{s_1}}^{\theta}\|f\|_{H^{s_2}}^{1-\theta} $ for $s=s_1\theta+s_2(1-\theta),\ \theta\in [0,1],$ and $\|\partial_i\partial_jf\|_{L^2}\leq\|\partial_i^2f\|_{L^2}^{\frac{1}{2}}\|\partial_j^2f\|_{L^{2}}^{\frac{1}{2}} .$  Then we have
$$ \|f_1\|_{H^2}\|f_2\|_{H^3}\leq \|f_1\|_{H^1}^{\frac{1}{2}}\|f_1\|_{H^3}^{\frac{1}{2}}\|f_2\|_{H^2}^{\frac{1}{2}}\|f_2\|_{H^4}^{\frac{1}{2}}
\leq\|f_1\|_{H^3}\|f_2\|_{H^2}+\|f_1\|_{H^1}\|f_2\|_{H^4},$$
which gives
\beno
\|f_1f_2\|_{H^3}\leq C\big(\|f_1\|_{H^3}\|f_2\|_{H^2}+\|f_1\|_{H^1}\|f_2\|_{H^4}\big).
\eeno

\subsection{Anisotropic bilinear estimates}

\begin{Lemma}\label{lem9a1}
If $\partial_x f_1=\partial_x f_2=0,$ then it holds that
\beno
&&\|f_1\|_{L^{\infty}}\leq C\big(\|f_1\|_{H^{1}}+\|\partial_zf_1\|_{H^{1}}\big),\\
&&\|f_1f_2\|_{L^{2}}\leq C\big(\|f_1\|_{H^{1}}+\|\partial_zf_1\|_{H^{1}}\big)\|f_2\|_{L^{2}},\\
&&\|f_1f_2\|_{L^{2}}\leq C\big(\|f_1\|_{L^{2}}+\|\partial_zf_1\|_{L^{2}}\big)\|f_2\|_{H^{1}}.
\eeno
\end{Lemma}\begin{proof}
As $\partial_x f_1=\partial_x f_2=0$, we may write
$f_l(x,y,z)=\sum\limits_{k\in\mathbb{Z}}e^{2\pi ikz}f_{l,k}(y)$ for $l=1,2.$ Then we have
$$
\|f_1\|_{H^1}^2+\|\partial_zf_1\|_{H^1}^2=\sum\limits_{k\in\mathbb{Z}}(1+(2\pi k)^2)(\|f_{1,k}\|_{H^1}^2+(2\pi k)^2\|f_{1,k}\|_{L^2}^2).
$$
Therefore,
\begin{align*}
\Big(\sum\limits_{k\in\mathbb{Z}}\|f_{1,k}\|_{H^1}\Big)^2\leq \sum\limits_{k\in\mathbb{Z}}(1+k^2)\|f_{1,k}\|_{H^1}^2\sum\limits_{k\in\mathbb{Z}}\dfrac{1}{1+k^2}\leq C\big(\|{f_1}\|_{H^1}^2+\|\partial_zf_1\|_{H^1}^2\big),
\end{align*}
which gives
\beno
\|f_1\|_{L^{\infty}}\leq\sum\limits_{k\in\mathbb{Z}}\|f_{1,k}\|_{L^{\infty}}\leq C\sum\limits_{k\in\mathbb{Z}}\|f_{1,k}\|_{H^1}\leq C\big(\|f_1\|_{H^{1}}+\|\partial_zf_1\|_{H^{1}}\big).
\eeno
which is the first inequality. Similarly, we have
\begin{align}\label{f2}\Big(\sum\limits_{k\in\mathbb{Z}}\|f_{1,k}\|_{L^2}\Big)^2\leq \sum\limits_{k\in\mathbb{Z}}(1+k^2)\|f_{1,k}\|_{L^2}^2\sum\limits_{k\in\mathbb{Z}}\dfrac{1}{1+k^2}\leq C\big(\|{f_1}\|_{L^2}^2+\|\partial_zf_1\|_{L^2}^2\big).
\end{align}

The second inequality follows from
\beno
\|f_1f_2\|_{L^{2}}\leq\|f_1\|_{L^{\infty}}\|f_2\|_{L^{2}}\leq C\big(\|f_1\|_{H^{1}}+\|\partial_zf_1\|_{H^{1}}\big)\|f_2\|_{L^{2}}.
\eeno

We write $e^{2\pi imz}f_{1,m}(y)f_2(x,y,z)=\sum\limits_{k\in\mathbb{Z}}e^{2\pi i(k+m)z}f_{1,m}(y)f_{2,k}(y)$. Then we have
\begin{align*}
\|e^{2\pi imz}f_{1,m}(y)f_2\|_{L^2}^2&=\sum\limits_{k\in\mathbb{Z}}\|f_{1,m}f_{2,k}\|_{L^2}^2\leq
\sum\limits_{k\in\mathbb{Z}}\|f_{1,m}\|_{L^2}^2\|f_{2,k}\|_{L^{\infty}}^2\\ &\leq C\sum\limits_{k\in\mathbb{Z}}\|f_{1,m}\|_{L^2}^2\|f_{2,k}\|_{H^{1}}^2\leq C\|f_{1,m}\|_{L^2}^2\| f_2\|_{H^1}^2,
\end{align*}
from which and \eqref{f2}, we infer that
\begin{align*}
\|f_1f_2\|_{L^2}=&\big\|\sum\limits_{m\in\mathbb{Z}}e^{2\pi imz}f_{1,m}(y)f_2\big\|_{L^2}\leq \sum\limits_{m\in\mathbb{Z}}\|e^{2\pi imz}f_{1,m}(y)f_2\|_{L^2}\\
\leq& C\sum\limits_{m\in\mathbb{Z}}\|f_{1,m}\|_{L^2}\| f_2\|_{H^1}\leq C\big(\|f_1\|_{L^{2}}+\|\partial_zf_1\|_{L^{2}}\big)\|f_2\|_{H^{1}},
\end{align*}
which gives the third inequality.
\end{proof}

\begin{Lemma}\label{lem9a}
If $\partial_x f_1=0,$ then it holds that
  \begin{align*}
  &\|f_1 f_2\|_{L^2}\leq C\| f_1\|_{H^1}\big(\|f_2\|_{L^2}+\|\partial_zf_2\|_{L^2}\big),\\
&\|\nabla(f_1 f_2)\|_{L^2}\leq C\big(\|f_1\|_{H^1}+\|\partial_zf_1\|_{H^1}\big)\| f_2\|_{H^1},\\
 &\|f_1 f_2\|_{H^2}\leq C\| f_1\|_{H^1}\big(\|f_2\|_{H^2}+\|\partial_zf_2\|_{H^2}\big)+C\| f_1\|_{H^3}\big(\|f_2\|_{L^2}+\|\partial_zf_2\|_{L^2}\big),\\
 &\|f_1 f_2\|_{H^3}\leq C\|f_1\|_{H^1}\big(\|f_2\|_{H^3}+\|\partial_zf_2\|_{H^3}\big)+C\| f_1\|_{H^3}\big(\|f_2\|_{H^1}+\|\partial_zf_2\|_{H^1}\big).
\end{align*}
\end{Lemma}

\begin{proof}
As $\partial_x f_1=0$, we may write $f_1(x,y,z)=f_1(y,z),\ f_2(x,y,z)=\sum\limits_{k\in\mathbb{Z}}e^{2\pi ikx}f_{2,k}(y,z)$. Then we have\begin{align*}
\|f_1f_2\|_{L^2}^2=&\sum\limits_{k\in\mathbb{Z}}\|f_1f_{2,k}\|_{L^2}^2,\ \ \|f_2\|_{L^2}^2=\sum\limits_{k\in\mathbb{Z}}\|f_{2,k}\|_{L^2}^2,\ \ \|\partial_zf_2\|_{L^2}^2=\sum\limits_{k\in\mathbb{Z}}\|\partial_zf_{2,k}\|_{L^2}^2,
\end{align*}
and then
\begin{align*}
\|f_1f_2\|_{L^2}^2=&\sum\limits_{k\in\mathbb{Z}}\|f_1f_{2,k}\|_{L^2}^2\leq \sum\limits_{k\in\mathbb{Z}}C\|f_1\|_{H^{1}}^2\big(\|f_{2,k}\|_{L^2}^2+\|\partial_zf_{2,k}\|_{L^2}^2\big)
\\=&C\|f_1\|_{H^{1}}^2\big(\|f_{2}\|_{L^2}^2+\|\partial_zf_{2}\|_{L^2}^2\big),
\end{align*}
which gives
\begin{align}\label{f4}
\|f_1f_2\|_{L^{2}}\leq C\|f_1\|_{H^{1}}\big(\|f_{2}\|_{L^2}+\|\partial_zf_{2}\|_{L^2}\big).
\end{align}
We also have
\begin{align}\label{f3}
\|f_1f_2\|_{L^{2}}\leq\|f_1\|_{L^{\infty}}\|f_2\|_{L^{2}}\leq C\big(\|f_1\|_{H^{1}}+\|\partial_zf_1\|_{H^{1}}\big)\|f_2\|_{L^{2}},
\end{align}

Notice that\begin{align*}
\|\partial_z^jf_2\|_{H^1}^2=\sum\limits_{k\in\mathbb{Z}}\big(\|\partial_z^jf_{2,k}\|_{H^1}^2+(2\pi k)^2\|\partial_z^jf_{2,k}\|_{L^2}^2\big),\quad j=0,1.
\end{align*}
Then we infer from Lemma \ref{lem9a1}  that
\begin{align*}
\|f_1f_2\|_{L^2}^2=&\sum\limits_{k\in\mathbb{Z}}\|f_1f_{2,k}\|_{L^2}^2\leq \sum\limits_{k\in\mathbb{Z}}C\|f_1\|_{L^2}^2(\|f_{2,k}\|_{H^{1}}^2+\|\partial_zf_{2,k}\|_{H^{1}}^2)
\\ \leq&C\|f_1\|_{L^2}^2\big(\|f_{2}\|_{H^{1}}^2+\|\partial_zf_{2}\|_{H^{1}}^2\big),
\end{align*}
and
\begin{align*}
\|f_1f_2\|_{L^2}^2=&\sum\limits_{k\in\mathbb{Z}}\|f_1f_{2,k}\|_{L^2}^2\leq \sum\limits_{k\in\mathbb{Z}}C\big(\|f_1\|_{L^2}^2+\|\partial_zf_1\|_{L^2}^2\big)\|f_{2,k}\|_{H^{1}}^2
\\
\leq&C(\|f_1\|_{L^2}^2+\|\partial_zf_1\|_{L^2}^2)\|f_{2}\|_{H^{1}}^2,
\end{align*}
which gives\begin{align}\label{f5}
\|f_1f_2\|_{L^{2}}\leq C\|f_1\|_{L^2}\big(\|f_{2}\|_{H^{1}}+\|\partial_zf_{2}\|_{H^{1}}\big),\\ \label{f6}\|f_1f_2\|_{L^{2}}\leq C\big(\|f_1\|_{L^2}+\|\partial_zf_{1}\|_{L^2}\big)\|f_{2}\|_{H^{1}} .
\end{align}

By \eqref{f6}, we have
\begin{align*}
&\|(\nabla f_1) f_2\|_{L^2}\leq C\big(\|\nabla f_1\|_{L^2}+\|\partial_z\nabla f_1\|_{L^2}\big)\| f_2\|_{H^1}\leq C\big(\|f_1\|_{H^1}+\|\partial_zf_1\|_{H^1}\big)\| f_2\|_{H^1},
\end{align*}
and by \eqref{f3},
\begin{align*}
&\|f_1\nabla f_2\|_{L^2}\leq C\big(\|f_1\|_{H^{1}}+\|\partial_zf_1\|_{H^{1}}\big)\|{\nabla f_2}\|_{L^{2}},
\end{align*}
which gives the second inequality. Similarly, by \eqref{f5} and \eqref{f4}, we have
\begin{align}\label{f7}
\|\nabla(f_1 f_2)\|_{L^2}\leq C\|f_1\|_{H^{1}}\big(\|f_{2}\|_{H^{1}}+\|\partial_zf_{2}\|_{H^{1}}\big).
\end{align}

Notice that for $k=2,3$,
\begin{align*}
\|f_1 f_2\|_{H^k}&\leq C\|f_1 f_2\|_{L^2}+C\sum_{j=1}^3\|\nabla \partial_j^{k-1}(f_1 f_2)\|_{L^2}\\&\leq C\|f_1 f_2\|_{L^2}+C\sum_{j=1}^3\sum_{l=0}^{k-1}\|\nabla (\partial_j^{l}f_1 \partial_j^{k-1-l}f_2)\|_{L^2}.
\end{align*}
By \eqref{f7} and interpolation, we deduce that for $0\leq l\leq k-1\leq2,$\begin{align*}
\|\nabla (\partial_j^{l}f_1 \partial_j^{k-1-l}f_2)\|_{L^2}\leq& C\|\partial_j^{l}f_1\|_{H^{1}}\big(\|\partial_j^{k-1-l}f_{2}\|_{H^{1}}+\|\partial_j^{k-1-l}\partial_zf_{2}\|_{H^{1}}\big)\\
\leq & C\|f_1\|_{H^{l+1}}\big(\|f_{2}\|_{H^{k-l}}+\|\partial_zf_{2}\|_{H^{k-l}}\big)\\
\leq & C\|f_1\|_{H^{1}}^{1-\frac{l}{2}}\|f_1\|_{H^{3}}^{\frac{l}{2}}\big(\|f_{2}\|_{H^{k}}^{1-\frac{l}{2}}\|f_{2}\|_{H^{k-2}}^{\frac{l}{2}}
+\|\partial_zf_{2}\|_{H^{k}}^{1-\frac{l}{2}}\|\partial_zf_{2}\|_{H^{k-2}}^{\frac{l}{2}}\big)\\
\leq & C\|f_1\|_{H^{1}}\big(\|f_{2}\|_{H^{k}}
+\|\partial_zf_{2}\|_{H^{k}}\big)+C\|f_1\|_{H^{3}}\big(\|f_{2}\|_{H^{k-2}}
+\|\partial_zf_{2}\|_{H^{k-2}}\big),
\end{align*}
and by \eqref{f4},
\begin{align*}
\|f_1 f_2\|_{L^2}\leq C\| f_1\|_{H^1}(\|f_2\|_{L^2}+\|\partial_zf_2\|_{L^2})\leq C\| f_1\|_{H^1}\big(\|f_2\|_{H^{k}}+\|\partial_zf_2\|_{H^{k}}\big).
\end{align*}
Summing up, we conclude the third and fourth inequality.
\end{proof}

\begin{Lemma}\label{lem9b}
It holds that for $j\in\{1,3\}$
\begin{align*}
&\|f_1\partial_jf_2\|_{L^2}\leq C\big(\|\partial_jf_1\|_{L^2}+\|f_1\|_{L^2}\big)\|\Delta f_2\|_{L^2},\\ &\|(f_1f_2)\|_{L^2}+\|\partial_j(f_1f_2)\|_{L^2}\leq C\big(\|\partial_jf_1\|_{L^2}+\|f_1\|_{L^2}\big)\| f_2\|_{H^2}.
\end{align*}
\end{Lemma}

\begin{proof}Let us consider the case of $j=1$. We write
$f_l(x,y,z)=\sum\limits_{k\in\mathbb{Z}}e^{2\pi ikx}f_{l,k}(y,z)$ for $l=1,2.$ Then we have
\beno
&&\|f_1\|_{L^2}^2+\|\partial_xf_1\|_{L^2}^2=\sum\limits_{k\in\mathbb{Z}}(1+(2\pi k)^2)\|f_{1,k}\|_{L^2}^2,\\
&&\|\Delta f_2\|_{L^2}^2=\sum\limits_{k\in\mathbb{Z}}\big((2\pi k)^4\|f_{2,k}\|_{L^2}^2+2(2\pi k)^2\|\nabla f_{2,k}\|_{L^2}^2+\|\Delta f_{2,k}\|_{L^2}^2\big).
\eeno
Therefore,
\begin{align}\label{f8}\Big(\sum\limits_{k\in\mathbb{Z}}\|f_{1,k}\|_{L^2}\Big)^2\leq \sum\limits_{k\in\mathbb{Z}}(1+k^2)\|f_{1,k}\|_{L^2}^2\sum\limits_{k\in\mathbb{Z}}\dfrac{1}{1+k^2}\leq C\big(\|{f_1}\|_{L^2}^2+\|\partial_xf_1\|_{L^2}^2\big).
\end{align}
By 2-D Gagliardo-Nirenberg inequality, we have
\begin{align}\label{f9}k^2\|f_{2,k}\|_{L^{\infty}}^2\leq Ck^2\|f_{2,k}\|_{L^{2}}\|f_{2,k}\|_{H^{2}}\leq C\big(k^4\|f_{2,k}\|_{L^2}^2+\|\Delta f_{2,k}\|_{L^2}^2\big). \end{align}
We write $e^{2\pi imx}f_{1,m}(y,z)\partial_xf_2(x,y,z)=\sum\limits_{k\in\mathbb{Z}}2\pi ike^{2\pi i(k+m)x}f_{1,m}(y,z)f_{2,k}(y,z)$. Then by \eqref{f9}, we get
\begin{align*}
&\|e^{2\pi imx}f_{1,m}(y,z)\partial_jf_2\|_{L^2}^2=\sum\limits_{k\in\mathbb{Z}}\|2\pi ikf_{1,m}f_{2,k}\|_{L^2}^2\leq
\sum\limits_{k\in\mathbb{Z}}(2\pi k)^2\|f_{1,m}\|_{L^2}^2\|f_{2,k}\|_{L^{\infty}}^2\\ &\leq C\sum\limits_{k\in\mathbb{Z}}\|f_{1,m}\|_{L^2}^2(k^4\|f_{2,k}\|_{L^2}^2+\|\Delta f_{2,k}\|_{L^2}^2)\leq C\|f_{1,m}\|_{L^2}^2\|\Delta f_2\|_{L^2}^2,
\end{align*}
from which and \eqref{f8}, we infer that
\begin{align*}
\|f_1\partial_xf_2\|_{L^2}=&\big\|\sum\limits_{m\in\mathbb{Z}}e^{2\pi imx}f_{1,m}(y,z)\partial_xf_2\big\|_{L^2}\leq \sum\limits_{m\in\mathbb{Z}}\|e^{2\pi imx}f_{1,m}(y,z)\partial_jf_2\|_{L^2}\\
\leq& \sum\limits_{l\in\mathbb{Z}} C\|f_{1,m}\|_{L^2}\|\Delta f_2\|_{L^2}\leq C\big(\|\partial_xf_1\|_{L^2}+\|f_1\|_{L^2}\big)\|\Delta f_2\|_{L^2}.
\end{align*}
This proves the first inequality for $j=1$. For the case of $j=3, $ let $ \widetilde{f}_l(x,y,z)={f}_l(z,y,x)$ for $l=1,2$. Then we have
\begin{align*}
\|f_1\partial_zf_2\|_{L^2}&=\|\widetilde{f}_1\partial_x\widetilde{f}_2\|_{L^2}\leq C\big(\|\partial_x\widetilde{f}_1\|_{L^2}+\|\widetilde{f}_1\|_{L^2}\big)\|\Delta \widetilde{f}_2\|_{L^2}\\&=C\big(\|\partial_zf_1\|_{L^2}+\|f_1\|_{L^2}\big)\|\Delta f_2\|_{L^2}.
\end{align*}

Using the first inequality of the lemma, we get
 \begin{align*}
 \|(f_1f_2)\|_{L^2}+\|\partial_j(f_1f_2)\|_{L^2}\leq& \|(f_1f_2)\|_{L^2}+\|\partial_jf_1f_2\|_{L^2}+\|f_1\partial_jf_2\|_{L^2}\\&\leq \big(\|\partial_jf_1\|_{L^2}+\|f_1\|_{L^2}\big)\| f_2\|_{L^{\infty}}+ C\big(\|\partial_jf_1\|_{L^2}+\|f_1\|_{L^2}\big)\|\Delta f_2\|_{L^2}\\&\leq C\big(\|\partial_jf_1\|_{L^2}+\|f_1\|_{L^2}\big)\| f_2\|_{H^2},
\end{align*}
which gives the second inequality.
\end{proof}

\begin{Lemma}\label{lem9c}
It holds that for $j\in\{1,3\}$
\begin{align*}
&\|f_1f_2\|_{L^2}\leq C\big(\|\partial_jf_1\|_{H^1}+\|f_1\|_{H^1}\big)\| f_2\|_{L^2}+C\big(\|\partial_jf_1\|_{L^2}+\|f_1\|_{L^2}\big)\| f_2\|_{H^1},\\ &\|\partial_j(f_1f_2)\|_{L^2}\leq C\big(\|(\partial_jf_1,f_1)\|_{H^1}\|(\partial_jf_2,f_2)\|_{L^2}+\|(\partial_jf_1,f_1)\|_{L^2}\|(\partial_jf_2,f_2)\|_{H^1}).
\end{align*}
If $P_0 f_1=0$, then for $k=1,2,3,$
\begin{align*}
&\|f_1f_2\|_{H^k}\leq C\|\partial_xf_1\|_{H^{k+1}}\| f_2\|_{L^2}+C\|\partial_xf_1\|_{L^2}\| f_2\|_{H^{k+1}}.
\end{align*}

\end{Lemma}\begin{proof}Let us consider the case of $j=1$. We write
$f_l(x,y,z)=\sum\limits_{k\in\mathbb{Z}}e^{2\pi ikx}f_{l,k}(y,z)$ for $l=1,2.$ Then we have
\beno
&&\|f_l\|_{L^2}^2+\|\partial_xf_l\|_{L^2}^2=\sum\limits_{k\in\mathbb{Z}}(1+(2\pi k)^2)\|f_{l,k}\|_{L^2}^2,\\
&&\|f_l\|_{H^1}^2+\|\partial_xf_l\|_{H^1}^2=\sum\limits_{k\in\mathbb{Z}}(1+(2\pi k)^2)\big(\|f_{l,k}\|_{H^1}^2+(2\pi k)^2\|f_{l,k}\|_{L^2}^2\big).\\
\eeno
Therefore, for $s=0,1$ we have
\begin{align*}\Big(\sum\limits_{k\in\mathbb{Z}}\|f_{1,k}\|_{H^s}\Big)^2\leq \sum\limits_{k\in\mathbb{Z}}(1+k^2)\|f_{1,k}\|_{H^s}^2\sum\limits_{k\in\mathbb{Z}}\dfrac{1}{1+k^2}\leq C\big(\|{f_1}\|_{H^s}^2+\|\partial_xf_1\|_{H^s}^2\big).
\end{align*}
By 2-D Gagliardo-Nirenberg inequality, we have $\|f_{l,k}\|_{L^{4}}^2\leq C\|f_{l,k}\|_{L^{2}}\|f_{l,k}\|_{H^{1}}$. Then
\begin{align}\label{f11}
\Big(\sum\limits_{k\in\mathbb{Z}}\|f_{1,k}\|_{L^4}\Big)^2&\leq C\Big(\sum\limits_{k\in\mathbb{Z}}\|f_{1,k}\|_{L^2}\Big)\Big(\sum\limits_{k\in\mathbb{Z}}\|f_{1,k}\|_{H^1}\Big)\\ \nonumber&\leq C\big(\|{f_1}\|_{L^2}+\|\partial_xf_1\|_{L^2}\big)\big(\|{f_1}\|_{H^1}+\|\partial_xf_1\|_{H^1}\big), \end{align}
and
\begin{align}\label{f12}
&\sum\limits_{k\in\mathbb{Z}}\|f_{2,k}\|_{L^4}^2\leq C\Big(\sum\limits_{k\in\mathbb{Z}}\|f_{1,k}\|_{L^2}^2\Big)^{\frac{1}{2}}
\Big(\sum\limits_{k\in\mathbb{Z}}\|f_{1,k}\|_{H^1}^2\Big)^{\frac{1}{2}}\leq C\|{f_2}\|_{L^2}\|{f_2}\|_{H^1},
 \end{align}
We write $e^{2\pi imx}f_{1,m}(y,z)f_2(x,y,z)=\sum\limits_{k\in\mathbb{Z}}e^{2\pi i(k+m)x}f_{1,m}(y,z)f_{2,k}(y,z)$. Then by \eqref{f12}, we get
\begin{align*}
\|e^{2\pi imx}f_{1,m}(y,z)f_2\|_{L^2}^2&=\sum\limits_{k\in\mathbb{Z}}\|f_{1,m}f_{2,k}\|_{L^2}^2\leq
\sum\limits_{k\in\mathbb{Z}}\|f_{1,m}\|_{L^4}^2\|f_{2,k}\|_{L^{4}}^2\\ &\leq C\|f_{1,m}\|_{L^4}^2\|{f_2}\|_{L^2}\|{f_2}\|_{H^1},
\end{align*}
from which and \eqref{f11}, we infer that
\begin{align*}
\|f_1f_2\|_{L^2}=&\big\|\sum\limits_{m\in\mathbb{Z}}e^{2\pi imx}f_{1,m}(y,z)f_2\big\|_{L^2}\leq \sum\limits_{m\in\mathbb{Z}}\|e^{2\pi imx}f_{1,m}(y,z)f_2\|_{L^2}\\
\leq& \sum\limits_{l\in\mathbb{Z}} C\|f_{1,m}\|_{L^4}(\|{f_2}\|_{L^2}\|{f_2}\|_{H^1})^{\frac{1}{2}}\\ \leq & C\big(\|\partial_xf_1\|_{L^2}+\|f_1\|_{L^2}\big)^{\frac{1}{2}}
\big(\|{f_1}\|_{H^1}+\|\partial_xf_1\|_{H^1}\big)^{\frac{1}{2}}(\|{f_2}\|_{L^2}\|{f_2}\|_{H^1})^{\frac{1}{2}}\\ \leq & C\big(\|\partial_xf_1\|_{L^2}+\|f_1\|_{L^2}\big)\|{f_2}\|_{H^1}+
\big(\|{f_1}\|_{H^1}+\|\partial_xf_1\|_{H^1}\big)\|{f_2}\|_{L^2}.
\end{align*}
This proves the first inequality for $j=1$. The case of $j=3$ is similar.  Thus,
\begin{align*}
 &\|\partial_j(f_1f_2)\|_{L^2}\leq \|\partial_jf_1f_2\|_{L^2}+\|f_1\partial_jf_2\|_{L^2}\leq C\|\partial_jf_1 \|_{H^{1}}\|(\partial_jf_2,f_2)\|_{L^2}\\&\quad+C\|\partial_jf_1 \|_{L^2}\|(\partial_jf_2,f_2)\|_{H^1}+ C\|(\partial_jf_1,f_1)\|_{L^2}\|\partial_jf_2 \|_{H^{1}}+C\|(\partial_jf_1,f_1)\|_{H^{1}}\|\partial_jf_2 \|_{L^{2}}\\&\leq C\|(\partial_jf_1,f_1) \|_{H^{1}}\|(\partial_jf_2,f_2)\|_{L^2}+C\|(\partial_jf_1,f_1)\|_{L^2}\|(\partial_jf_2,f_2)\|_{H^1},
\end{align*}
which gives the second inequality.

If $ P_0 f_1=0$ then the first inequality with $j=1$ becomes
\begin{align}\label{f13}
&\|f_1f_2\|_{L^2}\leq C\|\partial_xf_1\|_{H^1}\| f_2\|_{L^2}+C\|\partial_xf_1\|_{L^2}\| f_2\|_{H^1}.
\end{align}
Notice that for $k=1,2,3$
\begin{align*}&\|f_1 f_2\|_{H^k}\leq C\|f_1 f_2\|_{L^2}+C\sum_{j=1}^3\sum_{l=0}^{k}\|\partial_j^{l}f_1 \partial_j^{k-l}f_2\|_{L^2}
\end{align*}
By \eqref{f13} and interpolation, we deduce that for $0\leq l\leq k,$\begin{align*}
\|\partial_j^{l}f_1 \partial_j^{k-l}f_2\|_{L^2}\leq& C\|\partial_x\partial_j^{l}f_1\|_{H^1}\| \partial_j^{k-l}f_2\|_{L^2}+C\|\partial_x\partial_j^{l}f_1\|_{L^2}\|\partial_j^{k-l} f_2\|_{H^1}\\ \leq & C\|\partial_xf_1\|_{H^{l+1}}\|f_{2}\|_{H^{k-l}}+C\|\partial_xf_1\|_{H^{l}}\| f_2\|_{H^{k-l+1}}\\ \leq & C\sum_{l'=l}^{l+1}\|\partial_xf_1\|_{L^{2}}^{\frac{k-l'+1}{k+1}}\|\partial_xf_1\|_{H^{k+1}}^{\frac{l'}{k+1}}
\|f_2\|_{L^{2}}^{\frac{l'}{k+1}}\|f_2\|_{H^{k+1}}^{\frac{k-l'+1}{k+1}}
\\ \leq & C\|\partial_xf_1\|_{H^{k+1}}\|f_2\|_{L^{2}}+C\|\partial_xf_1\|_{L^{2}}\|f_2\|_{H^{k+1}},
\end{align*}
and
\begin{align*}&\|f_1 f_2\|_{L^2}\leq C\|\partial_xf_1\|_{H^{k+1}}\|f_2\|_{L^{2}}+C\|\partial_xf_1\|_{L^{2}}\|f_2\|_{H^{k+1}},
\end{align*}
Summing up, we conclude the third inequality.
\end{proof}

 \begin{Lemma}\label{lem9}
 Assume that $\|\bu^1\|_{H^4}<c_3$ with $c_3$ as in Lemma \ref{lem5}.   If $\partial_x f_1=0,\ P_0 f_2=0,$ then it holds that
\begin{align*}
 &\|f_1f_2\|_{L^2}\leq C\|f_1\|_{H^1}\big(\| f_2\|_{L^2}+\|(\partial_z-\kappa\partial_y) f_2\|_{L^2}\big), \\
&\|\nabla\Delta^{-1}(f_1 f_2)\|_{L^2}\leq C\|f_1\|_{L^2}\big(\| f_2\|_{L^2}+\|(\partial_z-\kappa\partial_y) f_2\|_{L^2}\big),\\
 &\|\nabla(f_1 f_2)\|_{L^2}\leq C\|f_1\|_{H^1}\big(\| f_2\|_{H^1}+\|(\partial_z-\kappa\partial_y)f_2\|_{H^1}\big),
 \end{align*}
 and for $j=2,3$,
 \begin{align*}
 \big\|[\partial_j\Delta^{-1},f_1]\Delta f_2\big\|_{H^1}\leq C\|f_1\|_{H^2}\big(\|\nabla f_2\|_{L^2}+\|(\partial_z-\kappa\partial_y)\nabla f_2\|_{L^2}\big).
\end{align*}
\end{Lemma}

\begin{proof}
Take $F_l(X,Y,Z)$ so that $F_l(x,V(t,y,z),z)=f_l(x,y,z)$.  Using the facts that $(\partial_z-\kappa\partial_y) f_2(x,y,z)=\partial_Z F_2(x,V(t,y,z),z) $ and that for $k=0,1$
\beno
\|F_l\|_{H^k}\sim \|f_l\|_{H^k},\quad \|\na (F_1 F_2)\|_{L^2}\sim \|\na (f_1f_2)\|_{L^2},
\quad \|\pa_ZF_2\|_{H^k}\sim \|(\pa_z-\kappa\pa_y)f_2\|_{H^k},
\eeno
we can deduce the first and third inequalities from \eqref{f4} and \eqref{f7} in Lemma \ref{lem9a}.

For the second inequality, we use the dual method and assume that $f_1,f_2$ are real valued. Let $f_3=\Delta^{-1}(f_1 f_2)$, then $P_0 f_3=0,$ using $\partial_x f_1=0 $, we have
\begin{align*}
\|\nabla f_3\|_{L^2}^2=&-\langle f_3,\Delta f_3\rangle=-\langle \Delta^{-1}(f_1 f_2),\Delta f_3\rangle=-\langle (f_1 f_2), f_3\rangle\\=&-\langle f_1 ,f_2f_3\rangle=-\langle f_1 ,P_0(f_2f_3)\rangle\leq \|f_1\|_{L^2}\|P_0(f_2f_3)\|_{L^2}.
\end{align*}
We write
$f_l(x,y,z)=\sum\limits_{k\in\mathbb{Z}}e^{2\pi ikx}f_{l,k}(y,z)$ for $l=2,3.$ Then we have $f_{l,0}=0$ and
\begin{align*} &\|f_2\|_{L^2}^2=\sum\limits_{k\in\mathbb{Z}}\|f_{2,k}\|_{L^2}^2,\ \ \|(\partial_z-\kappa\partial_y)f_2\|_{L^2}^2=\sum\limits_{k\in\mathbb{Z}}\|(\partial_z-\kappa\partial_y)f_{2,k}\|_{L^2}^2\\ &\|\nabla f_3\|_{L^2}^2=\sum\limits_{k\in\mathbb{Z}}(\|\nabla f_{3,k}\|_{L^2}^2+(2\pi k)^2\|f_{3,k}\|_{L^2}^2)\geq\sum\limits_{k\in\mathbb{Z}}\| f_{3,k}\|_{H^1}^2,\\ 
&P_0(f_2f_3)=\sum\limits_{k\in\mathbb{Z}}f_{2,k}(y,z)f_{3,-k}(y,z).
\end{align*}
By the first inequality, we have 
$$\|f_{2,k}f_{3,-k}\|_{L^2}\leq C\|f_{3,-k}\|_{H^1}\big(\| f_{2,k}\|_{L^2}+\|(\partial_z-\kappa\partial_y) f_{2,k}\|_{L^2}\big),$$ 
which gives
\begin{align*} \|P_0(f_2f_3)\|_{L^2}&\leq\sum\limits_{k\in\mathbb{Z}}\|f_{2,k}f_{3,-k}\|_{L^2}
\leq\sum\limits_{k\in\mathbb{Z}}C\|f_{3,-k}\|_{H^1}\big(\| f_{2,k}\|_{L^2}+\|(\partial_z-\kappa\partial_y) f_{2,k}\|_{L^2}\big)\\ &\leq C\Big(\sum\limits_{k\in\mathbb{Z}}\|f_{3,-k}\|_{H^1}^2\Big)^{\frac{1}{2}}
\Big(\sum\limits_{k\in\mathbb{Z}}\| f_{2,k}\|_{L^2}^2+\|(\partial_z-\kappa\partial_y) f_{2,k}\|_{L^2}^2\Big)^{\frac{1}{2}}\\ &\leq C\|\nabla f_3\|_{L^2}\big(\| f_2\|_{L^2}+\|(\partial_z-\kappa\partial_y) f_2\|_{L^2}\big).
\end{align*}
Then we get
\begin{align*}
&\|\nabla f_3\|_{L^2}^2\leq \|f_1\|_{L^2}\|P_0(f_2f_3)\|_{L^2}\leq C\|f_1\|_{L^2}\|\nabla f_3\|_{L^2}\big(\| f_2\|_{L^2}+\|(\partial_z-\kappa\partial_y) f_2\|_{L^2}\big),
\end{align*}
which implies the second inequality.

Let $f_4=[\partial_j\Delta^{-1},f_1]\Delta f_2$. Then we have $P_0 f_4=0$ and
\begin{align*}
f_4&=\pa_jf_1f_2-2\pa_j\Delta^{-1}\big(\na f_1\cdot\na f_2\big)-\pa_j\Delta^{-1}\big(\Delta f_1f_2\big)\\
&=\pa_jf_1f_2-\pa_j\Delta^{-1}\big(\na f_1\cdot\na f_2\big)-\pa_j\text{div}\Delta^{-1}\big(\na f_1f_2\big)\\
&=\text{div}\Delta^{-1}\nabla(\pa_jf_1f_2)-\pa_j\Delta^{-1}\big(\na f_1\cdot\na f_2\big)-\text{div}\Delta^{-1}\big((\na \pa_jf_1)f_2\big)-\text{div}\Delta^{-1}\big(\na f_1\pa_jf_2\big)\\
&=\text{div}\Delta^{-1}(\pa_jf_1\nabla f_2)-\pa_j\Delta^{-1}\big(\na f_1\cdot\na f_2\big)-\text{div}\Delta^{-1}\big(\na f_1\pa_jf_2\big).\end{align*}
from which and the first inequality, we infer that 
\begin{align*}
\| f_4\|_{H^1}\leq &C\|\nabla f_4\|_{L^2}\leq C\big(\|\nabla f_1\cdot\nabla f_2\|_{L^2}+\|\nabla f_1\partial_j f_2\|_{L^2}+\|\partial_jf_1\nabla f_2\|_{L^2}\big)\\ \leq& C\|\nabla f_1\|_{H^{1}}\big(\|\nabla f_2\|_{L^2}+\|(\partial_z-\kappa\partial_y)\nabla f_2\|_{L^2}\big)\\ \leq& C\| f_1\|_{H^{2}}\big(\|\nabla f_2\|_{L^2}+\|(\partial_z-\kappa\partial_y)\nabla f_2\|_{L^2}\big),
\end{align*}
which gives the fourth inequality.
\end{proof}

\subsection{The velocity estimates in terms of $(u^2, u^3)$}

For the nonzero mode, we only need to estimate $u_{\neq}^2$ and $u_{\neq}^3,$ since $u_{\neq}^1$ can be estimated by using the incompressible condition $\text{div}u_{\neq}=\partial_xu_{\neq}^1+\partial_yu_{\neq}^2+\partial_zu_{\neq}^3=0.$

\begin{Lemma}\label{lemuM3}
It holds that
\beno
&&\|\nabla^k(\partial_x,\partial_z)\partial_xu_{\neq}\|_{L^{2}}\leq C\big(\| \nabla^k(\partial_x^2+\partial_z^2)u_{\neq}^3\|_{L^2}+\| \nabla^k\Delta u_{\neq}^2\|_{L^2}\big)\quad\text{for}\,\, k\ge 0,\\
&&\| \Delta\partial_xu\|_{L^2}+\| \Delta\partial_zu_{\neq}^3\|_{L^2}\leq C\big(\nu^{-\frac{1}{3}}\| \nabla(\partial_x^2+\partial_z^2)u_{\neq}^3\|_{L^2}+\nu^{\frac{1}{3}}\| \nabla\Delta u_{\neq}^3\|_{L^2}+\| \nabla\Delta u_{\neq}^2\|_{L^2}\big),\\
&& \| e^{\frac{5}{2}\nu^{1/3}t}\partial_xu_{\neq}\|_{Y_0}^2+\| e^{\frac{5}{2}\nu^{1/3}t}\partial_zu_{\neq}^3\|_{Y_0}^2\leq CE_3E_5.
\eeno
\end{Lemma}

\begin{proof}
Thanks to $\text{div}u_{\neq}=\partial_xu_{\neq}^1+\partial_yu_{\neq}^2+\partial_zu_{\neq}^3=0,$ we have
\begin{align*}
\|\nabla^k(\partial_x,\partial_z)\partial_xu_{\neq}\|_{L^{2}}\leq&\|\nabla^k(\partial_x,\partial_z)\partial_xu_{\neq}^1\|_{L^{2}}
+\|\nabla^k(\partial_x,\partial_z)\partial_xu_{\neq}^2\|_{L^{2}}+\|\nabla^k(\partial_x,\partial_z)\partial_xu_{\neq}^3\|_{L^{2}}\\
\leq&\|\nabla^k(\partial_x,\partial_z)\partial_yu_{\neq}^2\|_{L^{2}}+\|\nabla^k(\partial_x,\partial_z)\partial_zu_{\neq}^3\|_{L^{2}}
\\&+\|\nabla^k(\partial_x,\partial_z)\partial_xu_{\neq}^2\|_{L^{2}}+\|\nabla^k(\partial_x,\partial_z)\partial_xu_{\neq}^3\|_{L^{2}}\\ \leq&
C\big(\|\nabla^k\nabla^2u_{\neq}^2\|_{L^{2}}+\|\nabla^k\partial_x^2u_{\neq}^3\|_{L^{2}}
+\|\nabla^k\partial_x\partial_zu_{\neq}^3\|_{L^{2}}+\|\nabla^k\partial_z^2u_{\neq}^3\|_{L^{2}}\big)\\ \leq&C\big(\| \nabla^k(\partial_x^2+\partial_z^2)u_{\neq}^3\|_{L^2}+\| \nabla^k\Delta u_{\neq}^2\|_{L^2}\big),
\end{align*}
which gives the first inequality. \smallskip

Using the fact that $\partial_xu=\partial_xu_{\neq} $ and $\text{div}u_{\neq}=0,$ we deduce that
\begin{align*}
&\| \Delta\partial_xu\|_{L^2}+\| \Delta\partial_zu_{\neq}^3\|_{L^2}\leq\| \Delta\partial_xu_{\neq}^1\|_{L^2}+\| \Delta\partial_xu_{\neq}^2\|_{L^2}+\| \Delta\partial_xu_{\neq}^3\|_{L^2}+\| \Delta\partial_zu_{\neq}^3\|_{L^2}\\
&\leq\| \Delta\partial_yu_{\neq}^2\|_{L^2}+\| \Delta\partial_zu_{\neq}^3\|_{L^2}+\| \Delta\partial_xu_{\neq}^2\|_{L^2}+\| \Delta\partial_xu_{\neq}^3\|_{L^2}+\| \Delta\partial_zu_{\neq}^3\|_{L^2}
\\ &\leq C\| \nabla\Delta u_{\neq}^2\|_{L^2}+C\|\Delta (\partial_x,\partial_z)u_{\neq}^3\|_{L^{2}}\\ &\leq
C(\| \nabla\Delta u_{\neq}^2\|_{L^2}+\|\nabla (\partial_x^2+\partial_z^2)u_{\neq}^3\|_{L^{2}}^{\frac{1}{2}}\|\nabla \Delta u_{\neq}^3\|_{L^{2}}^{\frac{1}{2}})\\ &\leq C\big(\nu^{-\frac{1}{3}}\| \nabla(\partial_x^2+\partial_z^2)u_{\neq}^3\|_{L^2}+\nu^{\frac{1}{3}}\| \nabla\Delta u_{\neq}^3\|_{L^2}+\| \nabla\Delta u_{\neq}^2\|_{L^2}\big),
\end{align*}
which gives the second inequality. \smallskip

Similarly, we have
\begin{align*}
&\| \partial_xu_{\neq}\|_{L^2}+\| \partial_zu_{\neq}^3\|_{L^2}\leq
C\big(\| \Delta u_{\neq}^2\|_{L^2}^{\frac{1}{2}}\|  u_{\neq}^2\|_{L^2}^{\frac{1}{2}}+\| (\partial_x^2+\partial_z^2)u_{\neq}^3\|_{L^{2}}^{\frac{1}{2}}\| u_{\neq}^3\|_{L^{2}}^{\frac{1}{2}}\big),
\end{align*}
which gives
\begin{align*}
&\| e^{\frac{5}{2}\nu^{1/3}t}\partial_xu_{\neq}\|_{L^{\infty}L^2}^2+\| e^{\frac{5}{2}\nu^{1/3}t}\partial_zu_{\neq}^3\|_{L^{\infty}L^2}^2\\
&\leq
C\big(\| e^{2\nu^{1/3}t}\Delta u_{\neq}^2\|_{L^{\infty}L^2}\| e^{3\nu^{1/3}t} u_{\neq}^2\|_{L^{\infty}L^2}+\| e^{2\nu^{1/3}t}(\partial_x^2+\partial_z^2)u_{\neq}^3\|_{L^{\infty}L^2}\|e^{3\nu^{1/3}t} u_{\neq}^3\|_{L^{\infty}L^2}\big)\\
&\leq
C\big(\|\Delta u_{\neq}^2\|_{X_2}\|  u_{\neq}^2\|_{X_3}+\| (\partial_x^2+\partial_z^2)u_{\neq}^3\|_{X_2}\| u_{\neq}^3\|_{X_3}\big)\\ &\leq
C\big(E_3\|\partial_x^2  u_{\neq}^2\|_{X_3}+E_3\|\partial_x^2 u_{\neq}^3\|_{X_3}\big)\leq CE_3E_5.
\end{align*}
In the same way, we get
\begin{align*}
\| \nabla\partial_xu_{\neq}\|_{L^2}+\| \nabla\partial_zu_{\neq}^3\|_{L^2}\leq& C\|\nabla^2 u_{\neq}^2\|_{L^2}+C\|\nabla (\partial_x,\partial_z)u_{\neq}^3\|_{L^{2}}\\ \leq&
C\big(\| \nabla\Delta u_{\neq}^2\|_{L^2}^{\frac{1}{2}}\| \nabla u_{\neq}^2\|_{L^2}^{\frac{1}{2}}+\|\nabla (\partial_x^2+\partial_z^2)u_{\neq}^3\|_{L^{2}}^{\frac{1}{2}}\|\nabla u_{\neq}^3\|_{L^{2}}^{\frac{1}{2}}\big),
\end{align*}
which gives
\begin{align*}
&\| e^{\frac{5}{2}\nu^{1/3}t}\nabla\partial_xu_{\neq}\|_{L^{2}L^2}^2+\| e^{\frac{5}{2}\nu^{1/3}t}\nabla\partial_zu_{\neq}^3\|_{L^{2}L^2}^2\leq
C\big(\| e^{2\nu^{1/3}t}\nabla\Delta u_{\neq}^2\|_{L^{2}L^2}\| e^{3\nu^{1/3}t}\nabla u_{\neq}^2\|_{L^{2}L^2}\\&\quad+\| e^{2\nu^{1/3}t}\nabla(\partial_x^2+\partial_z^2)u_{\neq}^3\|_{L^{2}L^2}\|e^{3\nu^{1/3}t}\nabla u_{\neq}^3\|_{L^{2}L^2}\big)\\ &\leq
C\big(\nu^{-\frac{1}{2}}\|\Delta u_{\neq}^2\|_{X_2}\nu^{-\frac{1}{2}}\|  u_{\neq}^2\|_{X_3}+\nu^{-\frac{1}{2}}\| (\partial_x^2+\partial_z^2)u_{\neq}^3\|_{X_2}\nu^{-\frac{1}{2}}\| u_{\neq}^3\|_{X_3}\big)\leq C\nu^{-1}E_3E_5.
\end{align*}
This proves the third inequality.\end{proof}
\smallskip

The following lemma gives the estimates of  the zero mode $(\bu^2,\bu^3)$  in terms of $E_2=\|\Delta \bu^2\|_{Y_0}+\|\bu^3\|_{Y_0}+\|\nabla \bu^3\|_{Y_0}+\|\min(\nu^{\frac{2}{3}}+\nu t,1)^{\frac{1}{2}}\Delta \bu^3\|_{Y_0}$.

\begin{Lemma}\label{lemu2u3M2}
It holds that
\beno
&&\|\bu^2\|_{H^{2}}+\|\nabla \bu^2\|_{H^{1}}+\| \bu^3\|_{H^{1}}+\|\partial_z \bu^3\|_{H^{1}}\leq CE_2,\\
&&\|\bu^j\|_{L^{\infty}L^{\infty}}+\nu^{\frac{1}{2}}\|\nabla \bu^j\|_{L^{2}L^{\infty}}\leq CE_2\quad j\in\{2,3\},\\
&&\|\nabla(\bu^jf)\|_{L^{2}L^{2}}^2+\|\bu^j\nabla f\|_{L^{2}L^{2}}^2\leq C\nu^{-1}E_2^2\|f\|_{Y_0}^2\quad j\in\{2,3\},\\
&&\|\nabla(f\nabla\bu^3)\|_{L^2L^{2}}^2\leq C\nu^{-1}E_2^2\| f\|_{L^{\infty}H^{1}}^2.
\eeno
\end{Lemma}

\begin{proof}
Due to $\partial_y\bu^2+\partial_z \bu^3=0$, $\bu^2 $ has  nonzero frequency in $z$, and hence,
\begin{align*}
&\|\nabla \bu^2\|_{H^{1}}\leq\|\bu^2\|_{H^{2}}\leq C\| \Delta \bu^2\|_{L^2}\leq C\| \Delta \bu^2\|_{Y_0}\leq CE_2,\\
&\|\bu^3\|_{H^{1}}\leq\|\bu^3\|_{L^{2}}+\|\nabla \bu^3\|_{L^{2}}\leq\| \bu^3\|_{Y_0}+\|\nabla \bu^3\|_{Y_0}\leq  CE_2,\\
&\|\partial_z \bu^3\|_{H^{1}}=\|\partial_y \bu^2\|_{H^{1}}\leq\|\bu^2\|_{H^{2}}\leq  CE_2,
\end{align*}
which show the first inequality.

First of all, we have
\beno
\|\bu^2\|_{L^{\infty}}\leq C\|\bu^2\|_{H^{2}}\leq  CE_2,
\eeno
and by Lemma \ref{lem9a}, we have
\beno
\|\bu^3\|_{L^{\infty}}\leq C\big(\|\bu^3\|_{H^{1}}+\|\partial_z \bu^3\|_{H^{1}}\big)\leq CE_2.
\eeno
Similarly, we have
\beno
 &&\|\nabla \bu^2\|_{L^{\infty}}\leq C\|\nabla \bu^2\|_{H^{2}}\leq  C\|\nabla\Delta \bu^2\|_{L^{2}},\\
 &&\|\nabla \bu^3\|_{L^{\infty}}\leq C (\|\nabla \bu^3\|_{H^1}+\|\partial_z\nabla \bu^3\|_{H^1}),
 \eeno and 
 \begin{align} \nonumber \nu\big(\|\nabla \bu^3\|_{L^2H^1}^2+\|\partial_z\nabla \bu^3\|_{L^2H^1}^2\big)&\leq C\nu\big(\|\nabla \bu^3\|_{L^2L^2}^2+\|\nabla^2 \bu^3\|_{L^2L^2}^2+\|\partial_z\Delta \bu^3\|_{L^2L^2}^2\big)\\ \label{u3L2H1} &\leq C\big(\| \bu^3\|_{Y_0}^2+\|\nabla \bu^3\|_{Y_0}^2+\nu\|\partial_y\Delta \bu^2\|_{L^2L^2}^2\big)\\
 &\leq C\big(E_2^2+\|\Delta \bu^2\|_{Y_0}^2\big)\leq CE_2^2,
\end{align}
from which, we infer that
 \begin{align*}
 \nu\big(\|\nabla \bu^2\|_{L^{2}L^{\infty}}^2+\|\nabla \bu^3\|_{L^{2}L^{\infty}}^2)&\leq C\nu\big(\|\nabla \bu^3\|_{L^2H^1}^2+\|\partial_z\nabla \bu^3\|_{L^2H^1}^2+\|\nabla\Delta \bu^2\|_{L^2L^2}^2\big)
 \\ &\leq C(E_2^2+\|\Delta u_0^2\|_{Y_0}^2)\leq CE_2^2,
\end{align*}
which gives the second inequality.

For $j\in\{2,3\},$ we have
\begin{align*}
\|\nabla(\bu^jf)\|_{L^{2}L^{2}}^2+\|\bu^j\nabla f\|_{L^{2}L^{2}}^2&\leq C \big(\|\nabla \bu^j\|_{L^2L^{\infty}}^2\|f\|_{L^{\infty}L^{2}}^2+\|\bu^j\|_{L^{\infty}L^{\infty}}^2\|\nabla f\|_{L^{2}L^{2}}^2\big)\\
&\leq C\big(\nu^{-1}E_2^2\|f\|_{L^{\infty}L^{2}}^2+E_2^2\|\nabla f\|_{L^{2}L^{2}}^2\big)=C\nu^{-1}E_2^2\|f\|_{Y_0}^2,
\end{align*}
which gives the third inequality. By Lemma \ref{lem9a} and \eqref{u3L2H1} , we have\begin{align*}
\nu \|\nabla(f\nabla \bu^3)\|_{L^{2}L^{2}}^2\leq C\nu\big(\|\nabla \bu^3\|_{L^2H^1}^2+\|\partial_z\nabla \bu^3\|_{L^2H^1}^2\big)\| f\|_{L^{\infty}H^{1}}^2\leq CE_2^2\| f\|_{L^{\infty}H^{1}}^2,
\end{align*}
which gives the fourth inequality.
\end{proof}

\begin{Lemma}\label{lem5a}
It holds that
\beno
\|\bu(t)\|_{H^2}\leq E_1\min(\nu t,1).
\eeno
If $E_1< {c_{3}},$ then   $\|\kappa\nabla \bu^3\|_{H^1}\leq CE_2.$
\end{Lemma}

\begin{proof} Obviously, $\|\bu^1(t)\|_{H^2}\leq \|u_0^1(t)\|_{H^4}\leq E_1$
and
\beno
\|\bu^1(t)\|_{H^2}\leq \|\bu^1(1)\|_{H^2}+\int_1^t\|\partial_t\bu^1(s)\|_{H^2}ds\leq E_1\nu+\int_1^tE_1\nu ds=E_1\nu t.
\eeno
This proves the first inequality.

If $E_1<{c_{3}},$ then $\|\bu^1\|_{H^4}\leq E_1<{c_{3}}$. Then by Lemma \ref{lem5}, we have $\|\kappa\|_{H^1}\leq C\|\bu^1\|_{H^2}\leq CE_1\min(\nu t,1)$ and $\|\kappa\|_{H^3}\leq C\|\bu^1\|_{H^4}\leq CE_1$. Thus, $\|\kappa(t)\|_{H^2}\leq\|\kappa(t)\|_{H^1}^{\frac{1}{2}}\|\kappa(t)\|_{H^3}^{\frac{1}{2}}\leq CE_1\min(\nu t,1)^{\frac{1}{2}}.$
Then we infer that
\begin{align*}
\|\kappa\nabla \bu^3\|_{H^1}&\leq C\|\kappa\|_{H^2}\|\nabla \bu^3\|_{H^1}\leq CE_1\min(\nu t,1)^{\frac{1}{2}}(\|\nabla \bu^3\|_{L^2}+\|\Delta \bu^3\|_{L^2})\\&\leq C{E_1}(\|\nabla \bu^3\|_{L^2}+\|\min(\nu^{\frac{2}{3}}+\nu t,1)^{\frac{1}{2}}\Delta \bu^3\|_{L^2})\leq CE_1E_2\leq CE_2,
\end{align*}
which gives the second inequality.
\end{proof}

\section{Nonlinear interactions}

In this section, we study nonlinear interactions between different modes. Recall that the nonlinear terms $g_j(j=2,3)$ and ${G_2}$ are given by
\beno
g_j=(\bu^2\partial_y+\bu^3\partial_z) u^j+u_{\neq}\cdot\nabla u^j+\partial_j\big(p^{(2)}+p^{(3)}+p^{(4)}\big),\quad G_2=(g_{2}+\kappa g_3)_{\neq},
\eeno
where
\begin{align*}
\Delta p^{(2)}=-\partial_i\bu^j\partial_j\bu^i,\quad
\Delta p^{(3)}=-2\partial_{\al}\bu^{\beta}\partial_{\beta}u_{\neq}^{\al},\quad\Delta p^{(4)}=-\partial_iu_{\neq}^j\partial_ju_{\neq}^i.
\end{align*}
We write $g_j=\sum\limits_{k=1}^6g_{j,k}$ and $G_2=\sum\limits_{j=1}^6{G}_{2,j}$, where $G_{2,j}=(g_{2,j}+\kappa g_{3,j})_{\neq}$ and
\beno
 &&g_{j,1}=(\bu^2\partial_y+\bu^3\partial_z) u^j,\quad g_{j,2}=u_{\neq}\cdot\nabla \bu^j,\ g_{j,3}=u_{\neq}\cdot\nabla u_{\neq}^j,\\
 &&g_{j,k+2}=\partial_jp^{(k)}\quad \text{for}\,\, k\in\{2,3,4\}.
 \eeno

As $p_{\neq}^{(2)}=0, $  $(g_{j,4})_{\neq}=\partial_jp_{\neq}^{(2)}=0. $ As $W^2=u_{\neq}^2+\kappa u_{\neq}^3$,  we have
\ben\label{eq:G1}
{G_{2,1}}=(\bu^2\partial_y+\bu^3\partial_z){W^2}-u_{\neq}^3(\bu^2\partial_y+\bu^3\partial_z)\kappa.
\een

We define
\begin{align*}
&E_6=\sum_{j=2}^3\big(\|\partial_x^2u_{\neq}^j\|_{X_3}^2+\|\partial_x(\partial_z-\kappa\partial_y)u_{\neq}^j\|_{X_3}^2\big)
+\|\partial_x\nabla W^2\|_{X_3}^2,\\
&\|f\|_{Y_0^k}^2=\|f\|_{L^{\infty}H^k}^2+\nu\|\nabla f\|_{L^{2}H^k}
^2.
\end{align*}
We will use the following simple fact.
\begin{Lemma}\label{lemY02} $ \|f_1f_2\|_{Y_0^2}\leq C\|f_1\|_{Y_0^2}\|f_2\|_{Y_0^2}.$
\end{Lemma}
\begin{proof}
Using the fact that $\|f_1f_2\|_{H^2}\leq C\|f_1\|_{H^2}\|f_2\|_{H^2},$ and
\beno
\|\nabla(f_1f_2)\|_{H^2}\leq C\|\nabla f_1\|_{H^2}\|f_2\|_{H^2}+C\|f_1\|_{H^2}\|\nabla f_2\|_{H^2},
\eeno
 we deduce that
 \begin{align*}\|f_1f_2\|_{Y_0^2}^2=&\|f_1f_2\|_{L^{\infty}H^2}^2+\nu\|\nabla (f_1f_2)\|_{L^{2}H^2}^2\leq C\|f_1\|_{L^{\infty}H^2}^2\|f_2\|_{L^{\infty}H^2}^2\\&+C\nu\|\nabla f_1\|_{L^{2}H^2}^2\|f_2\|_{L^{\infty}H^2}^2+C\nu\|f_1\|_{L^{\infty}H^2}^2\|\nabla f_2\|_{L^{2}H^2}^2\\ \leq& C(\|f_1\|_{L^{\infty}H^2}^2+\nu\|\nabla f_1\|_{L^{2}H^2}^2)(\|f_2\|_{L^{\infty}H^2}^2+\|\nabla f_2\|_{L^{2}H^2}^2)=C\|f_1\|_{Y_0^2}^2\|f_2\|_{Y_0^2}^2.
\end{align*}
This completes the proof.
\end{proof}

In the sequel, we always assume that
\beno
E_1\le \ve_0, \quad E_2\le \ve_0\nu,\quad  E_3\le \ve_0\nu,
\eeno
where ${\ve_0\in(0,c_4)}$ is a sufficiently small constant independent of $\nu$ and $T$.

\subsection{Interaction between zero mode and nonzero mode}

\begin{Lemma}\label{lemg2g3M6}
It holds that
\begin{align*}
&\|e^{3\nu^{1/3}t} \nabla{G}_{2,1}\|_{L^2L^2}^2+\|e^{3\nu^{1/3}t}  \partial_xg_{3,1}\|_{L^2L^2}^2+\|e^{3\nu^{1/3}t} \Delta p^{(3)}\|_{L^2L^2}^2+\|e^{3\nu^{1/3}t} \nabla g_{2,2}\|_{L^2L^2}^2\\&\quad+\|e^{3\nu^{1/3}t} \nabla g_{3,2}\|_{L^2L^2}^2\leq C\nu^{-1}E_2^2E_6,\\
&\|e^{3\nu^{1/3}t}  \nabla(g_{3,1})_{\neq}\|_{L^2L^2}^2
\leq C\nu^{-\frac{5}{3}}E_2^2\big(E_6+\nu^{\frac{4}{3}}\|\Delta u_{\neq}^3\|_{X_3}^2\big).
\end{align*}
\end{Lemma}\begin{proof}
By \eqref{eq:G1} and  Lemma \ref{lemu2u3M2}, we have
\begin{align*}
\|e^{3\nu^{1/3}t}  \nabla {G}_{2,1}\|_{L^2L^2}^2\leq C\nu^{-1}E_2^2\big(\|e^{3\nu^{1/3}t}\nabla {W}^2\|_{Y_0}^2+\|e^{3\nu^{1/3}t}u_{\neq}^3 \nabla \kappa\|_{Y_0}^2\big),
\end{align*}
We  have
\beno
\|e^{3\nu^{1/3}t}\nabla {W}^2\|_{Y_0}^2\leq \|e^{3\nu^{1/3}t}\partial_x\nabla{W}^2\|_{Y_0}^2\leq \|\partial_x\nabla {W}^2\|_{X_3}^2\leq E_6.
\eeno
By Lemma \ref{lem5}, we have $ \|\nabla \kappa\|_{L^{\infty}}\leq C\| \kappa\|_{H^{3}}\leq C\|\bu^1\|_{H^4}\leq C$. Then we have
\beno
&&\|u_{\neq}^3 \nabla \kappa\|_{L^2}\leq\|u_{\neq}^3 \|_{L^2}\|\nabla \kappa\|_{L^{\infty}}\leq C\|\partial_x^2u_{\neq}^3 \|_{L^2},\\
&&\|\nabla(u_{\neq}^3 \nabla \kappa)\|_{L^2}\leq\|u_{\neq}^3 \nabla \kappa\|_{H^1}\leq C\|u_{\neq}^3 \|_{H^1}\|\nabla \kappa\|_{H^{2}}\leq C\|\nabla\partial_x^2u_{\neq}^3 \|_{L^2},
\eeno
therefore,
\begin{align*}
\|e^{3\nu^{1/3}t}u_{\neq}^3 \nabla \kappa\|_{Y_0}^2\leq C\big(\|e^{3\nu^{1/3}t}\partial_x^2u_{\neq}^3 \|_{L^{\infty}L^2}^2+\nu\|e^{3\nu^{1/3}t}\nabla\partial_x^2u_{\neq}^3 \|_{L^{\infty}L^2}^2\big)\leq C\|\partial_x^2u_{\neq}^3 \|_{X_3}^2\leq CE_6.
\end{align*}
This gives
\begin{align*}
\|e^{3\nu^{1/3}t}  \nabla {G}_{2,1}\|_{L^2L^2}^2\leq C\nu^{-1}E_2^2E_6.
\end{align*}

As $\partial_xg_{3,1}=\big(\bu^2\partial_y+\bu^3\partial_z\big)\partial_x u^3,$ by Lemma \ref{lemu2u3M2}, we have
\begin{align*}
\|e^{3\nu^{1/3}t}  \partial_xg_{3,1}\|_{L^2L^2}^2\leq& C\nu^{-1}E_2^2\|e^{3\nu^{1/3}t}\partial_x u^3\|_{Y_0}^2\leq C\nu^{-1}E_2^2\|\partial_x u^3\|_{X_3}^2\\\leq& C\nu^{-1}E_2^2\|\partial_x^2 u_{\neq}^3\|_{X_3}^2\leq C\nu^{-1}E_2^2E_6.
\end{align*}
Similarly, we have
\begin{align*}
\|e^{3\nu^{1/3}t}  \nabla(g_{3,1})_{\neq}\|_{L^2L^2}^2&\leq C\nu^{-1}E_2^2\|e^{3\nu^{1/3}t}\nabla u_{\neq}^3\|_{Y_0}^2\leq C\nu^{-1}E_2^2\|\nabla u_{\neq}^3\|_{X_3}^2\\
&\leq \nu^{-\frac{5}{3}}E_2^2\big(\|\partial_x^2 u_{\neq}^3\|_{X_3}^2+\nu^{\frac{4}{3}}\|\Delta u_{\neq}^3\|_{X_3}^2\big)\\
&\leq \nu^{-\frac{5}{3}}E_2^2\big(E_6+\nu^{\frac{4}{3}}\|\Delta u_{\neq}^3\|_{X_3}^2\big),
\end{align*}
which gives the second inequality. Here we used the fact that
\beno
\|\nabla f_{\neq}\|_{L^2}^2 \leq \| f_{\neq}\|_{L^2}\|\Delta f_{\neq}\|_{L^2}\leq \| \partial_x^2f_{\neq}\|_{L^2}\|\Delta f_{\neq}\|_{L^2}\leq \nu^{-\frac{2}{3}}\| \partial_x^2f_{\neq}\|_{L^2}^2+\nu^{\frac{2}{3}}\|\Delta f_{\neq}\|_{L^2}^2.
\eeno

For $k\in\{2,3\}$, by Lemma \ref{lem9} and Lemma \ref{lemu2u3M2}, we get
\begin{align*}
\|\nabla(u_{\neq}^{k}\partial_{k} \bu^2)\|_{L^2}&\leq C\|\nabla \bu^2\|_{H^1}(\| u_{\neq}^{k}\|_{H^1}+\|(\partial_z-\kappa\partial_y)u_{\neq}^{k}\|_{H^1})\\&\leq CE_2\big(\| \nabla \partial_x^2u_{\neq}^{k}\|_{L^2}+\|\nabla\partial_x(\partial_z-\kappa\partial_y) u_{\neq}^{k}\|_{L^2}\big),
\end{align*}
which gives
\begin{align*}
\|e^{3\nu^{1/3}t}  \nabla(u_{\neq}^{k}\partial_{k}\bu^2)\|_{L^2L^2}^2\leq& CE_2^2\big(\|e^{3\nu^{1/3}t} \nabla \partial_x^2u_{\neq}^{k}\|_{L^2L^2}^2+\|e^{3\nu^{1/3}t}\nabla\partial_x(\partial_z-\kappa\partial_y) u_{\neq}^{k}\|_{L^2L^2}^2)\\
\leq& C\nu^{-1}E_2^2\big(\| \partial_x^2u_{\neq}^{k}\|_{X_3}^2+\|\partial_x(\partial_z-\kappa\partial_y\big) u_{\neq}^{k}\|_{X_3}^2)\leq C\nu^{-1}E_2^2E_6.
\end{align*}
This shows that
\begin{align*}
\|e^{3\nu^{1/3}t}  \nabla{g}_{2,2}\|_{L^2L^2}^2\leq C\nu^{-1}E_2^2E_6.
\end{align*}

We write
\beno
g_{3,2}=u_{\neq}\cdot\nabla \bu^3=(u_{\neq}^{2}\partial_{y}+u_{\neq}^{3}\partial_{z}) \bu^3={W}^{2}\partial_{y} \bu^3+u_{\neq}^{3}\big(\partial_{z}\bu^3-\kappa\partial_{y}\bu^3\big).
\eeno
By Lemma \ref{lemu2u3M2}, we have
\begin{align*}
\|e^{3\nu^{1/3}t} \nabla({W}^{2}\partial_{y} \bu^3)\|_{L^2L^2}^2&\leq C\nu^{-1}E_2^2\|e^{3\nu^{1/3}t}{W}^{2}\|_{L^{\infty}H^1}^2\leq C\nu^{-1}E_2^2\|e^{3\nu^{1/3}t} \nabla W^{2}\|_{L^{\infty}L^2}^2\\&\leq C\nu^{-1}E_2^2\|e^{3\nu^{1/3}t} \nabla \partial_xW^{2}\|_{L^{\infty}L^2}^2\leq C\nu^{-1}E_2^2\|\nabla \partial_xW^{2}\|_{X_3}^2\leq C\nu^{-1}E_2^2E_6,
\end{align*}
and by Lemma \ref{lem9}, Lemma \ref{lemu2u3M2} and Lemma \ref{lem5a}, we have
\begin{align*}
\|\nabla(u_{\neq}^{3}(\partial_{z}\bu^3-\kappa\partial_{y}\bu^3)\|_{L^2}&\leq C\big(\|\partial_{z}\bu^3-\kappa\partial_{y}\bu^3\|_{H^1}\big)\big(\| u_{\neq}^{3}\|_{H^1}+\|(\partial_z-\kappa\partial_y) u_{\neq}^{3}\|_{H^1}\big)\\&\leq C{E_2}(\| \nabla \partial_x^2u_{\neq}^{3}\|_{L^2}+\|\nabla\partial_x(\partial_z-\kappa\partial_y)u_{\neq}^{3}\|_{L^2}).
\end{align*}
Similar to the estimate of $g_{2,2}$, we have
\begin{align*}
&\|e^{3\nu^{1/3}t}  \nabla(u_{\neq}^{3}(\partial_{z}\bu^3-\kappa\partial_{y}\bu^3)\|_{L^2L^2}^2\leq  C\nu^{-1}E_2^2E_6.
\end{align*}
This shows that
\begin{align*}
\|e^{3\nu^{1/3}t}  \nabla{g}_{3,2}\|_{L^2L^2}^2\leq C\nu^{-1}E_2^2E_6.
\end{align*}

Using $\partial_y\bu^2+\partial_z\bu^3=0$, we may write
\beno
\Delta p^{(3)}=-2\partial_{\al}\bu^{\beta}\partial_{\beta}u_{\neq}^{\al}=-2\partial_{\beta}(\partial_{\al}\bu^{\beta}u_{\neq}^{\al})
=-2(\partial_yg_{2,2}+\partial_zg_{3,2}),
\eeno
therefore,
\begin{align*}&\|e^{3\nu^{1/3}t} \Delta p^{(3)}\|_{L^2L^2}^2\leq C\|e^{3\nu^{1/3}t}  \nabla{g}_{2,2}\|_{L^2L^2}^2+C\|e^{3\nu^{1/3}t}  \nabla{g}_{3,2}\|_{L^2L^2}^2\leq C\nu^{-1}E_2^2E_6.
\end{align*}

Summing up, we conclude the first inequality.
\end{proof}

As $g_{j,5}=\partial_jp^{(3)},$ we also have
\begin{align}\label{gj5}&\|e^{3\nu^{1/3}t}  \nabla{g}_{j,5}\|_{L^2L^2}^2\leq C\|e^{3\nu^{1/3}t} \Delta p^{(3)}\|_{L^2L^2}^2\leq C\nu^{-1}E_2^2E_6, \ \ j\in\{2,3\}.
\end{align}

\begin{Lemma}\label{lemM3M5}
It holds that
\begin{align*}
&\|e^{2\nu^{1/3}t} \nabla (g_{2,1})_{\neq}\|_{L^2L^2}^2+\|e^{2\nu^{1/3}t} \partial_z (g_{3,1})_{\neq}\|_{L^2L^2}^2\leq C\nu^{-1}E_2^2E_3^2,\\  &\|e^{2\nu^{1/3}t} \nabla (\bu^1\partial_xu_{\neq}^2)\|_{L^2L^2}^2+\|e^{2\nu^{1/3}t} \Delta p^{(1)}\|_{L^2L^2}^2+\|e^{2\nu^{1/3}t} (\partial_x,\partial_z) (\bu^1\partial_xu_{\neq}^3)\|_{L^2L^2}^2\leq C\nu E_1^2E_3E_5.
\end{align*}
\end{Lemma}

\begin{proof}
As $ (g_{2,1})_{\neq}=\big(\bu^2\partial_y+\bu^3\partial_z\big)u_{\neq}^2,$ by Lemma \ref{lemu2u3M2}, we have
\begin{align*}\|e^{2\nu^{1/3}t}  \nabla(g_{2,1})_{\neq}\|_{L^2L^2}^2&\leq C\nu^{-1}E_2^2\|e^{2\nu^{1/3}t}\nabla u_{\neq}^2\|_{Y_0}^2\leq C\nu^{-1}M_2^2\|\nabla u_{\neq}^2\|_{X_2}^2\\&\leq C\nu^{-1}E_2^2\|\Delta u_{\neq}^2\|_{X_2}^2\leq C\nu^{-1}E_2^2E_3^2.
\end{align*}

As $ (g_{3,1})_{\neq}=\bu^{\al}\partial_{\al}u_{\neq}^3,$ we have $ \partial_z(g_{3,1})_{\neq}=\partial_z\bu^{\al}\partial_{\al}u_{\neq}^3+\bu^{\al}\partial_{\al}\partial_zu_{\neq}^3$. Then by Lemma \ref{lem9a} and Lemma \ref{lemu2u3M2}, we have
\begin{align*}
\|\partial_z(g_{3,1})_{\neq}\|_{L^2}\leq& C\big(\|\partial_z\bu^{2}\|_{H^1}+\|\partial_z\bu^{3}\|_{H^1}\big)\big(\|\nabla u_{\neq}^3\|_{L^2}+\|\nabla \partial_zu_{\neq}^{3}\|_{L^2}\big)\\&+\big(\|\bu^{2}\|_{L^{\infty}}+\|\bu^{3}\|_{L^{\infty}}\big)\|\nabla \partial_zu_{\neq}^{3}\|_{L^2}\\ \leq & CE_2\big(\|\nabla u_{\neq}^3\|_{L^2}+\|\nabla \partial_zu_{\neq}^{3}\|_{L^2}\big)\leq CE_2\|\nabla (\partial_x^2+\partial_z^2)u_{\neq}^3\|_{L^2},
\end{align*}
which gives
\begin{align*}
\|e^{2\nu^{1/3}t} \partial_z (g_{3,1})_{\neq}\|_{L^2L^2}^2&\leq CE_2^2\|e^{2\nu^{1/3}t}\nabla (\partial_x^2+\partial_z^2)u_{\neq}^3\|_{L^2L^2}^2\\&\leq CE_2^2\nu^{-1}\| (\partial_x^2+\partial_z^2)u_{\neq}^3\|_{X_2}^2\leq C\nu^{-1}E_2^2E_3^2.
\end{align*}
Summing up, we conclude the first inequality.

By Lemma \ref{lem5a}, we have
\begin{align*} \|\nabla (\bu^1\partial_xu_{\neq}^2)\|_{L^2}&\leq \|(\bu^1\partial_xu_{\neq}^2)\|_{H^1}\leq C\|\bu^1\|_{H^2}\|\partial_xu_{\neq}^2\|_{H^1}\leq C\nu tE_1\|\nabla\partial_xu_{\neq}^2\|_{L^2}\\ &\leq C\nu t{E_1}\|\partial_x^2 u_{\neq}^2\|_{L^2}^{\frac{1}{2}}\|\Delta u_{\neq}^2\|_{L^2}^{\frac{1}{2}}.
\end{align*}
By Lemma \ref{lem9b} and Lemma \ref{lem5a}, we have
\begin{align*}
\|(\partial_x,\partial_z) (\bu^1\partial_xu_{\neq}^3)\|_{L^2}&\leq  C\|\bu^1\|_{H^2}\big(\|\partial_xu_{\neq}^3\|_{L^2}+\|\partial_x^2u_{\neq}^3\|_{L^2}+\|\partial_z\partial_xu_{\neq}^3\|_{L^2}\big)\\
&\leq C\nu tE_1\|(\partial_x,\partial_z)\partial_xu_{\neq}^3\|_{L^2}\leq C\nu tE_1\|\partial_x^2 u_{\neq}^3\|_{L^2}^{\frac{1}{2}}\|(\partial_x^2+\partial_z^2) u_{\neq}^3\|_{L^2}^{\frac{1}{2}}.
\end{align*}
As $\Delta p^{(1)}=-2\big(\partial_y\bu^1\partial_xu_{\neq}^2+\partial_z\bu^1\partial_xu_{\neq}^3\big),$ by Lemma \ref{lem9b} and Lemma \ref{lem5a}, we have
\begin{align*}
\|\Delta p^{(1)}\|_{L^2}&\leq  C\|\partial_y\bu^1\|_{H^1}\|\partial_xu_{\neq}^2\|_{H^1}+C\|\bu^1\|_{H^2}\big(\|\partial_xu_{\neq}^3\|_{L^2}
+\|\partial_z\partial_xu_{\neq}^3\|_{L^2}\big)\\
&\leq  C\|\bu^1\|_{H^2}\big(\|\nabla\partial_xu_{\neq}^2\|_{L^2}
+\|(\partial_x,\partial_z)\partial_xu_{\neq}^3\|_{L^2}\big)\\
&\leq  C\nu tE_1\big(\|\partial_x^2 u_{\neq}^2\|_{L^2}^{\frac{1}{2}}\|\Delta u_{\neq}^2\|_{L^2}^{\frac{1}{2}}+\|\partial_x^2 u_{\neq}^3\|_{L^2}^{\frac{1}{2}}\|(\partial_x^2+\partial_z^2) u_{\neq}^3\|_{L^2}^{\frac{1}{2}}\big).
\end{align*}
Therefore, using $(\nu t)^2\leq C\nu^{\frac{4}{3}}e^{\nu^{1/3}t}$, we deduce  that
\begin{align*}
&\|e^{2\nu^{1/3}t} \nabla (\bu^1\partial_xu_{\neq}^2)\|_{L^2L^2}^2+\|e^{2\nu^{1/3}t} \Delta p^{(1)}\|_{L^2L^2}^2+\|e^{2\nu^{1/3}t} (\partial_x,\partial_z) (\bu^1\partial_xu_{\neq}^3)\|_{L^2L^2}^2\\
&\leq C E_1^2\big(\|(\nu t)^2e^{2\nu^{1/3}t} \partial_x^2 u_{\neq}^2\|_{L^2L^2}\|e^{2\nu^{1/3}t} \Delta u_{\neq}^2\|_{L^2L^2}\\&\quad+\|(\nu t)^2e^{2\nu^{1/3}t} \partial_x^2 u_{\neq}^3\|_{L^2L^2}\|e^{2\nu^{1/3}t} (\partial_x^2+\partial_z^2) u_{\neq}^3\|_{L^2L^2}\big)\\
&\leq C E_1^2\big(\|\nu^{\frac{4}{3}}e^{3\nu^{1/3}t} \partial_x^2 u_{\neq}^2\|_{L^2L^2}\nu^{-\frac{1}{6}}\|\Delta u_{\neq}^2\|_{X_2}\\&\quad+\|\nu^{\frac{4}{3}}e^{3\nu^{1/3}t} \partial_x^2 u_{\neq}^3\|_{L^2L^2}\nu^{-\frac{1}{6}}\| (\partial_x^2+\partial_z^2) u_{\neq}^3\|_{X_2}\big)\\ &\leq C E_1^2\big(\nu^{-\frac{1}{6}}\|\nu^{\frac{4}{3}} \partial_x^2 u_{\neq}^2\|_{X_3}\nu^{-\frac{1}{6}}M_3+\nu^{-\frac{1}{6}}\|\nu^{\frac{4}{3}} \partial_x^2 u_{\neq}^3\|_{X_3}\nu^{-\frac{1}{6}}E_3\big)\\
&\leq C E_1^2\nu^{-\frac{1}{6}}(\nu^{\frac{4}{3}}E_5)\nu^{-\frac{1}{6}}E_3=C\nu E_1^2E_3E_5,
\end{align*}
which gives the second inequality.
\end{proof}

\begin{Lemma}\label{lemM4M3M1}
It holds that \begin{align*}
&\|e^{2\nu^{1/3}t}  (u_{\neq}\cdot\nabla\bu)\|_{Y_0^2}+\|e^{2\nu^{1/3}t}  (\bu\cdot\nabla u_{\neq})\|_{Y_0^2}\\&\qquad\leq C\big(E_4E_1\nu+E_4E_2+E_1E_3+E_1E_5+E_3E_2/\nu\big).
\end{align*}
\end{Lemma}

\begin{proof}
By Lemma \ref{lemY02}, we have
\begin{align}\label{u2Y02}
\|e^{2\nu^{1/3}t} u_{\neq}^2\partial_y\bu\|_{Y_0^2}  \leq& C\|e^{2\nu^{1/3}t} u_{\neq}^2\|_{Y_0^2}\|\partial_y\bu\|_{Y_0^2}\leq C\|e^{2\nu^{1/3}t}\Delta u_{\neq}^2\|_{Y_0}\|\bu\|_{Y_0^4}\\ \nonumber \leq&C\| \Delta u_{\neq}^2\|_{X_2}\|\bu\|_{Y_0^4}\leq CE_3E_1.
\end{align}
Using Lemma \ref{lem9a}, Lemma \ref{lem5a} and $\|f\|_{H^2}\leq \|f\|_{H^3}^{\frac{2}{3}}\|f\|_{L^2}^{\frac{1}{3}}$,
we get
 \begin{align*}\| u_{\neq}^3\partial_z\bu\|_{H^2}&\leq C\|\partial_z\bu\|_{H^1}\|(\partial_x,\partial_z) u_{\neq}^3\|_{H^2}+C\|\partial_z\bu\|_{H^3}\|(\partial_x,\partial_z) u_{\neq}^3\|_{L^2}\\&\leq CE_1\nu t\|(\partial_x,\partial_z) u_{\neq}^3\|_{H^2}+CE_1\|(\partial_x,\partial_z) u_{\neq}^3\|_{L^2}\\&\leq CE_1\nu\|(\partial_x,\partial_z) u_{\neq}^3\|_{H^3}+CE_1(1+\nu t^3)\|(\partial_x,\partial_z) u_{\neq}^3\|_{L^2},
\end{align*}
and
\begin{align*}\|\nabla( u_{\neq}^3\partial_z\bu)\|_{H^2}&\leq C\|\partial_z\bu\|_{H^1}\|(\partial_x,\partial_z) u_{\neq}^3\|_{H^3}+C\|\partial_z\bu\|_{H^3}\|(\partial_x,\partial_z) u_{\neq}^3\|_{H^1}\\&\leq C{E_1}\nu\|(\partial_x,\partial_z) u_{\neq}^3\|_{H^4}+C{E_1}(1+\nu t^3)\|(\partial_x,\partial_z) u_{\neq}^3\|_{H^1},
\end{align*}
from which, Lemma \ref{lemuM3} and $1+\nu t^3\leq Ce^{\frac{1}{2}\nu^{1/3}t}$, we infer that
\begin{align*}
\|e^{2\nu^{1/3}t} u_{\neq}^3\partial_y\bu\|_{Y_0^2}^2  \leq& C(E_1\nu)^2\big(\|e^{2\nu^{1/3}t}(\partial_x,\partial_z)u_{\neq}^3\|_{L^{\infty}H^3}^2+\nu\|e^{2\nu^{1/3}t} \nabla(\partial_x,\partial_z)u_{\neq}^3\|_{L^2H^3}^2\big)\\&+CE_1^2\|e^{2\nu^{1/3}t}(1+\nu t^3)(\partial_x,\partial_z) u_{\neq}^3\|_{Y_0}^2\\
\leq& C(E_1\nu)^2E_4^2+CE_1^2\|e^{\frac{5}{2}\nu^{1/3}t}(\partial_x,\partial_z) u_{\neq}^3\|_{Y_0}^2\\
\leq& C(E_1\nu)^2E_4^2+CE_1^2E_3E_5,
\end{align*}
which gives
\begin{align}\label{u3Y02}
&\|e^{2\nu^{1/3}t} u_{\neq}^3\partial_z\bu\|_{Y_0^2}\leq C\nu E_1 E_4+CE_1(E_3+E_5).
\end{align}

By Lemma \ref{lem5a}, we have
\begin{align*}
\|\bu^1\partial_xu_{\neq}\|_{H^2}&\leq C\|\bu^1\|_{H^2}\|\partial_xu_{\neq}\|_{H^2}\leq CE_1\nu t\|\partial_xu_{\neq}\|_{H^2}\\
&\leq CE_1\nu\|\partial_x u_{\neq}\|_{H^3}+CE_1\nu t^3\|\partial_x u_{\neq}^3\|_{L^2},
\end{align*}
and
\begin{align*}\|\nabla(\bu^1\partial_xu_{\neq})\|_{H^2}&\leq\|\bu^1\partial_xu_{\neq}\|_{H^3} \leq C\|\bu^1\|_{H^4}\|\partial_xu_{\neq}\|_{H^1}+C\|\bu^1\|_{H^2}\|\partial_xu_{\neq}\|_{H^3}\\&\leq CE_1\|\nabla\partial_xu_{\neq}\|_{L^2}+CE_1\nu t\|\nabla\partial_xu_{\neq}\|_{H^2}\\&\leq CE_1(1+\nu t^3)\|\nabla\partial_xu_{\neq}\|_{L^2}+CE_1\nu\|\nabla\partial_x u_{\neq}\|_{H^3},
\end{align*}
from which,  Lemma \ref{lemuM3} and $(1+\nu t^3)\leq Ce^{\frac{1}{2}\nu^{1/3}t}$, we infer that
\begin{align*}
\| e^{2\nu^{1/3}t}\bu^1\partial_xu_{\neq}\|_{Y_0^2}^2  \leq& C(E_1\nu)^2\big(\|e^{2\nu^{1/3}t}\partial_xu_{\neq}\|_{L^{\infty}H^3}^2+\nu\|e^{2\nu^{1/3}t} \nabla\partial_xu_{\neq}\|_{L^2H^3}^2\big)\\&+CE_1^2\|e^{2\nu^{1/3}t}(1+\nu t^3)\partial_x u_{\neq}\|_{Y_0}^2\\
\leq& C(E_1\nu)^2E_4^2+CE_1^2E_3E_5,
\end{align*}
which gives
\begin{align}\label{u01Y02}
&\|e^{2\nu^{1/3}t}\bu^1\partial_xu_{\neq}\|_{Y_0^2}\leq C\nu E_1 E_4+CE_1(E_3+E_5).
\end{align}

As $\bu^2 $ has nonzero frequency in $z$,
we get by Lemma \ref{lemY02}  that
\begin{align}\label{u02Y02}
\|e^{2\nu^{1/3}t}\bu^2\partial_yu_{\neq}\|_{Y_0^2}  &\leq C\|\bu^2\|_{Y_0^2}\|e^{2\nu^{1/3}t}\partial_yu_{\neq}\|_{Y_0^2}\leq C\|\Delta \bu^2\|_{Y_0}\|e^{2\nu^{1/3}t}u_{\neq}\|_{Y_0^3}\\ \nonumber &\leq CE_2\|e^{2\nu^{1/3}t}\partial_xu_{\neq}\|_{Y_0^3}\leq CE_2E_4.
\end{align}
Using the facts that
\beno
&&\|\bu^3\|_{Y_0^2}\leq C\big(\|\bu^3\|_{Y_0}+\|\Delta \bu^3\|_{Y_0}\big)\leq C\big(E_2+\nu^{-\frac{1}{3}}\|\min(\nu^{\frac{2}{3}}+\nu t,1)^{\frac{1}{2}}\Delta \bu^3\|_{Y_0}\big)\leq CE_2\nu^{-\frac{1}{3}}, \\
&&\|f\|_{Y_0^2}\leq \|f\|_{Y_0^3}^{\frac{2}{3}}\|f\|_{Y_0}^{\frac{1}{3}},\\
&&\|e^{2\nu^{1/3}t}\partial_zu_{\neq}\|_{Y_0}\leq \|e^{2\nu^{1/3}t}\partial_z\partial_xu_{\neq}\|_{Y_0}\leq CE_3 \quad \text{(by Lemma \ref{lemuM3}) },
\eeno
and Lemma \ref{lemY02}, we deduce that

\begin{align}\label{u03Y02}\|e^{2\nu^{1/3}t}\bu^3\partial_zu_{\neq}\|_{Y_0^2}  &\leq C\|\bu^3\|_{Y_0^2}\|e^{2\nu^{1/3}t}\partial_zu_{\neq}\|_{Y_0^2}\\ \nonumber
&\leq CE_2\nu^{-\frac{1}{3}}\|e^{2\nu^{1/3}t}\partial_zu_{\neq}\|_{Y_0}^{\frac{1}{3}}\|e^{2\nu^{1/3}t}\partial_zu_{\neq}\|_{Y_0^3}^{\frac{2}{3}}\\ \nonumber &\leq CE_2\nu^{-\frac{1}{3}}E_3^{\frac{1}{3}}E_4^{\frac{2}{3}}\leq CE_2(\nu^{-1}E_3+E_4).
\end{align}
As $u_{\neq}\cdot\nabla \bu=u_{\neq}^2\partial_y\bu+u_{\neq}^3\partial_z\bu$ and $\bu\cdot\nabla u_{\neq}=\bu^1\partial_xu_{\neq}+\bu^2\partial_yu_{\neq}+\bu^3\partial_zu_{\neq},$ the lemma follows from
 \eqref{u2Y02}-\eqref{u03Y02}.
\end{proof}

\subsection{Interaction between nonzero modes}

\begin{Lemma}\label{lemg2g3M3}
It holds that
\begin{align*}&\|e^{4\nu^{1/3}t} \Delta p^{(4)}\|_{L^2L^2}^2+\|e^{4\nu^{1/3}t} \nabla g_{2,3}\|_{L^2L^2}^2+\nu^{\frac{2}{3}}\|e^{4\nu^{1/3}t} \nabla g_{3,3}\|_{L^2L^2}^2+\|e^{4\nu^{1/3}t} |u_{\neq}|^2\|_{L^2L^2}^2\\&+\|e^{4\nu^{1/3}t}  u_{\neq}\cdot\nabla u_{\neq}\|_{L^2L^2}^2+\|e^{4\nu^{1/3}t} (\partial_x,\partial_z)g_{3,3}\|_{L^2L^2}^2\leq C\nu^{-1}E_3^4.
\end{align*}
\end{Lemma}

\begin{proof}We introduce
\beno
A_0=\| (\partial_x^2+\partial_z^2)u_{\neq}^3\|_{L^2}+\| \Delta u_{\neq}^2\|_{L^2},\quad A_1=\| \nabla(\partial_x^2+\partial_z^2)u_{\neq}^3\|_{L^2}+\| \nabla\Delta u_{\neq}^2\|_{L^2},
\eeno
 and let $A=A_0A_1$.

 By Lemma \ref{lem9c}, we have
 \begin{align*}&\| |u_{\neq}|^2\|_{L^2}\leq C(\| \partial_xu_{\neq}\|_{H^1}+\| u_{\neq}\|_{H^1})(\| \partial_xu_{\neq}\|_{L^2}+\| u_{\neq}\|_{L^2})\leq C\| \nabla\partial_x^2u_{\neq}\|_{L^2}\| \partial_x^2u_{\neq}\|_{L^2},
\end{align*}
and for $k\in\{1,3\},$
\begin{align*}&\| u_{\neq}^k\partial_k u_{\neq}\|_{L^2}+\| \partial_x(u_{\neq}^k\partial_k u_{\neq})\|_{L^2}\leq C\| (\partial_xu_{\neq}^k,u_{\neq}^k)\|_{H^1}\| (\partial_x\partial_ku_{\neq},\partial_ku_{\neq})\|_{L^2}\\&\quad+C\| (\partial_xu_{\neq}^k,u_{\neq}^k)\|_{L^2}\| (\partial_x\partial_ku_{\neq},\partial_ku_{\neq})\|_{H^1}\leq C\| \nabla(\partial_x,\partial_z)\partial_xu_{\neq}\|_{L^2}\| (\partial_x,\partial_z)\partial_xu_{\neq}\|_{L^2},
\end{align*}
and
\begin{align*}
&\| \partial_z(u_{\neq}^k\partial_k u_{\neq}^3)\|_{L^2}\leq C\| (\partial_zu_{\neq}^k,u_{\neq}^k)\|_{H^1}\| (\partial_z\partial_ku_{\neq}^3,\partial_ku_{\neq}^3)\|_{L^2}\\&\quad+C\| (\partial_zu_{\neq}^k,u_{\neq}^k)\|_{L^2}\| (\partial_z\partial_ku_{\neq}^3,\partial_ku_{\neq}^3)\|_{H^1}\\& \leq C\| \nabla(\partial_x,\partial_z)\partial_xu_{\neq}\|_{L^2}\| (\partial_x^2+\partial_z^2)u_{\neq}^3\|_{L^2}+C\| (\partial_x,\partial_z)\partial_xu_{\neq}\|_{L^2}\| \nabla(\partial_x^2+\partial_z^2)u_{\neq}^3\|_{L^2}.
\end{align*}
For $k=2,$ by Lemma \ref{lem9b}, we have
\begin{align*}
\| u_{\neq}^k\partial_k u_{\neq}\|_{L^2}+\| (\partial_x,\partial_z)(u_{\neq}^k\partial_k u_{\neq})\|_{L^2}&\leq C\| (\partial_x\partial_ku_{\neq},\partial_z\partial_ku_{\neq},\partial_ku_{\neq})\|_{L^2}\| u_{\neq}^k\|_{H^2}\\&\leq C\| \nabla(\partial_x,\partial_z)u_{\neq}\|_{L^2}\| \Delta u_{\neq}^k\|_{L^2}\\&\leq C\| \nabla(\partial_x,\partial_z)\partial_xu_{\neq}\|_{L^2}\| \Delta u_{\neq}^2\|_{L^2}.
\end{align*}
Summing up, we get by Lemma \ref{lemuM3}  that
\begin{align*}
&\| |u_{\neq}|^2\|_{L^2}+\| u_{\neq}\cdot\nabla u_{\neq}\|_{L^2}+\| \partial_x(u_{\neq}\cdot\nabla u_{\neq})\|_{L^2}+\| \partial_z(u_{\neq}\cdot\nabla u_{\neq}^3)\|_{L^2}\\ &\leq C\| \nabla(\partial_x,\partial_z)\partial_xu_{\neq}\|_{L^2}(\| (\partial_x,\partial_z)\partial_xu_{\neq}\|_{L^2}+\| (\partial_x^2+\partial_z^2)u_{\neq}^3\|_{L^2}+\| \Delta u_{\neq}^2\|_{L^2})\\&\quad+C\| (\partial_x,\partial_z)\partial_xu_{\neq}\|_{L^2}\| \nabla(\partial_x^2+\partial_z^2)u_{\neq}^3\|_{L^2}\\ &\leq C\big(\| \nabla(\partial_x^2+\partial_z^2)u_{\neq}^3\|_{L^2}+\| \nabla\Delta u_{\neq}^2\|_{L^2}\big)\big(\| (\partial_x^2+\partial_z^2)u_{\neq}^3\|_{L^2}+\| \Delta u_{\neq}^2\|_{L^2}\big)=CA.
\end{align*}

For $k\in\{1,3\},$ by Lemma \ref{lem9b}, we have
\begin{align*}
\| \nabla(u_{\neq}^k\partial_k u_{\neq}^2)\|_{L^2}\leq& \| \nabla u_{\neq}^k\partial_k u_{\neq}^2\|_{L^2}+\|  u_{\neq}^k\partial_k\nabla u_{\neq}^2\|_{L^2}\\ \leq& C\| \nabla (\partial_ku_{\neq}^k,u_{\neq}^k)\|_{L^2}\| \Delta u_{\neq}^2\|_{L^2}+C\| (\partial_ku_{\neq}^k,u_{\neq}^k)\|_{L^2}\| \nabla\Delta u_{\neq}^2\|_{L^2},
\end{align*}
and for $k=2,$
\begin{align*}&\|\nabla( u_{\neq}^k\partial_k u_{\neq}^2)\|_{L^2}\leq \|( u_{\neq}^k\partial_k u_{\neq}^2)\|_{H^1}\leq C\| u_{\neq}^k\|_{H^2}\| \nabla u_{\neq}^2\|_{H^2}\leq C\|\Delta u_{\neq}^2\|_{L^2}\| \nabla \Delta u_{\neq}^2\|_{L^2}.
\end{align*}
Using Lemma \ref{lemuM3} and $\|f_{\neq}\|_{L^2}\leq \|\partial_xf_{\neq}\|_{L^2}$, we infer that
\begin{align*}\| \nabla(u_{\neq}\cdot\nabla u_{\neq}^2)\|_{L^2}\leq& C(\| (\partial_x,\partial_z)u_{\neq}\|_{L^2}+\|\Delta u_{\neq}^2\|_{L^2})\| \nabla \Delta u_{\neq}^2\|_{L^2}\\&+\| \nabla(\partial_x,\partial_z)u_{\neq}\|_{L^2}\|\Delta u_{\neq}^2\|_{L^2}\leq CA.
\end{align*}
This shows that
\begin{align}\label{uu}
&\| |u_{\neq}|^2\|_{L^2}+\| u_{\neq}\cdot\nabla u_{\neq}\|_{L^2}+\| \partial_x(u_{\neq}\cdot\nabla u_{\neq})\|_{L^2}+\| \partial_z(u_{\neq}\cdot\nabla u_{\neq}^3)\|_{L^2}\\ \nonumber&\qquad+\| \nabla(u_{\neq}\cdot\nabla u_{\neq}^2)\|_{L^2}\leq CA.
\end{align}

Notice that $\Delta p^{(4)}=-\partial_iu_{\neq}^j\partial_ju_{\neq}^i=-\partial_{j}(\partial_iu_{\neq}^ju_{\neq}^i)
=-\partial_{j}(u_{\neq}\cdot\nabla u_{\neq}^j) $ due to $\text{div}u_{\neq}=0.$
We know from \eqref{uu} that $\| \partial_k(u_{\neq}\cdot\nabla u_{\neq}^k)\|_{L^2}\leq CA $ for $k\in\{1,2,3\}$, thus $\| \Delta p^{(4)}\|_{L^2}\leq CA. $ Recall that $g_{2,3}=u_{\neq}\cdot\nabla u_{\neq}^2$ and $g_{3,3}=u_{\neq}\cdot\nabla u_{\neq}^3$. Then  by \eqref{uu}, we obtain \begin{align*}
\| |u_{\neq}|^2\|_{L^2}+\| u_{\neq}\cdot\nabla u_{\neq}\|_{L^2}+\| \Delta p^{(4)}\|_{L^2}+\| (\partial_x,\partial_z)g_{3,3}\|_{L^2}+\| \nabla g_{2,3}\|_{L^2}\leq CA.
\end{align*}
Therefore, we deduce that
 \begin{align*}
 &\|e^{4\nu^{1/3}t} \Delta p^{(4)}\|_{L^2L^2}^2+\|e^{4\nu^{1/3}t} \nabla g_{2,3}\|_{L^2L^2}^2+\|e^{4\nu^{1/3}t} |u_{\neq}|^2\|_{L^2L^2}^2\\&\quad+\|e^{4\nu^{1/3}t}  u_{\neq}\cdot\nabla u_{\neq}\|_{L^2L^2}^2+\|e^{4\nu^{1/3}t} (\partial_x,\partial_z)g_{3,3}\|_{L^2L^2}^2\\ &\leq C\|e^{4\nu^{1/3}t}A\|_{L^2(1,T)}^2\leq C\|e^{2\nu^{1/3}t}A_0\|_{L^{\infty}(1,T)}^2\|e^{2\nu^{1/3}t}A_1\|_{L^2(1,T)}^2 \leq C\nu^{-1}E_3^4,
\end{align*}
here we used the fact that
\begin{align*}\|e^{2\nu^{1/3}t}A_0\|_{L^{\infty}(1,T)}&\leq \|e^{2\nu^{1/3}t}(\partial_x^2+\partial_z^2)u_{\neq}^3\|_{L^{\infty}L^2}+\|\Delta u_{\neq}^2\|_{L^{\infty}L^2}\\&\leq \|(\partial_x^2+\partial_z^2)u_{\neq}^3\|_{X_2}+\|\Delta u_{\neq}^2\|_{X_2}\leq E_3,
\end{align*}
and
\begin{align*}\|e^{2\nu^{1/3}t}A_1\|_{L^2(1,T)}&\leq \|e^{2\nu^{1/3}t}\nabla(\partial_x^2+\partial_z^2)u_{\neq}^3\|_{L^{2}L^2}+\|\nabla\Delta u_{\neq}^2\|_{L^{2}L^2}\\&\leq \nu^{-\frac{1}{2}}\|(\partial_x^2+\partial_z^2)u_{\neq}^3\|_{X_2}+\nu^{-\frac{1}{2}}\|\Delta u_{\neq}^2\|_{X_2}\leq \nu^{-\frac{1}{2}}E_3.
\end{align*}

It remains to estimate $\nabla g_{3,3}$($g_{3,3}=u_{\neq}\cdot\nabla u_{\neq}^3=u_{\neq}^j\partial_ju_{\neq}^3$). For $k\in\{1,3\},$ by Lemma \ref{lem9c}, we get
\begin{align*}
\|u_{\neq}^k\partial_k u_{\neq}^3\|_{H^1}\leq& C\| \partial_xu_{\neq}^k\|_{H^2}\|\partial_ku_{\neq}^3\|_{L^2}+C\| \partial_xu_{\neq}^k\|_{L^2}\|\partial_ku_{\neq}^3\|_{H^2}\\ \leq &C(\| \Delta\partial_xu\|_{L^2}+\| \Delta\partial_zu_{\neq}^3\|_{L^2})(\|\partial_x^2u_{\neq}\|_{L^2}+\|(\partial_x^2+\partial_z^2)u_{\neq}^3\|_{L^2}),
\end{align*}
and for  $k=2,$
\begin{align*}&\|u_{\neq}^k\partial_k u_{\neq}^3\|_{H^1}\leq C\| u_{\neq}^k\|_{H^2}\|\partial_ku_{\neq}^3\|_{H^1}\leq C\|\Delta u_{\neq}^2\|_{L^2}\|\Delta u_{\neq}^3\|_{L^2},
\end{align*}
from which and Lemma \ref{lemuM3}, we infer that
\begin{align*}
\| \nabla g_{3,3}\|_{L^2}\leq&\|  u_{\neq}^1\partial_1 u_{\neq}^3\|_{H^1}+\|  u_{\neq}^2\partial_2 u_{\neq}^3\|_{H^1}+\|  u_{\neq}^3\partial_3 u_{\neq}^3\|_{H^1}\\ \leq &C\big(\| \Delta\partial_xu\|_{L^2}+\| \Delta\partial_zu_{\neq}^3\|_{L^2}\big)\big(\|\partial_x^2u_{\neq}\|_{L^2}+\|(\partial_x^2+\partial_z^2)u_{\neq}^3\|_{L^2}+\|\Delta u_{\neq}^2\|_{L^2}\big)\\
\leq &C\big(\nu^{-\frac{1}{3}}\| \nabla(\partial_x^2+\partial_z^2)u_{\neq}^3\|_{L^2}+\nu^{\frac{1}{3}}\| \nabla\Delta u_{\neq}^3\|_{L^2}+\| \nabla\Delta u_{\neq}^2\|_{L^2}\big)A_0.
\end{align*}
Therefore, we obtain
\begin{align*}
\|e^{4\nu^{1/3}t} \nabla g_{3,3}\|_{L^2L^2}^2 \leq& C\Big(\nu^{-\frac{2}{3}}\|e^{2\nu^{1/3}t} \nabla(\partial_x^2+\partial_z^2)u_{\neq}^3\|_{L^2L^2}^2+\nu^{\frac{2}{3}}\|e^{2\nu^{1/3}t} \nabla\Delta u_{\neq}^3\|_{L^2L^2}^2\\&\quad+\|e^{2\nu^{1/3}t} \nabla\Delta u_{\neq}^2\|_{L^2L^2}^2\Big)\|e^{2\nu^{1/3}t}A_0\|_{L^{\infty}(1,T)}^2\\ \leq& C\nu^{-1}\big(\nu^{-\frac{2}{3}}\|(\partial_x^2+\partial_z^2)u_{\neq}^3\|_{X_2}^2+\nu^{\frac{2}{3}}\| \Delta u_{\neq}^3\|_{X_3}^2+\|\Delta u_{\neq}^2\|_{X_2}^2\big)E_3^2\leq C\nu^{-1}\nu^{-\frac{2}{3}}E_3^4,
\end{align*}
which gives\begin{align*}&\nu^{\frac{2}{3}}\|e^{4\nu^{1/3}t} \nabla g_{3,3}\|_{L^2L^2}^2 \leq C\nu^{-1}E_3^4.
\end{align*}This completes the proof.
\end{proof}
As $g_{j,6}=\partial_jp^{(4)},$ we also have
\begin{align}\label{gj6}&\|e^{4\nu^{1/3}t}  \nabla{g}_{j,6}\|_{L^2L^2}^2\leq C\|e^{4\nu^{1/3}t} \Delta p^{(4)}\|_{L^2L^2}^2\leq C\nu^{-1}E_3^4, \ \ j\in\{2,3\}.
\end{align}

\begin{Lemma}\label{lem:HH-Y0}
It holds that
\begin{align*}
\|e^{2\nu^{1/3}t}(u_{\neq}\cdot\nabla u_{\neq})\|_{Y_0^2}\leq CE_4E_3.
\end{align*}
\end{Lemma}

\begin{proof}
For $k\in\{1,3\},$ by Lemma \ref{lem9c}, we have
\begin{align*}\| u_{\neq}^k\partial_k u_{\neq}\|_{H^2}&\leq C\| \partial_xu_{\neq}^k\|_{H^3}\| \partial_ku_{\neq}\|_{L^2}+C\| \partial_xu_{\neq}^k\|_{L^2}\| \partial_ku_{\neq}\|_{H^3}\\&\leq C\| (\partial_x,\partial_z)u_{\neq}\|_{H^3}\| (\partial_x,\partial_z)u_{\neq}\|_{L^2},
\end{align*}
and for $k=2,$
\begin{align*}\| u_{\neq}^k\partial_k u_{\neq}\|_{H^2}&\leq C\| u_{\neq}^k\|_{H^2}\| \partial_ku_{\neq}\|_{H^2}\leq C\| \Delta u_{\neq}^2\|_{L^2}\|u_{\neq}\|_{H^3},
\end{align*}
from which and Lemma \ref{lemuM3}, we infer that
\begin{align*}
\| u_{\neq}\cdot\nabla u_{\neq}\|_{H^2}\leq& C\| (\partial_x,\partial_z)u_{\neq}\|_{H^3}\big(\| \partial_x(\partial_x,\partial_z)u_{\neq}\|_{L^2}+\| \Delta u_{\neq}^2\|_{L^2}\big)\\
\leq &C\| (\partial_x,\partial_z)u_{\neq}\|_{H^3}\big(\|(\partial_x^2+\partial_z^2)u_{\neq}^3\|_{L^2}+\| \Delta u_{\neq}^2\|_{L^2}\big).
\end{align*}
Similarly, for $k\in\{1,3\},$ by Lemma \ref{lem9c}, we have
\begin{align*}\| \nabla(u_{\neq}^k\partial_k u_{\neq})\|_{H^2}\leq\|u_{\neq}^k\partial_k u_{\neq}\|_{H^3}\leq C\| (\partial_x,\partial_z)u_{\neq}\|_{H^4}\| (\partial_x,\partial_z)u_{\neq}\|_{L^2},
\end{align*}
and for $k=2,$
\begin{align*}
\|\nabla( u_{\neq}^k\partial_k u_{\neq})\|_{H^2}&\leq \|(\nabla u_{\neq}^k)\partial_k u_{\neq}\|_{H^2}+\| u_{\neq}^k\partial_k\nabla u_{\neq}\|_{H^2}\\ &\leq C\|\nabla u_{\neq}^k\|_{H^2}\| \partial_ku_{\neq}\|_{H^2}+C\| u_{\neq}^k\|_{H^2}\| \partial_k\nabla u_{\neq}\|_{H^2}\\& \leq C\| \nabla\Delta u_{\neq}^2\|_{L^2}\|u_{\neq}\|_{H^3}+C\| \Delta u_{\neq}^2\|_{L^2}\| u_{\neq}\|_{H^4},
\end{align*}
from which and  Lemma \ref{lemuM3}, we infer that
\begin{align*}&\|\nabla( u_{\neq}\cdot\nabla u_{\neq})\|_{H^2}\leq C\| (\partial_x,\partial_z)u_{\neq}\|_{H^4}\big(\| \partial_x(\partial_x,\partial_z)u_{\neq}\|_{L^2}+\| \Delta u_{\neq}^2\|_{L^2}\big)\\
&\quad+C\| \nabla\Delta u_{\neq}^2\|_{L^2}\|u_{\neq}\|_{H^3}\\
&\leq C\| \nabla(\partial_x,\partial_z)u_{\neq}\|_{H^3}\big(\| (\partial_x^2+\partial_z^2)u_{\neq}^3\|_{L^2}+\| \Delta u_{\neq}^2\|_{L^2}\big)+C\| \nabla\Delta u_{\neq}^2\|_{L^2}\|\partial_xu_{\neq}\|_{H^3}.
\end{align*}
Then we conclude that
\begin{align*}
&\|e^{2\nu^{1/3}t} (u_{\neq}\cdot\nabla u_{\neq})\|_{Y_0^2}^2=\|e^{2\nu^{1/3}t} (u_{\neq}\cdot\nabla u_{\neq})\|_{L^{\infty}H^2}^2+\nu\| e^{2\nu^{1/3}t}\nabla (u_{\neq}\cdot\nabla u_{\neq})\|_{L^{2}H^k}^2\\  &\leq C\big(\|e^{2\nu^{1/3}t}(\partial_x,\partial_z)u_{\neq}\|_{L^{\infty}H^3}^2+\nu\|e^{2\nu^{1/3}t} \nabla(\partial_x,\partial_z)u_{\neq}\|_{L^2H^3}^2\big)\big(\| (\partial_x^2+\partial_z^2)u_{\neq}^3\|_{L^{\infty}L^2}^2\\&\quad+\| \Delta u_{\neq}^2\|_{L^{\infty}L^2}^2\big)+C\nu\| \nabla\Delta u_{\neq}^2\|_{L^2L^2}^2\|e^{2\nu^{1/3}t}\partial_xu_{\neq}\|_{L^{\infty}H^3}^2\\  &\leq CE_4^2\big(\|(\partial_x^2+\partial_z^2)u_{\neq}^3\|_{X_2}^2+\| \Delta u_{\neq}^2\|_{X_2}^2\big)+C\| \Delta u_{\neq}^2\|_{X_2}^2E_4^2\leq CE_4^2E_3^2.
\end{align*}
This proves the lemma.
\end{proof}

\subsection{Interaction between zero modes}

\begin{Lemma}\label{lemM1M2}
It holds that
 \begin{align*}
 &\|\bu\cdot\nabla\bu\|_{Y_0^2}\leq CE_1E_2 .
\end{align*}
\end{Lemma}

\begin{proof}
As $\bu^2 $ has nonzero frequency in $z$,  we get by Lemma \ref{lemY02}  that
 \begin{align}\label{u02Y1}
 \|\bu^2\partial_y\bu\|_{Y_0^2}  \leq C\|\bu^2\|_{Y_0^2}\|\partial_y\bu\|_{Y_0^2}\leq C\|\Delta \bu^2\|_{Y_0}\|\bu\|_{Y_0^4} \leq CE_2E_1.
\end{align}
By Lemma \ref{lem5a}, we have $\|\partial_z\bu\|_{H^2}\leq \|\bu\|_{H^3}\leq \|\bu\|_{H^2}^{\frac{1}{2}} \|\bu\|_{H^4}^{\frac{1}{2}}\leq C{E_1}\min(\nu t,1)^{\frac{1}{2}}$. Then
\begin{align*}
\|\bu^3\partial_z\bu\|_{H^2}&\leq C\|\bu^3\|_{H^2}\|\partial_z\bu\|_{H^2}\leq C\big(\|\bu^3\|_{L^2}+\|\Delta\bu^3\|_{L^2}\big)E_1\min(\nu t,1)^{\frac{1}{2}}\\&\leq C\big(\|\bu^3\|_{L^2}+\|\min(\nu^{\frac{2}{3}}+\nu t,1)^{\frac{1}{2}}\Delta\bu^3\|_{Y_0}\big)E_1\leq CE_2E_1,
\end{align*}
which gives
\begin{align}\label{u03Y1}
\|\bu^3\partial_z\bu\|_{L^{\infty}H^2}&\leq CE_2E_1.
\end{align}

Notice that
\begin{align*}
&\|\nabla(\bu^3\partial_z\bu)\|_{H^2}\leq \|\bu^3\partial_z\bu\|_{H^3}
\leq C\|\bu^3\|_{H^2}\|\partial_z\bu\|_{H^3}+C\|\bu^3\|_{H^3}\|\partial_z\bu\|_{H^2}\\&\leq C\big(\|\bu^3\|_{L^2}+\|\Delta\bu^3\|_{L^2}\big)\|\partial_z\bu\|_{H^3}+C\big(\|\bu^3\|_{L^2}+\|\nabla\Delta\bu^3\|_{L^2}\big)\|\partial_z\bu\|_{H^2}\\
&\leq C\big(\|\bu^3\|_{L^2}\|\partial_z\bu\|_{H^3}+\|\Delta\bu^3\|_{L^2}\|\bu\|_{H^4}+\|\nabla\Delta\bu^3\|_{L^2}\|\partial_z\bu\|_{H^2}\big)\\
&\leq C\big(E_2\|\nabla\bu\|_{H^4}+\|\Delta\bu^3\|_{L^2}E_1+\|\nabla\Delta \bu^3\|_{L^2}E_1\min(\nu t,1)^{\frac{1}{2}}\big),
\end{align*}
which gives
\begin{align}\label{u03Y2}&
\nu\|\nabla(\bu^3\partial_z\bu)\|_{L^2H^2}^2\leq C\big(E_2^2\nu\|\nabla \bu\|_{L^2H^4}^2+\nu\|\Delta\bu^3\|_{L^2L^2}^2E_1^2\\ \nonumber&\quad+\nu\|\min(\nu^{\frac{2}{3}}+\nu t,1)^{\frac{1}{2}}\nabla\Delta \bu^3\|_{L^2L^2}^2{E_1^2}\big)\\ \nonumber &\leq C\big(E_2^2E_1^2+\|\nabla \bu^3\|_{Y_0}^2E_1^2+\|\min(\nu^{\frac{2}{3}}+\nu t,1)^{\frac{1}{2}}\Delta \bu^3\|_{Y_0}^2E_1^2)\leq CE_2^2E_1^2. \end{align}

As $\bu\cdot\nabla\bu=\bu^2\partial_y\bu+\bu^3\partial_z\bu,$
the lemma follows from {\eqref{u02Y1},} \eqref{u03Y1} and \eqref{u03Y2}.
\end{proof}

\section{Growth estimates for the zero mode}

In the sequel, we always assume that
\beno
E_1\le \ve_0, \quad E_2\le \ve_0\nu,\quad  E_3\le \ve_0\nu,
\eeno
where ${\ve_0\in (0,c_4)}$ is a sufficiently small constant independent of $\nu$ and $T$.

\subsection{Estimate of $E_2$}

We write the equation of  $\bu^2, \bu^3$ as
\begin{align*}
(\partial_t-\nu\Delta)\bu^j+\overline{g}_j=0,\ j=2,3.
\end{align*}
Then $L^2$ energy estimate gives
\begin{align*}
&\partial_t\big(\|\bu^2\|_{L^2}^2+\|\bu^3\|_{L^2}^2\big)+2\nu\big(\|\nabla \bu^2\|_{L^2}^2+\|\nabla\bu^3\|_{L^2}^2\big)\\
&=-2\langle\overline{g}_{\al},\bu^{\al}\rangle=-2\sum\limits_{k=1}^6\langle \overline{g}_{\al,k},\bu^{\al}\rangle.
\end{align*}
As $\partial_y\bu^2+\partial_z\bu^3=0$ and $\overline{g}_{j,1}=(\bu^2\partial_y+\bu^3\partial_z)\bu^j,$ we have $\langle \overline{g}_{\al,1},\bu^{\al}\rangle=0. $ As $g_{j,k+2}=\partial_jp^{(k)}$ for $k\in\{2,3,4\},$ we have $-\langle \overline{g}_{\al,k+2},\bu^{\al}\rangle=-\langle\partial_{\al}\overline{p}^{(k)},\bu^{\al}\rangle
=\langle \overline{p}^{(k)},\partial_{\al}\bu^{\al}\rangle=0. $ Moreover, $\overline{g}_{j,2}=0,\ \overline{g}_{j,3}=\overline{u_{\neq}\cdot\nabla u_{\neq}^j},\ \text{div}u_{\neq}=0$. Then we can deduce that  \begin{align*}
&\partial_t\big(\|\bu^2\|_{L^2}^2+\|\bu^3\|_{L^2}^2\big)+2\nu\big(\|\nabla \bu^2\|_{L^2}^2+\|\nabla\bu^3\|_{L^2}^2\big)=-2\langle \overline{g}_{\al,3},\bu^{\al}\rangle\\
&=-2\sum\limits_{k=2}^3\langle \overline{u_{\neq}\cdot\nabla u_{\neq}^k},\bu^{k}\rangle=2\sum\limits_{k=2}^3\langle u_{\neq}\cdot\nabla \bu^k,u_{\neq}^{k}\rangle\leq 2\||u_{\neq}|^2\|_{L^2}\|(\nabla \bu^2,\nabla \bu^3)\|_{L^2},
\end{align*}
which gives
\begin{align*}
&\|\bu^2(t)\|_{L^2}^2+\|\bu^3(t)\|_{L^2}^2+\nu\int_1^t\big(\|\nabla \bu^2(s)\|_{L^2}^2+\|\nabla \bu^3(s)\|_{L^2}^2\big)ds\\ &\leq \|\bu^2(1)\|_{L^2}^2+\|\bu^3(1)\|_{L^2}^2+\nu^{-1}\int_1^t\||u_{\neq}|^2(s)\|_{L^2}^2ds,
\end{align*}
from which and Lemma \ref{lemg2g3M3}, we infer that
\begin{align}\label{u2u3L2}
\|\bu^2\|_{Y_0}^2+\|\bu^3\|_{Y_0}^2\leq C\big(\|u(1)\|_{L^2}^2+\nu^{-1}\||u_{\neq}|^2\|_{L^2L^2}^2\big)\leq C\big(\|u(1)\|_{H^2}^2+\nu^{-2}E_3^4\big).
\end{align}

For $H^1$ estimate, we introduce $\overline{\omega}^1=\partial_y\bu^3-\partial_z\bu^2.$ As $ \partial_y\bu^2+\partial_z\bu^3=0$, we have
$\|\ov\omega^1\|_{L^2}^2=\|\nabla \bu^2\|_{L^2}^2+\|\nabla\bu^3\|_{L^2}^2$  and $ \Delta \bu^2=-\partial_z\ov\omega^1,\ \Delta \bu^3=\partial_y\ov\omega^1.$ And $\ov\omega^1$ satisfies
\beno
 (\partial_t-\nu\Delta) \ov\omega^1+\partial_y \ov g_3-\partial_z \ov g_2=0.
\eeno
 As $\partial_y\bu^2+\partial_z\bu^3=0,\ \ov g_{j,1}=(\bu^2\partial_y+\bu^3\partial_z)\bu^j,$ we have ${\partial_y\ov g_{3,1}-\partial_z\ov g_{2,1}}=(\bu^2\partial_y+\bu^3\partial_z)\ov\omega^1:=h_1. $ As $g_{j,k+2}=\partial_jp^{(k)}$ for $ k=2,3,4,$ we have $\partial_y\ov g_{3,k+2}-\partial_z\ov g_{2,k+2}=\partial_y\partial_z\ov p^{(k)}-\partial_z\partial_y\ov p^{(k)}=0. $ Moreover, $\ov g_{j,2}=0$. Then we have
 \begin{align*}
 &(\partial_t-\nu\Delta)\ov\omega^1+h_1+\partial_y\ov g_{3,3}-\partial_z\ov g_{2,3}=0.
\end{align*}
Energy estimate gives
\begin{align*}
&\partial_t\|\ov\omega^1\|_{L^2}^2+2\nu\|\nabla \ov\omega^1\|_{L^2}^2=-2\langle h_1+\partial_y\ov g_{3,3}-\partial_z\ov g_{2,3},\ov\omega^1\rangle.
\end{align*}
As $\partial_y\bu^2+\partial_z\bu^3=0,\ h_1=(\bu^2\partial_y+\bu^3\partial_z)\ov\omega^1,$ we have $\langle h_1,\ov\omega^1\rangle=0. $ Thus,
\begin{align*} \partial_t\|\ov\omega^1\|_{L^2}^2+2\nu\|\nabla \ov\omega^1\|_{L^2}^2
\leq 2\|(\ov g_{3,3}, \ov g_{2,3})\|_{L^2}\|\nabla \ov\omega^1\|_{L^2}.
\end{align*}
As $\ov g_{j,3}=\ov{u_{\neq}\cdot\nabla u_{\neq}^j}, $ we have \begin{align*}
\partial_t\|\ov\omega^1\|_{L^2}^2+\nu\|\nabla \ov\omega^1\|_{L^2}^2\leq \nu^{-1}\|(\ov g_{3,3}, \ov g_{2,3})\|_{L^2}^2\leq \nu^{-1}\|u_{\neq}\cdot\nabla u_{\neq}\|_{L^2}^2,
\end{align*}
which along with Lemma \ref{lemg2g3M3}  gives
\begin{align}\label{om1L2}
&\|\ov \omega^1\|_{Y_0}^2\leq C\big(\|u(1)\|_{H^1}^2+\nu^{-1}\|u_{\neq}\cdot\nabla u_{\neq}\|_{L^2L^2}^2\big)\leq C\big(\|u(1)\|_{H^2}^2+\nu^{-2}E_3^4\big).
\end{align}
This implies that
\begin{align}\label{u3H1}
&\|\nabla \bu^3\|_{Y_0}^2\leq \|\ov\omega^1\|_{Y_0}^2\leq C\big(\|u(1)\|_{H^2}^2+\nu^{-2}E_3^4\big).
\end{align}

Due to $\Delta \bu^2=-\partial_z\ov \omega^1$, we have
\begin{align*}
&(\partial_t-\nu\Delta) \Delta \bu^2-\partial_zh_1-\partial_y\partial_z\ov g_{3,3}+\partial_z^2 \ov g_{2,3}=0.
\end{align*}
Energy estimate gives
\begin{align*}
&\|\Delta \bu^2\|_{Y_0}^2\leq C\big(\|\Delta \bu^2(1)\|_{L^2}^2+\nu^{-1}\|h_1\|_{L^2L^2}^2+\nu^{-1}\|\partial_z(g_{3,3})\|_{L^2L^2}^2
+\nu^{-1}\|\partial_z(g_{2,3})\|_{L^2L^2}^2\big).
\end{align*}
As $h_1=(\bu^2\partial_y+\bu^3\partial_z)\ov\omega^1,$ we have $\|h_1\|_{L^2}\leq \big(\|\bu^2\|_{L^{\infty}}+\|\bu^3\|_{L^{\infty}}\big)\|\nabla \ov\omega^1\|_{L^2}\leq CE_2\|\nabla \ov\omega^1\|_{L^2},$ and then $\|h_1\|_{L^2L^2}^2\leq CE_2^2\|\nabla\ov\omega^1\|_{L^2L^2}^2\leq CE_2^2\nu^{-1}\|\ov\omega^1\|_{Y_0}^2$ (see Lemma \ref{lemu2u3M2}). By Lemma \ref{lemg2g3M3}, we have
\begin{align*}
\|\partial_z(g_{3,3})\|_{L^2L^2}^2
+\|\partial_z(g_{2,3})\|_{L^2L^2}^2\leq C\nu^{-1}E_3^4.
\end{align*}
Therefore, we get
\begin{align}\label{u2H2}
&\|\Delta\bu^2\|_{Y_0}^2\leq C\big(\|u(1)\|_{H^2}^2+\nu^{-2}E_2^2\| \ov\omega^1\|_{Y_0}^2+\nu^{-2}E_3^4\big).
\end{align}

For $\Delta \bu^3=\partial_y\ov\omega^1$, we use the equation \begin{align*}
&(\partial_t-\nu\Delta) \partial_y\ov\omega^1+\partial_yh_1+\partial_y^2\ov g_{3,3}-\partial_y\partial_z\ov g_{2,3}=0.
\end{align*}
Energy estimate gives
\begin{align*}
&\partial_t\|\partial_y\ov\omega^1\|_{L^2}^2+\nu\|\nabla \partial_y\ov\omega^1\|_{L^2}^2\leq C\nu^{-1}\big(\|h_1\|_{L^2}^2+\|\partial_y(g_{3,3})\|_{L^2}^2+\|\partial_z(g_{2,3})\|_{L^2}^2\big),
\end{align*}
which implies that
\begin{align*} &\partial_t(\min(\nu^{\frac{2}{3}}+\nu t,1)\|\partial_y\ov\omega^1\|_{L^2}^2)+\nu\min(\nu^{\frac{2}{3}}+\nu t,1)\|\nabla \partial_y\ov\omega^1\|_{L^2}^2\\ &\leq C\nu^{-1}\min(\nu^{\frac{2}{3}}+\nu t,1)\big(\|h_1\|_{L^2}^2+\|\partial_y(g_{3,3})\|_{L^2}^2+\|\partial_z(g_{2,3})\|_{L^2}^2\big)+\nu\|\partial_y\ov\omega^1\|_{L^2}^2\\ &\leq C\nu^{-1}\big(\|h_1\|_{L^2}^2+\|\nabla(g_{2,3})\|_{L^2}^2\big)+C\nu^{-\frac{1}{3}}(1+\nu^{\frac{1}{3}}t)\|\nabla(g_{3,3})\|_{L^2}^2
+\nu\|\nabla\ov\omega^1\|_{L^2}^2.
\end{align*}
By Lemma \ref{lemg2g3M3}, we have
\begin{align*} &\|\min(\nu^{\frac{2}{3}}+\nu t,1)^{\frac{1}{2}}\Delta \bu^3\|_{Y_0}^2 \leq C\big(\|\Delta\bu^3(1)\|_{L^2}^2 +\nu^{-1}\|h_1\|_{L^2L^2}^2+\nu^{-1}\|\nabla(g_{2,3})\|_{L^2L^2}^2\\&
\qquad+\nu^{-\frac{1}{3}}\|(1+\nu^{\frac{1}{3}}t)^{\frac{1}{2}}\nabla(g_{3,3})\|_{L^2L^2}^2
+\nu\|\nabla\ov\omega^1\|_{L^2L^2}^2)\\
&\leq C\big(\|u(1)\|_{H^2}^2 +\nu^{-2}E_2^2\|\ov\omega^1\|_{Y_0}^2+\nu^{-1}\|e^{4\nu^{1/3}t}\nabla(g_{2,3})\|_{L^2L^2}^2\\&
\qquad+\nu^{-\frac{1}{3}}\|e^{4\nu^{1/3}t}\nabla(g_{3,3})\|_{L^2L^2}^2
+\|\ov\omega^1\|_{Y_0}^2)\\ &\leq C\big(\|u(1)\|_{H^2}^2 +\nu^{-2}E_2^2\| \ov\omega^1\|_{Y_0}^2+\nu^{-2}E_3^4
+\|\ov \omega^1\|_{Y_0}^2\big),
\end{align*}
Due to $E_2< \underline{\varepsilon_0}\nu$, we have
\begin{align}\label{u3H2} &\|\min(\nu^{\frac{2}{3}}+\nu t,1)^{\frac{1}{2}}\Delta \bu^3\|_{Y_0}^2\leq C\big(\|u(1)\|_{H^2}^2 +\nu^{-2}E_3^4+\|\ov\omega^1\|_{Y_0}^2\big).
\end{align}

It follows from \eqref{u2u3L2}-\eqref{u3H2} that
\begin{align*} &E_2^2\leq C\big(\|u(1)\|_{H^2}^2 +\nu^{-2}E_3^4+\|\ov\omega^1\|_{Y_0}^2\big)\leq C\big(\|u(1)\|_{H^2}^2 +\nu^{-2}E_3^4\big),
\end{align*}
hence,
\begin{align}\label{M2e1}
E_2\leq C\big(\|u(1)\|_{H^2} +\nu^{-1}E_3^2\big).
\end{align}

\subsection{Estimate of $E_1$}
Recall that $\bu^1$ satisfies
\begin{align*}
(\partial_t-\nu\Delta)\bu^1+\bu^2+\ov{(u\cdot\nabla u^1)}=0.
\end{align*}
Energy estimate gives
\begin{align*}
\partial_t\|\bu^1\|_{L^2}^2+2\nu\|\nabla\bu^1\|_{L^2}^2=-2\langle \bu^2,\bu^1\rangle-2\big\langle \ov{(u\cdot\nabla u^1)},\bu^1\big\rangle.
\end{align*}
As $\bu^2 $ has  nonzero frequency in $z$, we get
\begin{align*}
-2\langle\bu^2,\bu^1\rangle\leq 2\|\partial_z\bu^2\|_{L^2}\|\partial_z\bu^1\|_{L^2}\leq 2\|\nabla \Delta\bu^2\|_{L^2}\|\nabla\bu^1\|_{L^2}.
\end{align*}
Using the facts $\text{div}u_{\neq}=0$ and $\partial_y\bu^2+\partial_z\bu^3=0,$ we get
\begin{align*}
-2\big\langle \ov{(u\cdot\nabla u^1)},\bu^1\big\rangle&=-2\big\langle \bu\cdot\nabla \bu^1+\ov{(u_{\neq}\cdot\nabla u_{\neq}^1)},\bu^1\big\rangle\\&=2\langle u_{\neq}\cdot\nabla \bu^1,u_{\neq}\rangle\leq 2\||u_{\neq}|^2\|_{L^2}\|\nabla \bu^1\|_{L^2}.
\end{align*}
Then we can deduce that
\begin{align*}
\partial_t\|\bu^1\|_{L^2}^2+\nu\|\nabla\bu^1\|_{L^2}^2\leq C\nu^{-1}\big(\|\nabla \Delta\bu^2\|_{L^2}^2+\||u_{\neq}|^2\|_{L^2}^2\big),
\end{align*}
which along with Lemma \ref{lemg2g3M3}  gives
\begin{align*}
\|\bu^1\|_{Y_0}^2&\leq C\big(\|\bu^1(1)\|_{L^2}^2+\nu^{-1}\|\nabla \Delta \bu^2\|_{L^2L^2}^2+\nu^{-1}\||u_{\neq}|^2\|_{L^2L^2}^2)\\
&\leq C\big(\|u(1)\|_{H^2}^2+\nu^{-2}E_2^2+\nu^{-2}E_3^4\big).
\end{align*}
This along with \eqref{u2u3L2} gives
\begin{align}\label{u0L2}
\|\bu\|_{Y_0}^2\leq C\big(\|u(1)\|_{H^2}^2+\nu^{-2}E_2^2+\nu^{-2}E_3^4\big).
\end{align}

As $\ov{(y\partial_x u)}=0$ and  $\ov{\Delta p^{L}}=0,$ we get
\begin{align}
&\partial_t\bu-\nu\Delta\bu+\left(\begin{array}{l}\bu^2\\0\\0\end{array}\right)+\ov{\mathbb{P}(u\cdot\nabla u)}=0.\label{NS2}
\end{align}
Due to $\text{div}\bu=0,$ the energy estimate gives
\begin{align*}
\partial_t\|\Delta^2\bu\|_{L^2}^2+2\nu\|\nabla \Delta^2\bu\|_{L^2}^2&=-2\langle \Delta^2\bu^2,\Delta^2\bu^1\rangle-2\langle \Delta^2\ov{(u\cdot\nabla u)},\Delta^2\bu\rangle\\
&\leq2\|\nabla \Delta \bu^2\|_{L^2}\|\nabla \Delta^2\bu^1\|_{L^2}+2\|\nabla \Delta \ov{(u\cdot\nabla u)}\|_{L^2}\|\nabla \Delta^2\bu\|_{L^2},
\end{align*}
which implies that
\begin{align*}
\|\Delta^2\bu\|_{Y_0}^2&\leq C\big(\|\bu(1)\|_{H^4}^2+\nu^{-1}\|\nabla \Delta\bu^2\|_{L^2L^2}^2+\nu^{-1}\|\nabla\Delta\ov{(u\cdot\nabla u)}\|_{L^2L^2}^2)\\
& \leq C\big(\|\bu(1)\|_{H^4}^2+\nu^{-2}E_2^2+\nu^{-2}\|\Delta\ov{(u\cdot\nabla u)}\|_{Y_0}^2\big).
\end{align*}
This along with \eqref{u0L2}  gives
\begin{align}\label{u0H4}
\|\bu\|_{L^{\infty}H^4}^2+\nu\|\nabla\bu\|_{L^{\infty}H^4}^2\leq& C\big(\| \bu\|_{Y_0}^2+\|\Delta^2\bu\|_{Y_0}^2\big)\\ \nonumber
\leq& C\big(\|\bu(1)\|_{H^4}^2+\nu^{-2}E_2^2+\nu^{-2}E_3^4+\nu^{-2}\|\ov{(u\cdot\nabla u)}\|_{Y_0^2}^2\big).
\end{align}

Now we estimate $ \partial_t\bu$. By \eqref{NS2}, we have
\begin{align*}
\|\partial_t\bu\|_{H^2}\leq \nu\|\Delta\bu\|_{H^2}+\|\bu^2\|_{H^2}+\|\ov{\mathbb{P}(u\cdot\nabla u)}\|_{H^2}.
\end{align*}
As $\nu\|\Delta\bu\|_{H^2}\leq \nu\| \bu\|_{H^4},\ \|\ov{\mathbb{P}(u\cdot\nabla u)}\|_{H^2}\leq \|\ov{(u\cdot\nabla u)}\|_{H^2},$ and $\|\bu^2\|_{H^2}\leq C\| \Delta \bu^2\|_{L^2}\leq CE_2, $ we infer that
\begin{align}\label{u0tH2}
&\|\partial_t\bu\|_{L^{\infty}H^2}\leq \nu\|\bu\|_{L^{\infty}H^4}+CE_2+\|\ov{(u\cdot\nabla u)}\|_{L^{\infty}H^2},
\end{align}
It follows from \eqref{u0H4} and \eqref{u0tH2}  that
\begin{align*}
E_1&=\|\bu\|_{L^{\infty}H^4}+
\nu^{\frac{1}{2}}\|\nabla\bu\|_{L^{2}H^4}+\big(\|\partial_t\bu\|_{L^{\infty}H^2}+\|\bu(1)\|_{H^2}\big)/\nu\\
&\leq 2\|\bu\|_{L^{\infty}H^4}+
\nu^{\frac{1}{2}}\|\nabla \bu\|_{L^{2}H^4}+\big(\|\bu(1)\|_{H^2}+CE_2+\|\ov{(u\cdot\nabla u)}\|_{L^{\infty}H^2}\big)/\nu\\
&\leq C\big(\|\bu(1)\|_{H^4}+\|\bu(1)\|_{H^2}/\nu+E_2/\nu+E_3^2/\nu+\|\ov{(u\cdot\nabla u)}\|_{Y_0^2}/\nu\big).
\end{align*}
As $\ov{(u\cdot\nabla u)}=\bu\cdot\nabla\bu+\ov{(u_{\neq}\cdot\nabla u_{\neq})}$, by Lemma \ref{lemM4M3M1} and Lemma \ref{lemM1M2}, we have
\begin{align*}
\|\ov{(u\cdot\nabla u)}\|_{Y_0^2}\leq \|\bu\cdot\nabla \bu\|_{Y_0^2}+\| u_{\neq}\cdot\nabla u_{\neq}\|_{Y_0^2}\leq C\big(E_1E_2+E_3{E_4}\big),
\end{align*}
and then
\begin{align*}
E_1\leq C\big(\|\bu(1)\|_{H^4}+\|\bu(1)\|_{H^2}/\nu+E_2/\nu+E_3^2/\nu+E_1E_2/\nu+E_3E_4/\nu\big).
\end{align*}
Due to $E_2\le \nu\ve_0$,  taking $\ve_0$ small enough such that $ C\ve_0<1/2$, we obtain
\begin{align}
E_1\leq& C\big(\|\bu(1)\|_{H^4}+\|\bu(1)\|_{H^2}/\nu+E_2/\nu+E_3^2/\nu+E_3E_4/\nu\big)\nonumber\\
\leq& C\big(\|\bu(1)\|_{H^4}+\|\bu(1)\|_{H^2}/\nu+E_2/\nu+E_3+E_4\big).
 \end{align}

\section{Decay estimates for the nonzero mode}

In this section, we assume that
\beno
E_1\le \ve_0, \quad E_2\le \ve_0\nu,\quad  E_3\le \ve_0\nu,
\eeno
where ${\ve_0\in(0,c_4)}$ is a sufficiently small constant independent of $\nu$ and $T$.

\subsection{Estimate of $E_5$}

In this part, we need to use the formulation \eqref{eq:W2} and more subtle structure of the system.
Let us first introduce
\begin{align}\label{eq:f12}
&\cL f_1=-\nabla V\cdot\nabla u_{\neq}^3,\quad \cL f_2=-\rho_1\nabla V\cdot\nabla u_{\neq}^3,\quad f_1(1)=f_2(1)=0.
\end{align}
Then we decompose
\beno
W^{2}=W^{2,1}+\nu f_2.
\eeno
Thanks to \eqref{eq:W2} and \eqref{eq:f12},  we find that
\begin{align}
 \nonumber\cL_1W^{2,1}+G_2-(\partial_t \kappa-\nu\Delta \kappa)u_{\neq}^3&=-2\nu\nabla\kappa\cdot\nabla u_{\neq}^3-\nu\cL_1 f_2\\ \nonumber&=-2\nu(\rho_1\nabla V\cdot\nabla u_{\neq}^3+\rho_2(\partial_z-\kappa\partial_y)u_{\neq}^3)-\nu\cL_1 f_2\\
\label{L1W21}&=-2\nu\rho_2(\partial_z-\kappa\partial_y)u_{\neq}^3-\nu f_{2,1},
\end{align}
where $f_{2,1}=\cL_1 f_2+2\rho_1\nabla V\cdot\nabla u_{\neq}^3$, which could be written as
\begin{align} \nonumber
f_{2,1}&=\cL f_2-2(\partial_y+\kappa\partial_z)\Delta^{-1}(\partial_yV\partial_xf_2)+2\rho_1\nabla V\cdot\nabla u_{\neq}^3\\  \nonumber&=-2(\partial_y+\kappa\partial_z)\Delta^{-1}(\partial_yV\partial_xf_2)+\rho_1\nabla V\cdot\nabla u_{\neq}^3\\ \label{f2e0}&=-(\partial_y+\kappa\partial_z)\Delta^{-1}(\partial_yVf_{2,2})
-(\partial_y+\kappa\partial_z)\Delta^{-1}(\partial_yV\rho_1\Delta u_{\neq}^3)+\rho_1\nabla V\cdot\nabla u_{\neq}^3, \end{align}with $f_{2,2}=2\partial_xf_2-\rho_1\Delta u_{\neq}^3=2\partial_x(f_2-\rho_1f_1)-\rho_1(\Delta u_{\neq}^3-2\partial_xf_1).$\smallskip

Using the facts that
\beno
&&\cL \partial_xf_1=\partial_x\cL f_1=-\nabla V\cdot\nabla \partial_xu_{\neq}^3,\\
&&\cL\Delta u_{\neq}^3=\Delta\cL u_{\neq}^3-\Delta V\partial_x u_{\neq}^3-2\nabla V\cdot\nabla\partial_xu_{\neq}^3,
\eeno
we find that
 \begin{align}\label{u3}
\cL(\Delta u_{\neq}^3-2\partial_xf_1)+\Delta V\partial_x u_{\neq}^3&=\Delta\cL u_{\neq}^3=-\Delta(\partial_z p^{L1}+(g_3)_{\neq})\\ \nonumber &=2\partial_z(\partial_yV\partial_xW^2)-\Delta(g_3)_{\neq},
\end{align}
where $p^{L1}=p^{L(1)}+p^{L(2)}$  with
\beno
 \Delta p^{L(1)}=-2\partial_yV\partial_xW^{2,1},\quad  \Delta p^{L(2)}=-2\nu\partial_yV\partial_xf_2.
 \eeno

 To proceed, we need the following lemmas.

\begin{Lemma}\label{lem:f1}
It holds that
\begin{align*}
\|\partial_x^2f_1\|_{X_3}^2+\|\partial_x(\partial_z-\kappa\partial_y)f_1\|_{X_3}^2\leq C\nu^{-\frac{4}{3}}E_6.
\end{align*}
\end{Lemma}

\begin{proof}
It follows from Proposition {\ref{prop:decay-L-1}} that
\begin{align*}
&\|\partial_x^2f_1\|_{X_3}^2+\|\partial_x(\partial_z-\kappa\partial_y)f_1\|_{X_3}^2\leq C\nu^{-\frac{1}{3}}\Big(\|e^{3\nu^{1/3}t} \partial_x^2 (\nabla V\cdot\nabla u_{\neq}^3)\|_{L^2L^2}^2\\&\qquad+\|e^{3\nu^{1/3}t} \partial_x(\partial_z-\kappa\partial_y) (\nabla V\cdot\nabla u_{\neq}^3)\|_{L^2L^2}^2\Big).
\end{align*}
First of all, we have
\beno
 \|\nabla V\|_{L^{\infty}}\leq1+\|\nabla \bu^1\|_{L^{\infty}}\leq 1+C\|\nabla \bu^1\|_{H^{3}}\leq 1+C\|\bu^1\|_{H^{4}}\leq C,
\eeno
and by Lemma \ref{lem5}, we have
\beno
\|(\partial_z-\kappa\partial_y)\nabla V\|_{L^{\infty}}\leq(1+\|\kappa\|_{L^{\infty}})\|\nabla^2 \bu^1\|_{L^{\infty}}\leq C\big(1+\|\kappa\|_{H^3}\big)\|\bu^1\|_{H^{4}}\leq   C.
\eeno
Then we deduce that
\begin{align*}
&\| \partial_x^2 (\nabla V\cdot\nabla u_{\neq}^3)\|_{L^2}=\|  \nabla V\cdot\nabla \partial_x^2u_{\neq}^3\|_{L^2}\leq \|  \nabla V\|_{L^{\infty}}\|\nabla \partial_x^2u_{\neq}^3\|_{L^2}\leq C\|\nabla \partial_x^2u_{\neq}^3\|_{L^2}, \end{align*}
and
\begin{align*}&\| \partial_x(\partial_z-\kappa\partial_y) (\nabla V\cdot\nabla u_{\neq}^3)\|_{L^2}\\&=\| (\partial_z-\kappa\partial_y) (\nabla V)\cdot(\nabla \partial_xu_{\neq}^3)+\nabla V\cdot\partial_x(\partial_z-\kappa\partial_y)\nabla u_{\neq}^3\|_{L^2}\\ &\leq \| (\partial_z-\kappa\partial_y) \nabla V\|_{L^{\infty}}\|\nabla \partial_xu_{\neq}^3\|_{L^2}+\|\nabla V\|_{L^{\infty}}\|\partial_x(\partial_z-\kappa\partial_y)\nabla u_{\neq}^3\|_{L^2}\\ &\leq C\big(\|\nabla \partial_xu_{\neq}^3\|_{L^2}+\|\partial_x(\partial_z-\kappa\partial_y)\nabla u_{\neq}^3\|_{L^2}\big).
\end{align*}
Using the facts that $\partial_x(\partial_z-\kappa\partial_y)\nabla u_{\neq}^3=\nabla \partial_x(\partial_z-\kappa\partial_y)u_{\neq}^3+(\nabla \kappa)\partial_x\partial_yu_{\neq}^3 $ and
\beno
\|(\nabla \kappa)\partial_x\partial_yu_{\neq}^3\|_{L^2}\leq \|\nabla \kappa\|_{L^{\infty}}\|\partial_x\partial_yu_{\neq}^3\|_{L^2} \leq C\| \kappa\|_{H^{3}}\|\partial_y\partial_x^2u_{\neq}^3\|_{L^2}\leq C\|\nabla\partial_x^2u_{\neq}^3\|_{L^2},
\eeno
 we infer that
 \begin{align}\label{u3xz}
 \|\partial_x(\partial_z-\kappa\partial_y)\nabla u_{\neq}^3\|_{L^2}\leq C\big(\|\nabla \partial_x(\partial_z-\kappa\partial_y)u_{\neq}^3\|_{L^2}+\|\nabla\partial_x^2u_{\neq}^3\|_{L^2}\big).
\end{align}

Summing up, we conclude the lemma.
\end{proof}

\begin{Lemma}\label{lem:f2}
It holds that
\begin{align*}
\|\partial_x^2f_2 \|_{X_3}^2\leq C\ve_{0}^2\nu^{-\frac{4}{3}}E_6,\quad
\|\partial_x\nabla f_2 \|_{X_3}^2\leq C\ve_{0}^2\nu^{-2}E_6.
\end{align*}

\end{Lemma}

\begin{proof}

As $\cL \partial_x^2f_2=\partial_x^2\cL f_2=-\rho_1\nabla V\cdot\nabla \partial_x^2u_{\neq}^3,$ we infer from Proposition \ref{prop:decay-L} that \begin{align*}
&\|\partial_x^2f_2 \|_{X_3}^2\leq C\nu^{-\frac{1}{3}}\|e^{3\nu^{1/3}t} \rho_1\nabla V\cdot\nabla \partial_x^2u_{\neq}^3\|_{L^2L^2}^2.
\end{align*}
As $\|\nabla V\|_{L^{\infty}}\leq C$ and $\|\rho_1\|_{L^{\infty}}\leq\|\rho_1\|_{H^{2}}\leq C\| \bu^1\|_{H^{4}}\leq  C\ve_{0} $(by Lemma \ref{lem5}),
we get
\beno
\|\rho_1\nabla V\cdot\nabla \partial_x^2u_{\neq}^3\|_{L^2}\leq \|\rho_1\|_{L^{\infty}}\|\nabla V\|_{L^{\infty}}\|\nabla \partial_x^2u_{\neq}^3\|_{L^2}\leq   C\ve_0\|\nabla \partial_x^2u_{\neq}^3\|_{L^2},
\eeno
which implies that
\begin{align*}
\|\partial_x^2f_2 \|_{X_3}^2\leq C\ve_{0}^2\nu^{-\frac{4}{3}}\| \partial_x^2u_{\neq}^3\|_{X_3}^2\leq C\ve_{0}^2\nu^{-\frac{4}{3}}E_6.
\end{align*}

For $j\in\{1,2,3\}$, we have
\beno
\cL \partial_x\partial_jf_2=\partial_x\partial_j\cL f_2-\partial_jV\partial_x^2f_2=-\partial_j(\rho_1\nabla V\cdot\nabla \partial_xu_{\neq}^3)-\partial_jV\partial_x^2f_2,\quad \partial_x\partial_jf_2(1)=0.
\eeno
Then it follows from Proposition \ref{prop:decay-L} that
\begin{align*}
&\|\partial_x\partial_jf_2 \|_{X_3}^2\leq C\nu^{-\frac{1}{3}}\|e^{3\nu^{1/3}t}\partial_jV\partial_x^2f_2\|_{L^2L^2}^2+C\nu^{-1}\|e^{3\nu^{1/3}t} \rho_1\nabla V\cdot\nabla \partial_xu_{\neq}^3\|_{L^2L^2}^2.
\end{align*}
Using  the facts that  $\|\partial_jV\partial_x^2f_2\|_{L^2}\leq \|\partial_jV\|_{L^{\infty}}\|\partial_x^2f_2\|_{L^2}\leq   C\|\partial_x^2f_2\|_{L^2}$ and $\|\rho_1\nabla V\cdot\nabla \partial_xu_{\neq}^3\|_{L^2}\\ \leq \|\rho_1\|_{L^{\infty}}\|\nabla V\|_{L^{\infty}}\|\nabla \partial_xu_{\neq}^3\|_{L^2}\leq   C\ve_{0}\|\nabla \partial_x^2u_{\neq}^3\|_{L^2}, $
we deduce that
\begin{align*}
&\|e^{3\nu^{1/3}t}\partial_jV\partial_x^2f_2\|_{L^2L^2}^2\leq C\|e^{3\nu^{1/3}t}\partial_x^2f_2\|_{L^2L^2}^2\leq C\nu^{-\frac{1}{3}}\|\partial_x^2f_2\|_{X_3}^2,\\
&\|e^{3\nu^{1/3}t} \rho_1\nabla V\cdot\nabla \partial_xu_{\neq}^3\|_{L^2L^2}^2\leq C\ve_{0}^2\|e^{3\nu^{1/3}t} \nabla \partial_x^2u_{\neq}^3\|_{L^2L^2}^2\leq C{\ve_{0}^2}\nu^{-1}\| \partial_x^2u_{\neq}^3\|_{X_3}^2.
\end{align*}
This shows that
\begin{align*}
&\|\partial_x\nabla f_2 \|_{X_3}^2\leq C\nu^{-\frac{2}{3}}\|\partial_x^2f_2 \|_{X_3}^2+C\ve_{0}^2\nu^{-2}\| \partial_x^2u_{\neq}^3\|_{X_3}^2\leq C\ve_{0}^2\nu^{-2}E_6.
\end{align*}
\end{proof}

\begin{Lemma}\label{lem7}
It holds that
 \begin{align*}
&\|\partial_x(f_2-\rho_1f_1)\|_{X_3}^2\leq C\nu^{-\f23}\ve_0^2E_6.
\end{align*}
\end{Lemma}

\begin{proof}
Using the facts that
\beno
\cL (\rho_1f)-\rho_1\cL f=(\partial_t \rho_1-\nu\Delta \rho_1)f-2\nu\nabla\rho_1\cdot\nabla f =(\partial_t \rho_1+\nu\Delta \rho_1)f-2\nu\text{div}(f\nabla\rho_1 )
\eeno
and $\cL f_2=\rho_1\cL f_1$, we deduce that
\begin{align*}
\cL\partial_x(f_2-\rho_1f_1)&=\partial_x\cL(f_2-\rho_1f_1)=\partial_x(\rho_1\cL f_1-\cL(\rho_1f_1))=-\partial_x\big((\partial_t \rho_1+\nu\Delta \rho_1) f_1-2\nu\text{div}(f_1\nabla\rho_1 )\big)\\&=-\Delta h_1+2\nu\text{div}(\partial_xf_1\nabla\rho_1 )
\end{align*}
where $\Delta h_1=(\partial_t \rho_1+\nu\Delta \rho_1)\partial_xf_1$.
It follows from Proposition \ref{prop:decay-L} that
\begin{align}\label{f1Xa1}
\|\partial_x(f_2-\rho_1f_1)\|_{X_3}^2\leq& C\big(\nu^{-1}\|e^{3\nu^{1/3}t}\nabla h_1\|_{L^2L^2}^2+\nu\|e^{3\nu^{1/3}t} \partial_xf_1\nabla\rho_1\|_{L^2L^2}^2\big).
\end{align}
By Lemma \ref{lem5}, we have
\beno
\|\partial_t \rho_1+\nu\Delta \rho_1\|_{L^{2}}\leq  \|\partial_t \rho_1\|_{L^{2}}+\nu\|\rho_1\|_{H^2}\leq C(\|\partial_t\bu^1\|_{H^2}+\nu\|\bu^1\|_{H^4})\leq C\nu\ve_0
\eeno
and $\|\nabla \rho_1\|_{H^{1}}\leq \|\rho_1\|_{H^2}\leq C\|\bu^1\|_{H^4}\leq C\ve_0. $ Then we infer from Lemma \ref{lem9}  that
\begin{align}\label{f1L2}
\|\partial_xf_1\nabla\rho_1\|_{L^2}&\leq C\|\nabla\rho_1\|_{H^{1}}\big(\|\partial_xf_1\|_{L^2}+\|\partial_x(\partial_z-\kappa\partial_y)f_1\|_{L^2}\big)\\ \nonumber&\leq C\ve_0\big(\|\partial_x^2f_1\|_{L^2}+\|\partial_x(\partial_z-\kappa\partial_y)f_1\|_{L^2}\big)
\end{align}
and
\begin{align}\label{h1H1}
\| \nabla h_1\|_{L^2}&\leq \|\partial_t \rho_1+\nu\Delta \rho_1\|_{L^{2}}\big(\|\partial_xf_1\|_{L^2}+\|\partial_x(\partial_z-\kappa\partial_y)f_1\|_{L^2}\big)\\  \nonumber&\leq C\nu\ve_0\big(\|\partial_x^2f_1\|_{L^2}+\|\partial_x(\partial_z-\kappa\partial_y)f_1\|_{L^2}\big).
\end{align}
By \eqref{f1Xa1}, \eqref{f1L2} and \eqref{h1H1}, we have
\begin{align*}
\|\partial_x(f_2-\rho_1f_1)\|_{X_3}^2\leq& C\nu \ve_0^2\big(\|e^{3\nu^{1/3}t}\partial_x^2f_1\|_{L^2L^2}^2+\|e^{3\nu^{1/3}t} \partial_x(\partial_z-\kappa\partial_y)f_1\|_{L^2L^2}^2\big)\\
\leq &C\nu \ve_0^2\nu^{-\frac{1}{3}}\big(\|\partial_x^2f_1\|_{X_3}^2+\| \partial_x(\partial_z-\kappa\partial_y)f_1\|_{X_3}^2\big).
\end{align*}
This along with Lemma \ref{lem:f1}  gives our result.
\end{proof}
\smallskip

\begin{Proposition}\label{prop2}
It holds that
\beno
&&E_6+\nu^{\frac{4}{3}}\|\Delta u_{\neq}^3\|_{X_3}^2\leq C\big(\|u(1)\|_{H^2}^2+\nu^{-2}E_3^4\big),\\
&&E_5\leq C\big(\|u(1)\|_{H^2}+\nu^{-1}E_3^2\big).
\eeno

\end{Proposition}

\begin{proof}

{\bf Step 1.} \,It follows from Proposition \ref{prop:decay-L}  {and \eqref{u3}} that
\begin{align*}
\|\Delta u_{\neq}^3-2\partial_xf_1\|_{X_3}^2\leq& C\Big(\|\Delta u_{\neq}^3(1)\|_{L^2}^2+\nu^{-\frac{1}{3}}\|e^{3\nu^{1/3}t}\Delta V\partial_x u_{\neq}^3\|_{L^2L^2}^2\\&+\nu^{-\frac{1}{3}}\|e^{3\nu^{1/3}t}\partial_z(\partial_yV\partial_x{W}^2)\|_{L^2L^2}^2
+\nu^{-1}\|e^{3\nu^{1/3}t}\nabla(g_3)_{\neq}\|_{L^2L^2}^2\Big).
\end{align*}
As  $\|\Delta V\|_{L^{\infty}}\leq C\|\bu^1\|_{H^{4}}\leq C$, we have \begin{align*}
&\|e^{3\nu^{1/3}t}\Delta V\partial_x u_{\neq}^3\|_{L^2L^2}^2\leq C\|e^{3\nu^{1/3}t}\partial_x^2 u_{\neq}^3\|_{L^2L^2}^2\leq C\nu^{-\frac{1}{3}}\|\partial_x^2u_{\neq}^3\|_{X_3}^2\leq C\nu^{-\frac{1}{3}}E_6,
\end{align*}
and using
\begin{align*}
\|\partial_z(\partial_yV\partial_x{W^2})\|_{L^2}\leq& {\|\partial_yV\partial_x{W^2}\|_{H^1}=}\|(1+\partial_y\bu^1)\partial_x{W^2}\|_{H^1}\\
\leq& C\big(1+\|\partial_y\bu^1\|_{H^2}\big)\|\partial_x{W^2}\|_{H^1}\leq C\big(1+\|\bu^1\|_{H^4}\big)\|\nabla\partial_x{W^2}\|_{L^2}\leq C\|\nabla\partial_x{W^2}\|_{L^2},
\end{align*}
we infer that
\begin{align*}
&\|e^{3\nu^{1/3}t}\partial_z(\partial_yV\partial_x{W^2})\|_{L^2L^2}^2\leq C\|e^{3\nu^{1/3}t}\nabla\partial_x{W^2}\|_{L^2L^2}^2\leq C\nu^{-\frac{1}{3}}\|\nabla\partial_x{W^2}\|_{X_3}^2\leq C\nu^{-\frac{1}{3}}E_6.
\end{align*}
This shows that
 \begin{align}\label{u3e1}
&\|\Delta u_{\neq}^3-2\partial_xf_1\|_{X_3}^2\leq C\big(\| u(1)\|_{H^2}^2+\nu^{-\frac{2}{3}}E_6
+\nu^{-1}\|e^{3\nu^{1/3}t}\nabla(g_3)_{\neq}\|_{L^2L^2}^2\big).
\end{align}

{\bf Step 2.} It follows from Proposition \ref{prop:decay-L1}  {and \eqref{L1W21}} that
\begin{align}\label{u2e2}
&\|\Delta W^{2,1}\|_{X_3}^2\leq C\Big(\|\Delta {W}^2(1)\|_{L^2}^2+\nu^{-1}\|e^{3\nu^{1/3}t}\nabla G_2\|_{L^2L^2}^2+\nu\|e^{3\nu^{1/3}t}\nabla f_{2,1}\|_{L^2L^2}^2\\ \nonumber&\quad+\nu\|e^{3\nu^{1/3}t}\nabla (\rho_2(\partial_z-\kappa\partial_y)u_{\neq}^3)\|_{L^2L^2}^2+\nu^{-1}\|e^{3\nu^{1/3}t}\nabla ((\partial_t \kappa-\nu\Delta \kappa)u_{\neq}^3)\|_{L^2L^2}^2\Big).
\end{align}
By Lemma \ref{lem5}, we have $\|\rho_2\|_{H^2}\leq C\|\bu^1\|_{H^{4}}\leq C\ve_{0},$ and then
\begin{align}
\nonumber \|\nabla (\rho_2(\partial_z-\kappa\partial_y)u_{\neq}^3)\|_{L^2}&\leq \| \rho_2(\partial_z-\kappa\partial_y)u_{\neq}^3\|_{H^1}\leq C\|\rho_2\|_{H^2}\| (\partial_z-\kappa\partial_y)u_{\neq}^3\|_{H^1}\\ \label{u3e5}&\leq C\ve_{0}\| \nabla(\partial_z-\kappa\partial_y)u_{\neq}^3\|_{L^2}\leq C\ve_{0}\| \nabla\partial_x(\partial_z-\kappa\partial_y)u_{\neq}^3\|_{L^2}.\end{align}
By Lemma \ref{lem5}, we have
\beno
\|\partial_t \kappa-\nu\Delta \kappa\|_{H^1}\leq \|\partial_t \kappa\|_{H^1}+\nu\| \kappa\|_{H^3}\leq C\big(\|\partial_t \bu^1\|_{H^2}+\nu\|\bu^1\|_{H^{4}})\leq C\nu\ve_{0},
\eeno
which along with Lemma \ref{lem9}  gives
\begin{align} \nonumber \|\nabla ((\partial_t \kappa-\nu\Delta \kappa)u_{\neq}^3)\|_{L^2}&\leq  C\|\partial_t \kappa-\nu\Delta \kappa\|_{H^1}(\| u_{\neq}^3\|_{H^1}+\| (\partial_z-\kappa\partial_y)u_{\neq}^3\|_{H^1})\\ \label{u3e6}&\leq  C\nu\ve_{0}\big(\| \nabla\partial_x^2 u_{\neq}^3\|_{L^2}+\| \nabla\partial_x(\partial_z-\kappa\partial_y)u_{\neq}^3\|_{L^2}\big).
\end{align}

Now we estimate $f_{2,1}.$ As $\|\rho_1\|_{L^{\infty}}\leq C\|\rho_1\|_{H^2}\leq C\|\bu^1\|_{H^4}\leq C\ve_{0}$, we have
\begin{align*}
\|f_{2,2}\|_{L^2}&\leq 2\|\partial_x(f_2-\rho_1f_1)\|_{L^2}+\|\rho_1\|_{L^{\infty}}\|\Delta u_{\neq}^3-2\partial_xf_1\|_{L^2}\\ \nonumber&\leq C\big(\|\partial_x(f_2-\rho_1f_1)\|_{L^2}+\ve_{0}\|\Delta u_{\neq}^3-2\partial_xf_1\|_{L^2}\big).
\end{align*}
As $\|\kappa\|_{H^3}\leq C\|\bu^1\|_{H^4}\leq C$ and $\|\partial_yV\|_{L^{\infty}}\leq C$, we get
 \begin{align}\label{f2e2}
 &\|(\partial_y+\kappa\partial_z)\Delta^{-1}(\partial_yVf_{2,2})\|_{H^1} \\ \nonumber&\leq C\big(\|\partial_y\Delta^{-1}(\partial_yVf_{2,2})\|_{H^1}+\|\kappa\|_{H^3}\|\partial_z\Delta^{-1}(\partial_yVf_{2,2})\|_{H^1}\big)\\ \nonumber&\leq C\|\Delta^{-1}(\partial_yVf_{2,2})\|_{H^2}\leq C\|\partial_yVf_{2,2}\|_{L^2}\leq C\|\partial_yV\|_{L^{\infty}}\|f_{2,2}\|_{L^2}\\&\leq C\|f_{2,2}\|_{L^2}\le C\big(\|\partial_x(f_2-\rho_1f_1)\|_{L^2}+\ve_{0}\|\Delta u_{\neq}^3-2\partial_xf_1\|_{L^2}\big).\nonumber
 \end{align}
 Notice that $\nabla V\cdot\nabla u_{\neq}^3=\partial_yV(\partial_y+\kappa\partial_z)u_{\neq}^3$. We write
 \begin{align*}
-(\partial_y+\kappa\partial_z)\Delta^{-1}(\partial_yV\rho_1\Delta u_{\neq}^3)+\rho_1\nabla V\cdot\nabla u_{\neq}^3&=-[(\partial_y+\kappa\partial_z)\Delta^{-1},\partial_yV\rho_1]\Delta u_{\neq}^3\\ &=-[\partial_y\Delta^{-1},\partial_yV\rho_1]\Delta u_{\neq}^3-\kappa[\partial_z\Delta^{-1},\partial_yV\rho_1]\Delta u_{\neq}^3.
\end{align*}
As $\|\kappa\|_{H^3}\leq C$ and $\|\partial_yV\rho_1\|_{H^2}\leq C(1+\|\partial_y\bu^1\|_{H^{2}})\|\rho_1\|_{H^2}\leq C\ve_{0}$, we get by Lemma \ref{lem9}  that
 \begin{align*}
&\|-(\partial_y+\kappa\partial_z)\Delta^{-1}(\partial_yV\rho_1\Delta u_{\neq}^3)+\rho_1\nabla V\cdot\nabla u_{\neq}^3\|_{H^1}\\ &\leq  \|[\partial_y\Delta^{-1},\partial_yV\rho_1]\Delta u_{\neq}^3\|_{H^1}+C\|\kappa\|_{H^3}\|[\partial_z\Delta^{-1},\partial_yV\rho_1]\Delta u_{\neq}^3\|_{H^1}\\ &\leq  C(1+\|\kappa\|_{H^3})\|\partial_yV\rho_1\|_{H^2}\big(\|\nabla u_{\neq}^3\|_{L^2}+\|(\partial_z-\kappa\partial_y)\nabla u_{\neq}^3\|_{L^2})\\ &\leq  C\ve_{0}\big(\|\nabla \partial_x^2u_{\neq}^3\|_{L^2}+\|\partial_x{(\partial_z-\kappa\partial_y)}\nabla u_{\neq}^3\|_{L^2}\big),
\end{align*}
which along with  \eqref{f2e2} {, \eqref{f2e0} and \eqref{u3xz}} gives
\begin{align}\label{f2e4}
\|\nabla f_{2,1}\|_{L^2}\leq& \|f_{2,1}\|_{H^1}\leq C\big(\|\partial_x(f_2-\rho_1f_1)\|_{L^2}+\ve_{0}\|\Delta u_{\neq}^3-2\partial_xf_1\|_{L^2}\big)\\ \nonumber &+C\ve_{0}\big(\|\nabla \partial_x(\partial_z-\kappa\partial_y)u_{\neq}^3\|_{L^2}+\|\nabla\partial_x^2u_{\neq}^3\|_{L^2}\big). \end{align}

By \eqref{f2e4}, \eqref{u3e5}, \eqref{u3e6}, \eqref{u3e1} and Lemma \ref{lem7}, we obtain
\begin{align*}
&\nu\|e^{3\nu^{1/3}t}\nabla f_{2,1}\|_{L^2L^2}^2+\nu\|e^{3\nu^{1/3}t}\nabla (\rho_2(\partial_z-\kappa\partial_y)u_{\neq}^3)\|_{L^2L^2}^2\\&\quad+\nu^{-1}\|e^{3\nu^{1/3}t}\nabla ((\partial_t \kappa-\nu\Delta \kappa)u_{\neq}^3)\|_{L^2L^2}^2\\
&\leq C\nu\big(\|e^{3\nu^{1/3}t}\partial_x(f_2-\rho_1f_1)\|_{L^2L^2}^2+{\ve_{0}^2}\|e^{3\nu^{1/3}t}(\Delta u_{\neq}^3-2\partial_xf_1)\|_{L^2L^2}^2\big)\\  &\quad+C\nu\ve_{0}^2\big(\|e^{3\nu^{1/3}t}\nabla \partial_x(\partial_z-\kappa\partial_y)u_{\neq}^3\|_{L^2L^2}^2+\|e^{3\nu^{1/3}t}\nabla\partial_x^2u_{\neq}^3\|_{L^2L^2}^2\big)\\
&\leq C\nu^{\frac{2}{3}}\big(\|\partial_x(f_2-\rho_1f_1)\|_{X_3}^2+\ve_{0}^2\|\Delta u_{\neq}^3-2\partial_xf_1\|_{X_3}^2\big)\\  &\quad+C\ve_{0}^2\big(\| \partial_x(\partial_z-\kappa\partial_y)u_{\neq}^3\|_{X_3}^2+\|\partial_x^2u_{\neq}^3\|_{X_3}^2\big)\\ &\leq C\ve_{0}^2E_6+C\nu^{\frac{2}{3}}\ve_{0}^2\big(\| u(1)\|_{H^2}^2+\nu^{-\frac{2}{3}}E_6
+\nu^{-1}\|e^{3\nu^{1/3}t}\nabla(g_3)_{\neq}\|_{L^2L^2}^2\big)\\
&\leq C\ve_{0}^2(E_6+\| u(1)\|_{H^2}^2)
+C\nu^{-\frac{1}{3}}\|e^{3\nu^{1/3}t}\nabla(g_3)_{\neq}\|_{L^2L^2}^2.
\end{align*}
Then we deduce from \eqref{u2e2} that
\begin{align}\label{u2e3}
\|\Delta W^{2,1}\|_{X_3}^2\leq& C\Big(\| u(1)\|_{H^2}^2+\nu^{-1}\|e^{3\nu^{1/3}t}\nabla{G_2}\|_{L^2L^2}^2+\ve_{0}^2E_6\\ \nonumber&\quad+\nu^{-\frac{1}{3}}\|e^{3\nu^{1/3}t}\nabla(g_3)_{\neq}\|_{L^2L^2}^2\Big).
\end{align}

{\bf Step 3.}\, Now we estimate $u_{\neq}^j\ (j=2,3).$ Recall that
\beno
\cL u_{\neq}^j+(g_j)_{\neq}=-\partial_j p^{L1}=-\partial_j p^{L(1)}-\partial_j p^{L(2)}.
\eeno
Using the facts that
\beno
\|\partial_x^2 f\|_{L^2}+\| \partial_x(\partial_z-\kappa\partial_y) f\|_{L^2}\leq C(1+\| \kappa\|_{L^{\infty}})\|\partial_x\nabla f\|_{L^2}\leq C\|\partial_x\nabla f\|_{L^2}
\eeno
and $\partial_xf_{\neq}=\partial_xf, $, we deduce from Proposition \ref{prop:decay-L-1}  that
\begin{align*}
&\|\partial_x^2u_{\neq}^j\|_{X_3}^2+\|\partial_x(\partial_z-\kappa\partial_y)u_{\neq}^j\|_{X_3}^2\leq C\Big(\|u_{\neq}^j(1)\|_{H^2}^2+\|e^{3\nu^{1/3}t}\Delta \partial_j p^{L(1)}\|_{L^2L^2}^2\\&\quad+\nu^{-\frac{1}{3}}\|e^{3\nu^{1/3}t} \partial_x\nabla \partial_j p^{L(2)}\|_{L^2L^2}^2+\nu^{-1}\|e^{3\nu^{1/3}t} \partial_xg_j\|_{L^2L^2}^2\Big).
\end{align*}

As $\|\Delta \partial_j p^{L(1)}\|_{L^2}\leq \|\Delta p^{L(1)}\|_{H^1} \leq C\big(1+\|\partial_y\bu^1\|_{H^2}\big)\|\partial_xW^{2,1}\|_{H^1}\leq C\|\nabla\partial_xW^{2,1}\|_{L^2},$ we get
\begin{align*}
&\|e^{3\nu^{1/3}t}\Delta \partial_j p^{L(1)}\|_{L^2L^2}^2\leq C\|e^{3\nu^{1/3}t}\nabla\partial_xW^{2,1} \|_{L^2L^2}^2\leq C\|\Delta W^{2,1} \|_{X_3}^2.
\end{align*}
As $\|\partial_x\nabla \partial_j p^{L(2)}\|_{L^2}\leq \|\partial_x\Delta p^{L(2)}\|_{L^2}=2\nu\|\partial_yV\partial_x^2f_2\|_{L^2}\leq C\nu\|\partial_x^2f_2\|_{L^2}, $ we get
\begin{align*}
&\|e^{3\nu^{1/3}t}\partial_x\nabla \partial_j p^{L(2)}\|_{L^2L^2}^2\leq C\nu^2\|e^{3\nu^{1/3}t}\partial_x^2f_2 \|_{L^2L^2}^2\leq C\nu^{\frac{5}{3}}\|\partial_x^2f_2 \|_{X_3}^2.
\end{align*}
Then we infer that
\begin{align*}
\|\partial_x^2u_{\neq}^j\|_{X_3}^2+\|\partial_x(\partial_z-\kappa\partial_y)u_{\neq}^j\|_{X_3}^2\leq& C\Big(\|u(1)\|_{H^2}^2+\|\Delta W^{2,1} \|_{X_3}^2\\ &+\nu^{\frac{4}{3}}\|\partial_x^2f_2 \|_{X_3}^2+\nu^{-1}\|e^{3\nu^{1/3}t} \partial_xg_j\|_{L^2L^2}^2\Big).
\end{align*}

{\bf Step 4.}  As $\|\partial_x\nabla W^2\|_{X_3}\leq \nu\|\partial_x\nabla f_2\|_{X_3}+\|\partial_x\nabla W^{2,1}\|_{X_3}\leq \nu\|\partial_x\nabla f_2\|_{X_3}+\|\Delta W^{2,1}\|_{X_3}$, we infer that
\begin{align}\label{M6}
E_6\leq& C\Big(\|u(1)\|_{H^2}^2+\|\Delta W^{2,1} \|_{X_3}^2+\nu^2\|\partial_x\nabla f_2\|_{X_3}^2\\ \nonumber&+\nu^{\frac{4}{3}}\|\partial_x^2f_2 \|_{X_3}^2+\nu^{-1}\|e^{3\nu^{1/3}t} (\partial_xg_2,\partial_xg_3)\|_{L^2L^2}^2\Big).
\end{align}
Using \eqref{M6}, \eqref{u2e3} and Lemma \ref{lem:f2}, we obtain\begin{align*}
E_6\leq& C\Big(\|u(1)\|_{H^2}^2+\nu^{-1}\|e^{3\nu^{1/3}t}\nabla G_2\|_{L^2L^2}^2+\ve_{0}^2E_6\\ &+\nu^{-\frac{1}{3}}\|e^{3\nu^{1/3}t}\nabla(g_3)_{\neq}\|_{L^2L^2}^2+\nu^{-1}\|e^{3\nu^{1/3}t} (\partial_xg_2,\partial_xg_3)\|_{L^2L^2}^2\Big).
\end{align*}
Taking $ \ve_{0}$ small enough so that $C\ve_{0}^2\leq 1/2$, we get \begin{align}\label{M6e1}
E_6\leq& C\Big(\|u(1)\|_{H^2}^2+\nu^{-1}\|e^{3\nu^{1/3}t}\nabla{G}_2\|_{L^2L^2}^2\\ \nonumber&+\nu^{-\frac{1}{3}}\|e^{3\nu^{1/3}t}\nabla(g_3)_{\neq}\|_{L^2L^2}^2+\nu^{-1}\|e^{3\nu^{1/3}t} (\partial_xg_2,\partial_xg_3)\|_{L^2L^2}^2\Big).
\end{align}
As $G_2=(g_{2}+\kappa g_3)_{\neq},$ we have
\begin{align}\label{g2e1}
&\|e^{3\nu^{1/3}t} \partial_xg_2\|_{L^2L^2}^2\leq C\big(\|e^{3\nu^{1/3}t}  \nabla {G}_2\|_{L^2L^2}^2+\|e^{3\nu^{1/3}t}\partial_xg_3\|_{L^2L^2}^2\big).
\end{align}
By \eqref{u3e1} and Lemma \ref{lem:f1}, we have
\begin{align*}
\|\Delta u_{\neq}^3\|_{X_3}^2\leq& C\big(\|\Delta u_{\neq}^3-2\partial_xf_1\|_{X_3}^2+\|\partial_xf_1\|_{X_3}^2\big)\\ \leq& C\big(\| u(1)\|_{H^2}^2+\nu^{-\frac{2}{3}}E_6
+\nu^{-1}\|e^{3\nu^{1/3}t}\nabla(g_3)_{\neq}\|_{L^2L^2}^2\big)+C\|\partial_x^2f_1\|_{X_3}^2\\ \leq& C\big(\|u(1)\|_{H^2}^2+\nu^{-\frac{4}{3}}E_6
+\nu^{-1}\|e^{3\nu^{1/3}t}\nabla(g_3)_{\neq}\|_{L^2L^2}^2\big),
\end{align*}
which along with \eqref{M6e1} and  \eqref{g2e1}  gives
\begin{align}\label{eq:E6}
E_6+\nu^{\frac{4}{3}}\|\Delta u_{\neq}^3\|_{X_3}^2\leq& C\big(\| u(1)\|_{H^2}^2+E_6
+\|e^{3\nu^{1/3}t}\nabla(g_3)_{\neq}\|_{L^2L^2}^2\big)\\ \leq& C\Big(\|u(1)\|_{H^2}^2+\nu^{-1}\|e^{3\nu^{1/3}t}  \nabla{G_2}\|_{L^2L^2}^2\nonumber\\&+\nu^{-1}\|e^{3\nu^{1/3}t}\partial_xg_3\|_{L^2L^2}^2+\nu^{-\frac{1}{3}}\|e^{3\nu^{1/3}t} \nabla (g_3)_{\neq}\|_{L^2L^2}^2\Big).\nonumber
\end{align}

{\bf Step 5.}  We denote
\begin{align*}
I_k=\|e^{3\nu^{1/3}t}  \nabla G_{2,k}\|_{L^2L^2}^2+\|e^{3\nu^{1/3}t}  \partial_xg_{3,k}\|_{L^2L^2}^2+\nu^{\frac{2}{3}}\|e^{3\nu^{1/3}t}  \nabla(g_{3,k})_{\neq}\|_{L^2L^2}^2
\end{align*}
for $k=1,\cdots 6$. We have $I_4=0.$ By Lemma \ref{lemg2g3M6}, we have
\begin{align*}
&I_1\leq C\nu^{-1}E_2^2\big(E_6+\nu^{\frac{4}{3}}\|\Delta u_{\neq}^3\|_{X_3}^2\big).
\end{align*}
Notice that
\begin{align*}
&\|\nabla G_{2,k}\|_{L^2}+\|\partial_xg_{3,k}\|_{L^2}+\nu^{\frac{2}{3}}\|  \nabla(g_{3,k})_{\neq}\|_{L^2}\\
&\leq  \|(g_{2,k})_{\neq}\|_{H^1}+C\|\kappa\|_{H^3}\|(g_{3,k})_{\neq}\|_{H^1}+C\|(g_{3,k})_{\neq}\|_{H^1}\\
&\leq C\big(\|(g_{2,k})_{\neq}\|_{H^1}+\|(g_{3,k})_{\neq}\|_{H^1}\big)\leq C\big(\|\nabla(g_{2,k})_{\neq}\|_{L^2}+\|  \nabla(g_{3,k})_{\neq}\|_{L^2}\big),
\end{align*}
which gives
\begin{align*}&I_k\leq C\big(\|e^{3\nu^{1/3}t}  \nabla(g_{2,k})_{\neq}\|_{L^2L^2}^2+\|e^{3\nu^{1/3}t}  \nabla(g_{3,k})_{\neq}\|_{L^2L^2}^2\big).
\end{align*}
Then by Lemma \ref{lemg2g3M6}, we have $ I_2\leq C\nu^{-1}E_2^2E_6,$ and by \eqref{gj5},  we have $ I_5\leq C\nu^{-1}E_2^2E_6,$ and by \eqref{gj6}, we have $I_6\leq C\nu^{-1}E_3^4$. Now we estimate $I_3.$ Notice that
\begin{align*}
\|\kappa(t)\|_{H^1}\leq& C\|\bu^1(t)\|_{H^2}\leq C\|\bu^1(1)\|_{H^2}+C\int_1^t\|\partial_t\bu^1(s)\|_{H^2}ds\\
\leq& CE_1\nu+C\int_1^tE_1\nu ds=CtE_1\nu\leq Ct\nu,
\end{align*}
and then $\|\kappa(t)\|_{H^2}\leq\|\kappa(t)\|_{H^1}^{\frac{1}{2}}\|\kappa(t)\|_{H^3}^{\frac{1}{2}}\leq C(t\nu)^{\frac{1}{2}}$. Thus,
\begin{align*}
\|\nabla{G}_{2,3}\|_{L^2}\leq& \|\nabla{g}_{2,3}\|_{L^2}+\|\kappa (g_{3,3})_{\neq}\|_{H^1}\leq \|\nabla{g}_{2,3}\|_{L^2}+C\|\kappa\|_{H^2}\| (g_{3,3})_{\neq}\|_{H^1}\\ \leq& \|\nabla{g}_{2,3}\|_{L^2}+C(t\nu)^{\frac{1}{2}}\| \nabla(g_{3,3})\|_{L^2},
\end{align*}
from which, Lemma \ref{lemg2g3M3} and the fact that $(t\nu)^{\frac{1}{2}}e^{3\nu^{1/3}t}\leq C\nu^{\frac{1}{3}}e^{4\nu^{1/3}t},$ we infer that
\begin{align*}
I_3\leq&C\big(\|e^{3\nu^{1/3}t}  \nabla{g}_{2,3}\|_{L^2L^2}^2+\|(t\nu)^{\frac{1}{2}}e^{3\nu^{1/3}t} \nabla( g_{3,3})\|_{L^2L^2}^2\big)\\&+\|e^{3\nu^{1/3}t}  \partial_xg_{3,3}\|_{L^2L^2}^2+\nu^{\frac{2}{3}}\|e^{3\nu^{1/3}t}  \nabla(g_{3,3})_{\neq}\|_{L^2L^2}^2\\ \leq & C\big(\|e^{4\nu^{1/3}t} \nabla{g}_{2,3}\|_{L^2L^2}^2+\nu^{\frac{2}{3}}\|e^{4\nu^{1/3}t} \nabla g_{3,3}\|_{L^2L^2}^2+\|e^{4\nu^{1/3}t} \partial_xg_{3,3}\|_{L^2L^2}^2\big)\leq C\nu^{-1}E_3^4.
\end{align*}

Summing up, we conclude that
\begin{align*}
&\|e^{3\nu^{1/3}t}  \nabla G_{2}\|_{L^2L^2}^2+\|e^{3\nu^{1/3}t}  \partial_xg_{3}\|_{L^2L^2}^2+\nu^{\frac{2}{3}}\|e^{3\nu^{1/3}t}  \nabla(g_{3})_{\neq}\|_{L^2L^2}^2\\&\leq C\nu^{-1}E_2^2\big(E_6+\nu^{\frac{4}{3}}\|\Delta u_{\neq}^3\|_{X_3}^2\big)+C\nu^{-1}E_3^4.
\end{align*}
from which and \eqref{eq:E6}, we infer that
\begin{align*}
E_6+\nu^{\frac{4}{3}}\|\Delta u_{\neq}^3\|_{X_3}^2 \leq C\|u(1)\|_{H^2}^2+ C\nu^{-2}E_2^2\big(E_6+\nu^{\frac{4}{3}}\|\Delta u_{\neq}^3\|_{X_3}^2\big)+C\nu^{-2}E_3^4.
\end{align*}
Due to ${E_2}< \ve_{0}\nu$, we get
 \begin{align*}
E_6+\nu^{\frac{4}{3}}\|\Delta u_{\neq}^3\|_{X_3}^2 \leq 2C\|u(1)\|_{H^2}^2+2C\nu^{-2}E_3^4
\end{align*}
 by taking $\ve_{0}$ small enough.

Finally, as $\partial_x^2u^j=\partial_x^2u_{\neq}^j $, we have \begin{align*}
E_5^2\leq 2E_6\leq C\big(\|u(1)\|_{H^2}^2+\nu^{-2}E_3^4\big).
\end{align*}

This completes the proof.\end{proof}

\subsection{Estimate of $E_3$}
Recall that for $j=2,3$,
\begin{align*}
\cL_0 u_{\neq}^j+\partial_j p^{L}+\partial_j p^{(1)}+\bu^1\partial_xu_{\neq}^j+(g_j)_{\neq}=0,\quad \Delta p^{L}=-2\partial_x{u}^2.
\end{align*}
Thanks to $\Delta\big(\cL_0u_{\neq}^2+\partial_j p^{L}\big)=\cL_0 (\Delta u_{\neq}^2),$ we obtain
\begin{align*}
&\cL_0\Delta u_{\neq}^2+\Delta\big(\partial_y p^{(1)}+\bu^1\partial_xu_{\neq}^2+(g_2)_{\neq}\big)=0,\\
&\cL_0u_{\neq}^3+\partial_z p^{L}+\partial_z p^{(1)}+\bu^1\partial_xu_{\neq}^3+(g_3)_{\neq}=0.
\end{align*}
As $\partial_z p^{L}=-2\partial_x\partial_z\Delta^{-1}u_{\neq}^2, $
we infer from Proposition \ref{prop:decay-L0-2} that
\begin{align*}
&\|\Delta u_{\neq}^2\|_{X_2}^2+\|(\partial_x^2+\partial_z^2)u_{\neq}^3\|_{X_2}^2\leq C\Big(\|\Delta u_{\neq}^2(1)\|_{L^2}^2+\|u_{\neq}^3(1)\|_{H^2}^2\\&\quad+\nu^{-1}\|e^{2\nu^{1/3}t} \nabla (\partial_y p^{(1)}+\bu^1\partial_xu_{\neq}^2+(g_2)_{\neq})\|_{L^2L^2}^2\\&\quad+\nu^{-1}\|e^{2\nu^{1/3}t} (\partial_x,\partial_z) (\partial_z p^{(1)}+\bu^1\partial_xu_{\neq}^3+(g_3)_{\neq})\|_{L^2L^2}^2\Big)\\
&\leq C\Big(\|u(1)\|_{H^2}^2+\nu^{-1}\|e^{2\nu^{1/3}t} \nabla (\bu^1\partial_xu_{\neq}^2+(g_2)_{\neq})\|_{L^2L^2}^2+\nu^{-1}\|e^{2\nu^{1/3}t} \Delta p^{(1)}\|_{L^2L^2}^2\\&\quad+\nu^{-1}\|e^{2\nu^{1/3}t} (\partial_x,\partial_z) (\bu^1\partial_xu_{\neq}^3+(g_3)_{\neq})\|_{L^2L^2}^2\Big).
\end{align*}
As $g_j=\sum\limits_{k=1}^6g_{j,k},\ (g_{j,4})_{\neq}=0,$ by Lemma \ref{lemg2g3M6}, Lemma \ref{lemg2g3M3} and  Lemma \ref{lemM3M5}, we obtain
\begin{align*}
&\|\Delta u_{\neq}^2\|_{X_2}^2+\|(\partial_x^2+\partial_z^2)u_{\neq}^3\|_{X_2}^2\\
&\leq  C\big(\|u(1)\|_{H^2}^2+\nu^{-2}E_2^2E_3^2+E_1^2E_3E_5+\nu^{-2}E_2^2E_6+\nu^{-2}E_3^4\big).
\end{align*}
Due to $E_2< \ve_{0}\nu {,\ E_3< \ve_{0}\nu}$ and $E_1< \ve_{0},$ by Proposition \ref{prop2}, we have
\begin{align*}
E_3^2\leq& C\big(\|\Delta u_{\neq}^2\|_{X_2}^2+\|(\partial_x^2+\partial_z^2)u_{\neq}^3\|_{X_2}^2+\nu^{\frac{4}{3}}\|\Delta u_{\neq}^3\|_{X_3}^2\big)\\ \leq&  C\big(\|u(1)\|_{H^2}^2+\nu^{-2}E_2^2E_3^2+E_1^2E_3E_5+{E_6+}\nu^{\frac{4}{3}}\|\Delta u_{\neq}^3\|_{X_3}^2+\nu^{-2}E_3^4\big)\\ \leq&  C(\|u(1)\|_{H^2}^2+\nu^{-2}E_2^2E_3^2+E_1^4E_3^2+E_5^2+\nu^{-2}E_3^4)\\ \leq&  C(\|u(1)\|_{H^2}^2+\nu^{-2}E_2^2E_3^2+E_1^4E_3^2+\nu^{-2}E_3^4) {\leq C\big(\|u(1)\|_{H^2}^2+\ve_0^2E_3^2\big)}.
\end{align*}
This shows that
\ben
E_3\le C\|u(1)\|_{H^2}.
\een

\subsection{Estimate of $E_4$}

As $\|(\partial_x,\partial_z)f_{\neq}\|_{H^3}\leq C\|(\partial_x^4,\partial_z^4,\partial_x\partial_y^3,\partial_z\partial_y^3)f_{\neq}\|_{L^2}, $ we have
\beno
\|e^{a\nu^{1/3}t}(\partial_x,\partial_z)f_{\neq}\|_{Y_0^3}^2\leq C\|(\partial_x^4,\partial_z^4,\partial_x\partial_y^3,\partial_z\partial_y^3)f_{\neq}\|_{X_a}^2.
\eeno
Notice that
\begin{align*}
&\cL_0\partial_x^4f_{\neq}=\partial_x^4\cL_0f_{\neq},\quad \cL_0\partial_z^4f_{\neq}=\partial_z^4\cL_0f_{\neq},\quad\cL_0\partial_yf_{\neq}=\partial_y\cL_0f_{\neq}-\partial_xf_{\neq},\\ &\cL_0\partial_y^2f_{\neq}=\partial_y^2\cL_0f_{\neq}-2\partial_x\partial_yf_{\neq},\quad\cL_0\partial_y^3f_{\neq}=\partial_y^3\cL_0f_{\neq}-3\partial_x\partial_y^2f_{\neq},\\ &\cL_0\partial_x\partial_y^3f_{\neq}=\partial_x\partial_y^3\cL_0f_{\neq}-3\partial_x^2\partial_y^2f_{\neq},\quad \cL_0\partial_z\partial_y^3f_{\neq}=\partial_z\partial_y^3\cL_0f_{\neq}-3\partial_x\partial_z\partial_y^2f_{\neq}.\end{align*}
It follows from Proposition \ref{prop:decay-L0} that
for $a\in[0,4],$
\begin{align*}
&\|(\partial_x^4,\partial_z^4,\partial_x\partial_y^3,\partial_z\partial_y^3)f_{\neq}\|_{X_a}^2\leq C\Big(\|f(1)\|_{H^4}^2+\nu^{-1}\|e^{a\nu^{1/3}t}\nabla\Delta\cL_0f_{\neq}\|_{L^2L^2}^2
\\&\quad+\nu^{-1}\|e^{a\nu^{1/3}t}\partial_x\partial_y(\partial_x,\partial_z)f_{\neq}\|_{L^2L^2}^2\Big)\\&\leq C\big(\|f(1)\|_{H^4}^2+\nu^{-2}\|e^{a\nu^{1/3}t}\Delta\cL_0f_{\neq}\|_{Y_0}^2
+\nu^{-1}\|e^{a\nu^{1/3}t}\nabla\partial_x(\partial_x,\partial_z)f_{\neq}\|_{L^2L^2}^2\big).
\end{align*}
Recall that
\begin{align*}
&\cL_0 u+\mathbb{P}\left(\begin{array}{l}u^2\\0\\0\end{array}\right)+\mathbb{P}(u\cdot\nabla u)=0.
\end{align*}
Then we infer that
\begin{align*} &E_4\leq C\|(\partial_x^4,\partial_z^4,\partial_x\partial_y^3,\partial_z\partial_y^3)u_{\neq}\|_{X_2}\leq C\Big(\|u(1)\|_{H^4}+\nu^{-1}\|e^{2\nu^{1/3}t}\mathbb{P}( \Delta u_{\neq}^2,0,0)^T\|_{Y_0}\\&\qquad+\nu^{-1}\|e^{2\nu^{1/3}t}\mathbb{P}\Delta (u\cdot\nabla u)_{\neq}\|_{Y_0}
+\nu^{-\frac{1}{2}}\|e^{2\nu^{1/3}t}\nabla\partial_x(\partial_x,\partial_z)u_{\neq}\|_{L^2L^2}\Big).
\end{align*}
As $\mathbb{P}\nabla=\nabla\mathbb{P},\ \mathbb{P}\Delta=\Delta\mathbb{P},\  \|\mathbb{P}f\|_{L^2}\leq\|f\|_{L^2}, $ we have
\begin{align*}
&\|e^{2\nu^{1/3}t}\mathbb{P}( \Delta u_{\neq}^2,0,0)^T\|_{Y_0}\leq \|e^{2\nu^{1/3}t}\Delta u_{\neq}^2\|_{Y_0} \leq C\|\Delta u_{\neq}^2\|_{X_2}\leq CE_3,\\
&\|e^{2\nu^{1/3}t}\mathbb{P}\Delta ( u\cdot\nabla u)_{\neq}\|_{Y_0}\leq \|e^{2\nu^{1/3}t}\Delta ( u\cdot\nabla u)_{\neq}\|_{Y_0}\leq \|e^{2\nu^{1/3}t}\Delta (u\cdot\nabla u)_{\neq}\|_{Y_0^2}, \end{align*}
and by Lemma \ref{lemuM3},
\beno
\|e^{2\nu^{1/3}t}\nabla\partial_x(\partial_x,\partial_z)u_{\neq}\|_{L^2L^2}\leq C\nu^{-\frac{1}{2}}E_3.
\eeno
Thus, we get
\begin{align}\label{M4e1}
&E_4\leq C\big(\|u(1)\|_{H^4}+\nu^{-1}E_3+\nu^{-1}\|e^{2\nu^{1/3}t} (u\cdot\nabla u)_{\neq}\|_{Y_0^2}\big).
\end{align}
As $(u\cdot\nabla u)_{\neq}=(u_{\neq}\cdot\nabla u_{\neq})_{\neq}+u_{\neq}\cdot\nabla \bu+\bu\cdot\nabla u_{\neq}, $ by Lemma \ref{lemM4M3M1}  {and Lemma \ref{lem:HH-Y0}}, we have
\begin{align*}
&\| e^{2\nu^{1/3}t}(u\cdot\nabla u)_{\neq}\|_{Y_0^2}\leq \| e^{2\nu^{1/3}t}u_{\neq}\cdot\nabla u_{\neq}\|_{Y_0^2}+\|e^{2\nu^{1/3}t}  (u_{\neq}\cdot\nabla \bu)\|_{Y_0^2}\\&\quad+\|e^{2\nu^{1/3}t}  (\bu\cdot\nabla u_{\neq})\|_{Y_0^2}\leq C\big(E_4(E_3+E_1\nu+E_2)+E_1E_3+E_1E_5+E_3E_2/\nu\big).
\end{align*}
Then we conclude that
\begin{align*}
E_4\leq& C\big(\|u(1)\|_{H^4}+\nu^{-1}E_3+E_4((E_3+E_2)/\nu+E_1)\\&+E_1(E_3+E_5)/\nu+E_3E_2/\nu^2\big).
\end{align*}
By taking $\ve_0$ small enough so that $C((E_3+E_2)/\nu+E_1)<1/2$, we deduce that
\ben
E_4\leq C\big(\|u(1)\|_{H^4}+\nu^{-1}E_3+\nu^{-1}E_5\big).
\een

\section*{Acknowledgement}
Z. Zhang is partially supported by NSF of China under Grant 11425103.


\begin{thebibliography}{99}


\bibitem{BT} J. S. Baggett and L. N. Trefethen, {\it Low-dimensional models of subcritical transition to turbulence}, Phys. Fluids,  9 (1997), 1043-1053.

\bibitem{BW} M. Beck and C. E. Wayne, {\it  Metastability and rapid convergence to quasi-stationary bar states for the two-dimensional Navier-Stokes equations}, Proc. Roy. Soc. Edinburgh Sect. A, 143 (2013), 905-927.

\bibitem{BCV} J. Bedrossian, M. Coti Zelati, V. Vicol, {\it Vortex axisymmetrization, inviscid damping, and vorticity depletion in the linearized 2D Euler equations},
arXiv 1711.03668.

\bibitem{BGM1} J. Bedrossian, P. Germain and N. Masmoudi, {\it Dynamics near the subcritical transition of the 3D Couette flow I: Below threshold case}, arXiv 1506.03720.

\bibitem{BGM2} J. Bedrossian, P. Germain and N. Masmoudi, {\it  Dynamics near the subcritical transition of the 3D Couette flow II: Above threshold case}, arXiv 1506.03721.

\bibitem{BGM3} J. Bedrossian, P. Germain and N. Masmoudi, {\it On the stability threshold for the 3D Couette flow in Sobolev regularity}, Annals of Math., 185(2017), 541-608.

\bibitem{BGM} J. Bedrossian, P. Germain and N. Masmoudi, {\it Stability of the Couette flow at high Reynolds number in 2D and 3D}, arXiv:1712.02855.

\bibitem{BM} J. Bedrossian and N. Masmoudi, {\it Inviscid damping and the asymptotic stability of planar shear flows in the 2D Euler equations},
Publ. Math. Inst. Hautes \'{E}tudes Sci., 122(2015), 195-300.

\bibitem{BMV} J. Bedrossian, N. Masmoudi and V. Vicol,  {\it Enhanced dissipation and inviscid damping in the inviscid limit of the Navier-Stokes equations near the two dimensional Couette flow},  Arch. Ration. Mech. Anal., 219(2016), 1087-1159.

\bibitem{BWV} J. Bedrossian, F. Wang and V. Vicol, {\it The Sobolev stability threshold for 2D shear flows near Couette}, arXiv:1604.01831.

\bibitem{Cha} S. J. Chapman, {\it Subcritical transition in channel flows},
            J.  Fluid Mech., 451(2002), 35-97.

\bibitem{CKR} P. Constantin, A. Kiselev, L. Ryzhik and  A. Zlatos, {\it Diffusion and mixing in fluid flow},  Ann. of Math. (2), 168(2008), 643-674.

\bibitem{DR} P. Drazin and W. Reid, Hydrodynamic Stability, Cambridge Monographs Mech. Appl. Math., Cambridge Univ. Press, New York, 1981.

\bibitem{DBL} Y. Duguet, L. Brandt and B.  Larsson, {\it Towards minimal perturbations in transitional plane Couette flow}, Phys. Rev. E, 82 (2010), 026316, 13.

\bibitem{EP} T. Ellingsen and E. Palm, {\it Stability of linear flow}, Phys. Fluids,  18(1975), 487-488.

\bibitem{Ga} T. Gallay, {\it Enhanced dissipation and axisymmetrization of two-dimensional viscous vortices}, arXiv:1707.05525.

\bibitem{GG} T. Gebhardt and S. Grossmann, {\it  Chaos transition despite linear stability}, Phys. Rev. E,  50(1994), 3705-3711.

\bibitem{IMM} S. Ibrahim, Y. Maekawa and  N. Masmoudi, {\it On pseudospectral bound for non-selfadjoint operators and its application to stability of Kolmogorov flows}, arXiv:1710.05132.

\bibitem{Kel} L. Kelvin, {\it Stability of fluid motion-rectilinear motion of viscous fluid between two parallel plates}, Phil. Mag., 24(1887), 188-196.

\bibitem{Kla} S. Klainerman, {\it Long-time behavior of solutions to nonlinear evolution equations},
Arch. Rational Mech. Anal., 78(1982), 73-98.

\bibitem{LWZ} T. Li, D. Wei and Z. Zhang, {\it  Pseudospectral bound and transition threshold for the 3D Kolmogorov flow}, arXiv:1801.05645.

\bibitem{LK} M. Liefvendahl and G. Kreiss, {\it Bounds for the threshold amplitude for plane Couette flow}, J. Nonlinear Math. Phys., 9 (2002), 311-324.

\bibitem{LX} Z. Lin and M. Xu, {\it Metastability of Kolmogorov flows and inviscid damping of shear flows}, arXiv:1707.00278.

\bibitem{LHR} A. Lundbladh, D. Henningson and S. Reddy,  {\it Threshold amplitudes for transition in channel flows}, in Transition, Turbulence and Combustion, Springer-Verlag, New York, 1994, pp. 309-318.

\bibitem{Orr} W. Orr, {\it The stability or instability of steady motions of a perfect liquid and of a viscous liquid. Part I: A perfect liquid, Proc. Royal Irish Acad. Sec. A: Math. Phys. Sci.},  27(1907), 9-68.

\bibitem{OK} S. Orszag and L. Kells, {\it Transition to turbulence in plane Poiseuille and plane Couette flow}, J. of Fluid Mech., 96(1980), 159-205.

\bibitem{RSB} S. Reddy, P. Schmid, J. Baggett and D. Henningson, {\it On stability of streamwise streaks and transition thresholds in plane channel flows}, J. Fluid Mech., 365 (1998), 269-303.

\bibitem{Rey} O. Reynolds, {\it  An experimental investigation of the circumstances which determine whether the motion of water shall be direct or sinuous, and of the law of resistance in parallel channels}, Proc. R. Soc. Lond. , 35(1883), 84.

\bibitem{Rom} V.  A. Romanov, {\it Stability of plane-parallel Couette flow}, Funkcional. Anal. i Prilo\v{z}en,  7 (1973),  62-73.

\bibitem{Sch} P. Schmid and D. Henningson, Stability and Transition in Shear Flows,  Applied Mathematical Sciences 142, Springer-Verlag, New York, 2001.

\bibitem{TTR} L. Trefethen, A. Trefethen, S. Reddy and T. Driscoll, {\it Hydrodynamic stability without eigenvalues}, Science, 261(1993), 578-584.

\bibitem{Tre} L. N. Trefethen, {\it Pseudospectra of linear operators}, SIAM Review, 39(1997), 383-406.

\bibitem{Wal} F. Waleffe, {\it Transition in shear flows. Nonlinear normality versus non-normal linearity}, Phys. Fluids, 7(1995), 3060-3066.

\bibitem{WZZ1} D. Wei, Z. Zhang and W. Zhao, {\it Linear inviscid damping for a class of monotone shear flow in Sobolev spaces},
Comm. Pure Appl. Math., 71(2018), 617-687.

\bibitem{WZZ2} D. Wei, Z. Zhang and W. Zhao, {\it Linear inviscid damping and vorticity depletion for shear flows}, arXiv:1704.00428.

\bibitem{WZZ3} D. Wei, Z. Zhang and W. Zhao, {\it Linear inviscid damping and enhanced dissipation for the Kolmogorov flow}, arXiv:1711.01822.

\bibitem{Yag} A. Yaglom, Hydrodynamic instability and transition to turbulence, Fluid Mech. Appl. 100, Springer-Verlag, New York, 2012.

\bibitem{Z1}  C. Zillinger, {\it Linear inviscid damping for monotone shear flows}, Trans. Amer. Math. Soc., 369 (2017), 8799-8855.

\bibitem{Z2} C. Zillinger, {\it Linear inviscid damping for monotone shear flows in a finite periodic channel, boundary effects, blow-up and critical Sobolev regularity}, Arch. Ration. Mech. Anal., 221(2016), 1449-1509.


\end{thebibliography}
\end{document}